\documentclass{amsart}
\usepackage{amssymb,amsmath}
\usepackage{pdflscape}
\usepackage{mathrsfs}
\usepackage{tikz}
\usepackage{graphicx}
\usepackage{caption}
\usepackage{subcaption}
\usepackage{rotating}
\usepackage{afterpage}
\usepackage{hyperref}

\usetikzlibrary{calc,intersections, arrows, shapes}
\newcommand{\red}{\textcolor{red}}
\newcommand{\blue}{\textcolor{blue}}

\newcommand{\RR}{\mathbb{R}}

\newcommand{\NN}{\mathbb{N}}

\newcommand{\Sym}{\mathcal{S}}
\newcommand{\PSD}{\mathcal{S}_+}
\newcommand{\rank}{\textup{rank}\,}

\newcommand{\rankpsd}{\textup{rank}_{\textup{psd}}}

\newcommand{\sqrtrank}{\textup{rank}_{\! \! {\sqrt{\ }}}\,}

\newtheorem{theorem}{Theorem}[section]

\newtheorem{lemma}[theorem]{Lemma}
\newtheorem{corollary}[theorem]{Corollary}
\newtheorem{proposition}[theorem]{Proposition}
\theoremstyle{definition}
\newtheorem{definition}[theorem]{Definition}
\newtheorem{example}[theorem]{Example}

\theoremstyle{remark}
\newtheorem{remark}[theorem]{Remark}
\newtheorem{problem}[theorem]{Problem}

\title[Psd-minimal $4$-Polytopes]{Four-Dimensional Polytopes of\\Minimum Positive Semidefinite Rank}

\author[Gouveia]{Jo{\~a}o Gouveia}
\address{CMUC, Department of Mathematics,
  University of Coimbra, 3001-454 Coimbra, Portugal}
\email{jgouveia@mat.uc.pt} 

\author[Pashkovich]{Kanstanstin Pashkovich}
\address{Department of Combinatorics and Optimization, 
University of Waterloo, 200 University Ave. W
Waterloo, Ontario, Canada, N2L 3G1}
\email{kanstantsin.pashkovich@gmail.com} 

\author[Robinson]{Richard Z. Robinson}
\address{Department of Mathematics, University of Washington, Box
  354350, Seattle, WA 98195, USA} \email{rzr@uw.edu}

\author[Thomas]{Rekha R. Thomas}
\address{Department of Mathematics, University of Washington, Box
  354350, Seattle, WA 98195, USA} \email{rrthomas@uw.edu}

\thanks{Gouveia was partially supported by the Centre for Mathematics of the
University of Coimbra -- UID/MAT/00324/2013, funded by the Portuguese
Government through FCT/MEC and co-funded by the European Regional Development Fund through the Partnership Agreement PT2020,  Pashkovich by F.R.S.-FNRS research project T.0100.13, Semaphore 14620017, Robinson by 
the U.S. National Science Foundation Graduate Research Fellowship under Grant No. DGE-0718124, and Thomas by 
the U.S. National Science Foundation grant DMS-1418728.}  

\date{\today}

\begin{document}

\begin{abstract} The positive semidefinite (psd) rank of a polytope is the size of the smallest psd cone 
that admits an affine slice that projects linearly onto the polytope. The psd rank of a $d$-polytope is at least $d+1$, and when equality holds we say that the polytope is psd-minimal. In this paper we develop new tools for the study of psd-minimality and use them to give a complete classification of psd-minimal $4$-polytopes.
The main tools introduced are trinomial obstructions, a new algebraic obstruction for psd-minimality, and the slack ideal of a polytope, which encodes the space of realizations of a polytope up to projective equivalence. 

Our central result is that there are $31$ combinatorial classes of psd-minimal $4$-polytopes. 
We provide combinatorial information and an explicit psd-minimal realization in each class. 
For $11$ of these classes, every polytope in them is psd-minimal, and these are precisely the combinatorial classes of the known projectively unique $4$-polytopes. We give a complete characterization of psd-minimality in the 
remaining classes, encountering in the process counterexamples to some open conjectures.
\end{abstract}


\maketitle

\section{Introduction} \label{sec:introduction}

The {\em positive semidefinite (psd) rank} of a convex set 
was introduced in \cite{gouveia2011lifts}, and it can be seen as a 
measure of geometric complexity of that set.
Let $\Sym^k$ denote the vector space of all real symmetric $k \times k$ matrices with the inner product 
$\langle A, B \rangle := \textup{trace}(AB)$, and let $\PSD^k$ be the cone of positive semidefinite matrices in $\Sym^k$. A polytope $P \subset \RR^d$ is said to have a {\em psd lift} of size $k$ if there is an 
affine space $L \subset \Sym^k$ and a linear map $\pi \,:\, \Sym^k \rightarrow \RR^d$ such that $P = \pi(\PSD^k \cap L)$.
The psd rank of $P$, $\rankpsd(P)$, is the smallest $k$ such that $P$ has a psd lift of size $k$.
Linear optimization over a polytope can be achieved via semidefinite programming 
over its psd lift. Thus a lift of small size (alternatively, small psd rank of the polytope) implies, in principle, the possibility of efficiently solving a linear optimization problem over this polytope. These features have attracted much research on the psd rank of polytopes in 
recent years, with several exciting new results coming from optimization and computer science
\cite{BrietDadushPokutta}, \cite{LeeRaghavendraSteurer}, \cite{LeeWeiWolf}, \cite{LeeWei}, \cite{FawziParrilo}, \cite{FawziSaundersonParrilo}.
For a survey on psd rank of nonnegative matrices, see \cite{FGPRT}.

Since polytopes of small psd rank can admit efficient algorithms 
for linear optimization, there is much incentive to understand them. If $P$ is a $d$-polytope, then $\rankpsd(P) \geq d+1$, and if equality holds, we say that $P$ is {\em psd-minimal}. 
These polytopes are a natural place to start the study of small psd rank, a task that was initiated in \cite{gouveia2013polytopes}. 
A well-known example of a psd-minimal 
polytope is the {\em stable set polytope} of a {\em perfect graph} \cite{LovaszThetaBody}. In this case, psd-minimality implies 
that the size of a largest stable set in a perfect graph can be 
found in polynomial time, while computing the size of a largest stable set in a graph is NP-hard in general. The existence of a small psd lift for the stable set polytopes of perfect graphs is the only known proof of the polynomial time solvability of the stable set problem in this class of graphs. 

The stable set polytope of a perfect graph is also an example of a 
{\em 2-level polytope} \cite{gouveia2010thetabodies}. These are polytopes with the property that for each facet of the polytope there is a unique parallel translate of its affine hull containing all vertices of the polytope that are not on the facet.
All 2-level polytopes are psd-minimal and affinely equivalent to $0/1$-polytopes with additional special properties, yet they are far from well-understood and offer many challenges. Several groups of researchers have been studying them recently \cite{GrandeSanyal},
\cite{Sam&Co2level}, \cite{Sam&Co2levelEnumeration}. 
The study of psd-minimal polytopes is an even more ambitious task than that of 2-level polytopes, yet it is an important step in understanding the phenomenon of small psd rank, and offer a rich set of examples 
for honing psd rank techniques. 

The only psd-minimal $2$-polytopes are triangles and quadrilaterals. This already helps in the classification of higher dimensional psd-minimal polytopes, due to the following useful lemma.
\begin{lemma}\cite[Proposition~3.8]{gouveia2013polytopes} \label{lem:face_inheritance}
Any face of a psd-minimal polytope is psd-minimal.
\end{lemma}
Therefore psd-minimal $3$-polytopes can only have triangular or quadrilateral facets but still turn out to be interesting. Call an octahedron in $\RR^3$, {\em biplanar}, if there are two distinct planes each containing four vertices of the octahedron. A complete classification of psd-minimal $3$-polytopes  is known.

\begin{theorem} \cite[Theorem~4.11]{gouveia2013polytopes} \label{thm:classification in 3-space}
The psd-minimal $3$-polytopes are precisely simplices,  quadrilateral pyramids, bisimplices, biplanar octahedra and their polars.
\end{theorem}

Recall that if $P$ is a $d$-polytope containing the origin in its interior, then its polar is $P^\circ := \{ y \in \RR^d \,:\, \langle x, y \rangle \leq 1 \,\,\forall \,\, x \in P \}$. A polytope $P$ has a psd lift of size $k$ if and only if $P^\circ$ has a psd lift of size $k$, and in particular, $\rankpsd(P) = \rankpsd(P^\circ)$. The polar of a combinatorial octahedron is a combinatorial cube. The above theorem shows that psd rank is not a combinatorial property; not all polytopes that are combinatorially equivalent to octahedra and cubes are psd-minimal, the embedding matters. Since psd-minimality is closed under polarity and the combinatorial type of the polar depends only on the combinatorial type of the original polytope, we frequently refer to \emph{dual pairs} of combinatorial types, namely, a pair of combinatorial classes such that the polytopes in one are the polars of those in the other. 

If $P \subset \RR^d$ is a polytope with vertices $v_1, \ldots, v_n$, and facet inequalities 
$a_j^\top x \leq \beta_j$ for $j=1,\ldots,m$, then the {\em slack matrix} of $P$ is the nonnegative matrix $S_P \in \RR^{n \times m}$ such that $(S_P)_{ij}: = \beta_j - a_j^\top v_i$, the slack of the vertex $v_i$ in the facet inequality $a_j^\top x \leq \beta_j$. If $P$ is a $d$-polytope with $d \geq 1$, then $\rank(S_P) = d+1$, and note that the zero pattern in $S_P$ records the vertex-facet incidence structure of $P$, or equivalently, the combinatorics of $P$. 

An $\PSD^k${\em-factorization} of $S_P$ is an assignment of psd matrices $A_1, \ldots, A_n\in \PSD^k$ and $B_1, \ldots, B_m \in \PSD^k$ to the rows and columns, respectively, of $S_P$ such that $(S_P)_{ij} = \beta_j - a_j^\top v_i = \langle A_i, B_j \rangle$. The psd rank of $S_P$, denoted as $\rankpsd(S_P)$, is the minimum $k$ for which $S_P$ admits a factorization through $\PSD^k$. The connection to psd lifts of the polytope $P$ comes via the result from \cite{gouveia2011lifts} that 
$P$ has a psd lift of size $k$ if and only if $S_P$ has an $\PSD^k$-factorization, and hence 
$\rankpsd(P) = \rankpsd(S_P)$.  

There is another notion of rank that plays a key role in the study of psd-minimality.  A Hadamard square root of $S_P$ is a matrix obtained by replacing every positive entry in $S_P$ with one of its two square roots. 
The positive Hadamard square root $\sqrt[+]{S_P}$ is the Hadamard square root obtained by replacing every positive entry by its positive square root. 
The {\em square root rank} of $S_P$, denoted as $\sqrtrank(S_P)$, is the minimum rank 
of a Hadamard square root of $S_P$. 

\begin{theorem} \cite{gouveia2013polytopes} \label{thm:square root rank}
A $d$-polytope $P \subset \RR^d$ with $d \geq 1$ is psd-minimal if and only if 
$$\sqrtrank(S_P) = d+1\,.$$
\end{theorem}

\subsection{Contribution}

In this paper, we classify psd-minimal $4$-polytopes. The techniques used in the classification of psd-minimal $3$-polytopes do not have a natural generalization to $4$-polytopes, so new tools and approaches had to be developed. These are of interest beyond the scope of this classification, and offer insight on psd-minimality in general. As a by-product, we obtain simpler proofs of the 
classification results for $2$- and $3$-dimensional polytopes.

In terms of techniques, this paper has two main contributions. The first is the notion of \emph{trinomial obstruction}, that provides a simple certificate that a polytope is not psd-minimal.
This obstruction is powerful enough to completely classify the combinatorial types of psd-minimal $4$-polytopes. The second main technical contribution is the notion of the {\em slack ideal} of a polytope. 
We show that the positive real zeros of this ideal are essentially in bijection with the different projective 
equivalence classes in the combinatorial class of the polytope. Using slack ideals we develop a general computational algebra procedure for characterizing the psd-minimal polytopes in a combinatorial class.
This is used to complete the classification.

At a high level, our results are as follows. 
We prove that there are $31$ combinatorial classes of psd-minimal $4$-polytopes, and they are described in Table~\ref{table:list}. 
Combinatorial information about each class such as its $f$-vector and types of facets is listed. We also exhibit an 
explicit psd-minimal polytope in each of the $31$ combinatorial classes. 
We then proceed to characterize the psd-minimal polytopes in each of these classes.
In $11$ of them, every polytope is psd-minimal. Coincidentally these classes are also precisely the classes of the $11$ known projectively unique $4$-polytopes. For the remaining $20$ classes, we derive precise conditions for psd-minimality. In most cases, the conditions can be seen as affine constraints on the entries of the slack matrices. However we get two interesting new behaviors.  
For two pairs of primal-dual classes, including the $4$-cube and its dual, there are two essentially different 
types of psd-minimal realizations. For two other pairs of primal-dual classes, psd-minimality is characterized by non-linear algebraic conditions, and 
they settle negatively some open conjectures (\cite[Problem 2]{beasley_et_al:DR:2013:4019} and generalizations). In particular they give us the first examples of psd-minimal polytopes whose minimality cannot be certified by the positive Hadamard 
square root of the slack matrix.

\subsection{Organization}

The results in this paper naturally split into two parts, which guides the organization of the sections.
In the first part, we consider combinatorial properties of, and obstructions to, psd-minimality.  In Section~\ref{sec:trinomial obstructions} we develop a lower bounding method for the square root rank of a matrix and illustrate it by deriving a short proof of the psd-minimality results for $2$- and $3$-polytopes 
up to combinatorial equivalence. In Section~\ref{sec:facet intersection obstructions} we specify the obstructions to psd-minimality on the slack matrices of $4$-polytopes and derive four combinatorial results that constrain the types of facet intersections that are possible for psd-minimal $4$-polytopes. These results allow us in Section~\ref{sec:31 combinatorial classes} to precisely identify the 
$31$ combinatorial classes of psd-minimal $4$-polytopes concluding the first part of the paper. An explicit psd-minimal polytope in every combinatorial class, as well as various combinatorial properties of each class can be found in Table~\ref{table:list}. We refer to each class by its 
number in Table~\ref{table:list}. 

In the second half of the paper, we identify the conditions for psd-minimality in each of the 31 combinatorial classes. This requires new algebraic and geometric tools. In Section~\ref{sec:polys1-11} we introduce the slack ideal of a polytope, and prove that if this ideal is binomial, then all polytopes that are combinatorially equivalent to this 
polytope are psd-minimal. Next we prove that if a $d$-polytope has $d+2$ vertices, its slack ideal is binomial, and finally we show that the slack ideals of classes 1-11 are binomial. In Section~\ref{sec:polys12-15} we consider the combinatorial classes 12-15 from Table~\ref{table:list} which come in two dual pairs. 
The slack ideal is used to derive a parametrization of the slack matrices of the psd-minimal polytopes in each of these classes. As mentioned above, these examples are particularly interesting since they provide counterexamples to several conjectures about psd-minimality that one might entertain based on the results in \cite{gouveia2013polytopes}. We discuss these features in detail. We conclude in Section~\ref{sec:polys16-31} with the precise conditions under which the remaining classes in Table~\ref{table:list} are psd-minimal. As an illustration of our new methods, 
we use them to reprove that biplanarity is a necessary and sufficient condition for the psd-minimality of octahedra,
finishing a new proof of Theorem \ref{thm:classification in 3-space}.
The codes for the main calculations in this paper can be found at:  
{\small \url{http://kanstantsinpashkovich.bitbucket.org/computations/psd_minimal_four_polytopes.html}}
in the form of Sage \cite{sage} worksheets that rely on Macaulay2 \cite{M2}.

\subsection{Slack matrices of projectively equivalent polytopes}

We conclude the introduction with a simple result that relates the slack matrices of projectively equivalent 
polytopes which is used extensively in the later parts of this paper.

Recall that two polytopes $P$ and $Q$ in $\RR^d$ are projectively equivalent if and only if there exists a projective transformation sending $P$ to $Q$, i.e. $Q = \phi(P)$ where
\begin{align} 
\phi \,:\, \RR^d \rightarrow \RR^d, \,\,\,\, \phi(x) := \frac{Bx + b}{c^\top x + \gamma}
\end{align}
for some $B \in \RR^{d \times d}, b \in \RR^d, c \in \RR^d$ and  $\gamma \in \RR$ 
such that
\begin{align} \label{eq:conditions}
\det \begin{bmatrix} B & b \\ c^\top & \gamma \end{bmatrix} \neq 0.
\end{align}
A convenient way to think of projective equivalence is in terms of homogenizations. Recall that given a polytope $P \subseteq \RR^d$ with vertices $v_1, \ldots, v_n$, its {\it homogenization} is the convex cone
$\textup{homog}(P) \subseteq \RR^{d+1}$ spanned by $(v_1,1) , \ldots, (v_n,1) \in \RR^{d+1}$.

\begin{lemma} \label{lem:projective equivalence = linear equivalence of homogenizations}
Two polytopes $P, Q \subset \RR^d$ are projectively equivalent if and only if 
the cones $\textup{homog}(P)$ and $\textup{homog}(Q)$ in $\RR^{d+1}$ are linearly isomorphic.
\end{lemma}
\begin{proof}
The cones $\textup{homog}(P)$ and $\textup{homog}(Q)$ are linearly isomorphic if and only if there exists 
a linear map from $\RR^{d+1} \rightarrow \RR^{d+1}$ with invertible representing matrix 
$$\begin{bmatrix} B & b \\ c^\top & \gamma \end{bmatrix}$$ that sends $\textup{homog}(P)$ to $\textup{homog}(Q)$.
By equating $Q$ to the dehomogenization of the image of $\textup{homog}(P)$, one sees that this happens if and only if for $\varphi(x):=\frac{Bx+b}{c^\top x + \gamma}$ we have $Q=\varphi(P)$, i.e., 
if and only if $P$ and $Q$ are projectively equivalent.
\end{proof}

\begin{corollary}\label{cor:proj_equivalent slack matrix}
Two polytopes are projectively equivalent if and only if they have the same slack matrix up to permutations and positive scalings of rows and columns.
\end{corollary}
\begin{proof}
Slack matrices of cones are defined analogously to slack matrices of polytopes: each entry 
of the slack matrix is indexed by an extreme ray and a facet of the cone, and contains the corresponding slack value.
By definition, scaling the rows and columns of a slack matrix of a cone by positive real numbers 
produces another slack matrix of the same cone. 

Moreover, a slack matrix of a polytope $P$  is a slack matrix of the cone $\textup{homog}(P)$ and  
therefore, by Lemma~\ref{lem:projective equivalence = linear equivalence of homogenizations}, it suffices to prove that two pointed cones are linearly isomorphic if and only if they have a common slack matrix. 
It is easy to see that two linearly isomorphic cones have the same slack matrices.
On the other hand, the reasoning in \cite[Theorem~14]{slackmatrixpaper} 
shows that every pointed cone is linearly isomorphic to the cone spanned by the rows of its slack matrix. Thus if two pointed cones have the same slack matrix, they are linearly isomorphic to the same cone, and hence to each other.
 \end{proof}

\noindent{\bf Acknowledgments.} We thank Arnau Padrol and Serkan Ho{\c s}ten for useful inputs to this paper. 
We also thank the referees for their valuable comments.

\section{Trinomial obstructions to psd-minimality} \label{sec:trinomial obstructions}
Recall that two polytopes $P$ and $Q$ are 
{\em combinatorially equivalent} if they have the same vertex-facet incidence structure. 
In this section we describe a simple algebraic obstruction to psd-minimality based on the combinatorics of a given polytope, therefore providing an obstruction for all polytopes in the given combinatorial class. 
Our main tool is a symbolic version of the slack matrix of a polytope defined as follows.

\begin{definition} The {\em symbolic slack matrix} of a $d$-polytope $P$ is the matrix, $S_P(x)$, obtained by replacing all positive entries in the slack matrix $S_P$ of $P$ with distinct variables $x_1, \ldots, x_t$.
\end{definition}

Note that two $d$-polytopes $P$ and $Q$ are in the same combinatorial class if and only if
$S_P(x) = S_Q(x)$ up to permutations of rows and columns, and names of variables. In this paper we say 
that a polynomial $f \in \RR[x_1,\ldots,x_t]$ is a {\em monomial} if it is of the form $f = \pm x^a$ where $x^a = x_1^{a_1} \cdots x_t^{a_t}$ and $a=(a_1,\ldots,a_t) \in \NN^t$.  We refer to a  sum of two distinct monomials as a {\em binomial} and to the sum of three distinct monomials as a {\em trinomial}. 
 This differs from the usual terminology where nontrivial coefficients are allowed. 
 
\begin{lemma}[Trinomial Obstruction Lemma] \label{lem:trinomial obstruction}
Suppose the symbolic slack matrix $S_P(x)$ of a $d$-polytope $P$ has a $(d+2)$-minor 
that is a trinomial. Then no polytope in the combinatorial class of $P$ can be psd-minimal.
\end{lemma}

\begin{proof}  Suppose $Q$ is psd-minimal and combinatorially equivalent to $P$.
Hence, we can assume that $S_P(x)$ equals $S_Q(x)$. By Theorem~\ref{thm:square root rank} there is some $u=(u_1,\ldots, u_t)\in\RR^t$, with no coordinate equal to zero, such that $S_Q=S_P(u^2_1,\ldots,u^2_t)$ and $\rank S_P(u)=d+1$. Since $S_Q$ is the slack matrix of a $d$-polytope, we have  
$$
\rank S_P(u^2_1,\ldots,u^2_t)=d+1 = \rank S_P(u_1,\ldots, u_t).
$$

Now suppose $D(x)$ is a trinomial $(d+2)$-minor of $S_P(x)$. Up to sign, $D(x)$ has the form $x^a + x^b + x^c$ or $x^a - x^b + x^c$ for some $a, b, c\in \NN^t$. In either case, 
it is not possible for $D(u^2_1,\ldots,u^2_t)=D(u_1,\ldots,u_t)=0$.
\end{proof}

An interesting property of this obstruction is that it reflects the fact that faces of psd-minimal polytopes are psd-minimal (see Lemma~\ref{lem:face_inheritance}): if a face of a polytope is not psd-minimal due to a trinomial obstruction, then the non psd-minimality of the polytope can also be verified by a trinomial obstruction.

\begin{proposition} \label{prop:trinomial}
Let $P$ be a $d$-polytope with a facet $F$ such that some $(d+1)$-minor of 
$S_F(x)$ is a trinomial. Then $S_P(x)$ has a trinomial $(d+2)$-minor. 
\end{proposition}

\begin{proof} Let vertices $v_1, \ldots, v_{d+1}$ of $F$ and 
facets $F'_1, \ldots, F'_{d+1}$ of $F$ index a $(d+1) \times (d+1)$ submatrix of $S_F(x)$ with a trinomial determinant.  Let $F_i$ be the unique facet of $P$ that shares $F'_i$ with the facet $F$ of $P$. Pick a vertex $v_{d+2}$ of $P$ not lying on $F$. Then the determinant of the submatrix of $S_P(x)$ indexed by $v_1,\ldots, v_{d+2}$ and $F_1,\ldots, F_{d+1}, F$ is a trinomial.
\end{proof}

\subsection{Psd-minimal $2$-polytopes}
Lemma~\ref{lem:trinomial obstruction}  yields simple proofs of the combinatorial part of the classification results for psd-minimal $2$- and $3$-polytopes 
that were obtained in \cite{gouveia2013polytopes}.

\begin{proposition} \cite[Theorem~4.7]{gouveia2013polytopes} \label{prop:psd minimal in the plane}
The psd-minimal $2$-polytopes are precisely all triangles and quadrilaterals.
\end{proposition}

\begin{proof}
Let $P$ be an $n$-gon where $n > 4$. Then $S_P(x)$ has a submatrix of the form 
$$ \begin{bmatrix} 
0 & x_1 & x_2 & x_3 \\ 0 & 0 & x_4 & x_5 \\ x_6 & 0 & 0 & x_7 \\ x_8 & x_9 & 0 & 0 
\end{bmatrix},$$
whose determinant is $x_1x_4x_7x_8 -  x_2x_5x_6x_9 + x_3x_4x_6x_9 $ up to sign. 
By Lemma~\ref{lem:trinomial obstruction}, no $n$-gon with $n > 4$ can be psd-minimal.

Since all triangles are projectively equivalent, by verifying the psd-minimality of one, they are 
all seen to be psd-minimal. Similarly, for quadrilaterals.
\end{proof}

\subsection{Combinatorial classes of psd-minimal  $3$-polytopes}
Lemma~\ref{lem:trinomial obstruction} can also be used to derive  \cite[Proposition~4.10]{gouveia2013polytopes}, 
which gives  the psd-minimal classification of $3$-polytopes
up to combinatorial equivalence. 

Using Proposition~\ref{prop:psd minimal in the plane}, together with Lemma~\ref{lem:face_inheritance} and the invariance of psd rank under polarity, we get that that any $3$-polytope $P$ with a vertex of degree larger than four, or a facet that is an $n$-gon where $n > 4$, 
cannot be  psd-minimal.  This is enough to prove a stronger version of \cite[Lemma~4.9]{gouveia2013polytopes}.

\begin{lemma} \label{lem:no square facet incident to a degree 4 vertex}
If $P$ is a $3$-polytope with a vertex of degree four and a quadrilateral 
facet incident to this vertex, then $S_P(x)$ contains a trinomial $5$-minor.
\end{lemma}

\begin{proof} Let $v$ be the vertex of degree four incident to facets $F_1,F_2,F_3,F_4$ such that $[v_1,v]=F_1\cap F_2$, $[v_2,v]=F_2\cap F_3$, $[v_3, v]=F_3\cap F_4$ and $F_4\cap F_1$ are edges of $P$, where $v_1$, $v_2$ and $v_3$ are vertices of $P$. 

Suppose $F_4$ is quadrilateral. Then $F_4$ has a vertex $v_4$ that is different from, and non-adjacent to, $v$. Therefore, $v_4$ does not lie on  $F_1$, $F_2$ or  $F_3$.  Consider the $5 \times 5$ submatrix of $S_P(x)$ with rows indexed by $v,v_1,v_2,v_3,v_4$ and columns by 
$F_1,F_2,F_3,F_4,F$ where $F$ is a facet not containing $v$. This matrix has the form 
$$
\begin{bmatrix} 0 & 0 & 0 & 0 & x_1\\ 0 & 0 & x_2 & x_3 & *\\ x_4 & 0 & 0 & x_5 & * \\ x_6 & x_7 & 0 & 0 & * \\ x_8 & x_9 & x_{10} & 0 & * \\ 
\end{bmatrix},
$$
and its determinant is a trinomial. 
\end{proof}

\begin{proposition} \label{prop:combinatorial classification in 3-space}
The psd-minimal $3$-polytopes are combinatorially equivalent to simplices, quadrilateral pyramids, bisimplices, octahedra or 
their duals.
\end{proposition}

\begin{proof}
Suppose $P$ is a psd-minimal $3$-polytope. If $P$ contains only vertices of degree three and triangular facets, then $P$ is a simplex.

For all remaining cases, $P$ must have a vertex of degree four or a quadrilateral facet.  Since psd rank is preserved under polarity, we may assume that $P$ has a vertex $u$ of degree four.  By Lemma~\ref{lem:no square facet incident to a degree 4 vertex}, the neighborhood of $u$ looks as follows.
\begin{center}
\begin{tikzpicture}
  [scale=.2,auto=center]
  \coordinate (p) at (6,6);
  \coordinate (p1) at (1,11);
  \coordinate (p2) at (1,1);
  \coordinate (p3) at (11,1);
  \coordinate (p4) at (11,11);
  \node[label=$u$] at (p) {};
  \node[label=$v_1$] at ($(p1)-(1.2,1.7)$) {};
  \node[label=$v_2$] at ($(p2)-(1.2,1.3)$) {};
  \node[label=$v_3$] at ($(p3)+(1.2,-1.3)$) {};
  \node[label=$v_4$] at ($(p4)+(1.2,-1.7)$)  {};
  \foreach \from/\to in {p/p1,p/p2,p/p3,p/p4,p3/p4,p4/p1,p1/p2,p2/p3}
    \draw (\from) -- (\to);
\end{tikzpicture}
\end{center}

Suppose $P$ has five vertices. If all edges of $P$ are in the picture, i.e. the picture is a Schlegel diagram of $P$, then $P$ is a quadrilateral pyramid. Otherwise $P$ has one more edge, and this edge is $[v_1,v_3]$ or $[v_2, v_4]$, yielding a bisimplex in either case. 

If $P$ has more than five vertices, then we may assume that $P$ has a vertex $v$ that is a 
neighbor of $v_1$ different from $u$, $v_2$, $v_4$.  Then $v_1$ is a degree four vertex and thus, by Lemma~\ref{lem:no square facet incident to a degree 4 vertex}, all facets of $P$ containing $v_1$ are triangles.  
This implies that $v$ is a neighbor of $v_2$ and $v_4$. Applying the same logic to either $v_2$ or $v_4$, we get that $v$ is also a neighbor of $v_3$. Since all these vertices now have degree four, there could be no further vertices in $P$, and so $P$ is an octahedron.
Hence $P$ is combinatorially 
equal to, or dual to, one of the polytopes seen so far.
\end{proof}

Proposition~\ref{prop:combinatorial classification in 3-space} proves the combinatorial part of 
Theorem~\ref{thm:classification in 3-space}. The rest of the proof can be seen in 
Section~\ref{sec:polys16-31}.

\section{Facet intersection obstructions for psd-minimal $4$-polytopes} \label{sec:facet intersection obstructions}

We now use the trinomial obstruction lemma to show that facets of psd-minimal $4$-polytopes can only intersect in a limited number of ways. In this section and beyond, when we refer to a concrete polytope we refer to its combinatorial type; for example, ``cube"  is used as a shortcut for ``a polytope of the same combinatorial type as a cube". Similarly, we refer to any quadrilateral as a ``square''.

\subsection{Combinatorial obstruction lemmas}

\begin{definition} \label{def:wedge}
An edge of a $3$-polytope $P$ is called a {\it wedge} if all the vertices of $P$ are contained in the 
two facets of $P$ that intersect at this edge.
\end{definition}

\begin{lemma}\label{lem:edge_exclusion}
If two facets of a psd-minimal $4$-polytope $P$ intersect at an edge, then this edge must be a wedge of both facets.
\end{lemma}
\begin{proof}
Suppose the contrary: there is a facet $F_1$ of $P$ which intersects another facet $F_2$ of $P$ at an edge $[u,v]$ that is not a wedge of $F_1$.

Let $v_1$ be a vertex of $P$ not contained in $F_1$. Let $F_3$ ($F_4$ respectively) be a facet of $P$ which contains $u$ but not $v$ (contains $v$ but not $u$ respectively). Let $F_5$ and $F_6$ be two facets of $P$, which contain $u$ and $v$ and induce two different facets of $F_1$.  Let $v_2$ be a vertex of both $F_6$ and $F_1$, but not of $F_5$; let $v_3$ be a vertex of both $F_5$ and $F_1$, but not of $F_6$. Moreover, since $[u,v]$ is not a wedge of $F_1$ there is a vertex $v_4$ of $F_1$, which does not lie in $F_5$ or $F_6$.

Consider the $6 \times 6$ submatrix of $S_P(x)$ with rows indexed by $v_1,v,u,v_2,v_3,v_4$ and columns by 
$F_1, F_3, F_4, F_5, F_6,F_2$. This matrix has the form
$$
\begin{bmatrix}
 x_1 & * & * & * & * & * \\
 0 & x_2 & 0 & 0 & 0 & 0 \\
 0 & 0 & x_3 & 0 & 0 & 0 \\
 0 & * & * & x_4 & 0 & x_5 \\
 0 & * & * & 0 & x_6 & x_7 \\
 0 & * & * & x_8 & x_9 & x_{10}
\end{bmatrix},$$
and its determinant is a trinomial, contradicting Lemma~\ref{lem:trinomial obstruction}.
\end{proof} 

\begin{lemma}\label{lem:vertex_exclusion}
If two facets of a psd-minimal $4$-polytope $P$ intersect at a vertex, none of their facets containing that vertex is a square.
\end{lemma}
\begin{proof}
Let $F_1$ be a facet of $P$ which intersects another facet $F_2$ of $P$ at a vertex $v$. Suppose $v$ is contained in a square facet $F'$ of $F_1$. 

Let $F_3$ be a facet of $P$ that does not contain $v$. Pick $v_1$ as  a vertex of $P$ not contained in $F_1$.
Let $F_4$ be the facet of $P$ which intersects $F_1$ at $F'$. Take $v_2$ as a vertex of $F_1$ not contained in $F'$. Let $v_3$ and $v_4$ be the neighbors of $v$ in the square $F'$, and let $F_5$ and $F_6$ be facets of $P$ which intersect $F'$ at the edge $[v_4,v]$ and $[v_3,v]$, respectively. Denote by $v_5$ the vertex of $F'$ different from $v$, $v_3$ and $v_4$. 

Consider the $6 \times 6$ submatrix of $S_P(x)$ with rows indexed by $v,v_1,v_2,v_3,v_4,v_5$ and columns by 
$F_3, F_1, F_4, F_5, F_6,F_2$. This matrix has the form
$$\begin{bmatrix}
 x_1 & 0 & 0 & 0 & 0 & 0 \\
 * & x_2 & * & * & * & * \\
 * & 0 & x_3 & * & * & * \\
 * & 0 & 0 & x_4 & 0 & x_5 \\
 * & 0 & 0 & 0 & x_6 & x_7\\
 * & 0 & 0 & x_8 & x_9 & x_{10}
\end{bmatrix},$$
and its determinant is a trinomial, contradicting Lemma~\ref{lem:trinomial obstruction}.
\end{proof}

While the above two lemmas are general, the next one deals specifically with octahedral facets. Recall that a combinatorial octahedron is psd-minimal if and only if it is biplanar (see Theorem~\ref{thm:classification in 3-space}). 

\begin{lemma} \label{lem:vertex_exclusion_octahedron}
An octahedral facet of a psd-minimal $4$-polytope cannot intersect another facet of the polytope at a vertex.
\end{lemma}

\begin{proof}
Suppose there are facets $F$ and $G$ of a psd-minimal $4$-polytope $P$, such that $F$ is octahedral and $F\cap G=\{v_1\}$, where $v_1$ is a vertex of $P$. 

Consider the submatrix $M$ of $S_P$ with rows  indexed by the vertices $v_1,\ldots, v_6$ of the octahedral facet and columns indexed by the facets  $F_1,\ldots, F_8$ and $G$ of $P$, such that $F_1\cap F$, \ldots, $F_8\cap F$ are different facets of $F$. The symbolic form of $M$ is 
$$M(x) = \begin{bmatrix}
 x_1& x_2& x_3& x_4& 0& 0& 0& 0& 0  \\
 0& 0& 0& 0& x_5& x_6& x_7& x_8& z_1 \\
 x_9& x_{10}& 0& 0& x_{11}& x_{12}& 0& 0& z_2 \\
 0& 0& x_{13}& x_{14}& 0& 0& x_{15}& x_{16}& z_3 \\
 x_{17}& 0& x_{18}& 0& x_{19}& 0& x_{20}& 0& z_4 \\
 0& x_{21}& 0& x_{22} & 0& x_{23}& 0& x_{24}& z_5
\end{bmatrix}\,.$$

Observe that $\sqrtrank{M} = 4$, since otherwise $M$ extended by $F$ and a vertex of $P$, which is not contained in $F$, has square root rank bigger than five, contradicting the psd-minimality of the $4$-polytope $P$. On the other hand, 
$\rank{M} = 4$ because the rows of $M$ are indexed by the vertices of the $3$-polytope $F$.

Let $M'$ be the matrix obtained from $M$ by dropping the column indexed by $G$.
Since the octahedron $F$ is psd-minimal, without loss of generality we may assume that the first four rows are linearly dependent in both $M'$ and each of its Hadamard square roots of 
rank four. For a justification of this assumption, we refer the reader to Remark~\ref{rem:psd minimal cube dependencies}.~\footnote{This fact can be proved independently but follows easily from the algebraic machinery developed in the second half of the paper.}
Thus both $M$ and each of its square roots of rank four lie in the variety of the ideal generated by  the 
$4$-minors of the upper left $4\times 8$ submatrix of $M(x)$, and the $5$-minors of $M(x)$. 
Using a computer algebra system one can verify that this ideal contains trinomials, such as 
$$x_{11}x_{16}x_{17}z_5-x_1x_{11}x_{16}x_{18}z_5-x_3x_9x_{16}x_{19}z_5.$$ 
As in Lemma~\ref{lem:trinomial obstruction}, since both $M$ and its square root must satisfy this trinomial, no  
polytope combinatorially equivalent to $P$ can be psd-minimal.
\end{proof}

\subsection{Possible facet intersections of a psd-minimal $4$-polytope}

Based on the three lemmas above we can now provide a short list of allowed intersections among the facets of a psd-minimal $4$-polytope. This is the key tool in finding the 
combinatorial classes of psd-minimal $4$-polytopes.

\begin{proposition}\label{prop:intersections}
Let $P$ be a psd-minimal $4$-polytope and $F$ and $G$ be facets of $P$ intersecting at a vertex or an edge. Then, either $F$ is a simplex or one of the following conditions (illustrated in Figure~\ref{fig:intersections}) hold:
\begin{enumerate}
\item $F$ is a bisimplex and $F \cap G$ is a vertex;
\item $F$ is a triangular prism and $F \cap G$ is one of the edges linking the two triangular faces;
\item $F$ is a square pyramid and $F \cap G$ is the apex or an edge of the base. 
\end{enumerate}
\end{proposition}
\begin{proof}
This follows from the previous lemmas as all the other possible edge intersections are not wedges and hence are not intersections by Lemma~\ref{lem:edge_exclusion}. Similarly, all the other possible vertex intersections lie in a square face and hence are not intersections by Lemma~\ref{lem:vertex_exclusion}, except for the vertices of the octahedron, which are not intersections directly from Lemma~\ref{lem:vertex_exclusion_octahedron}.
\end{proof}

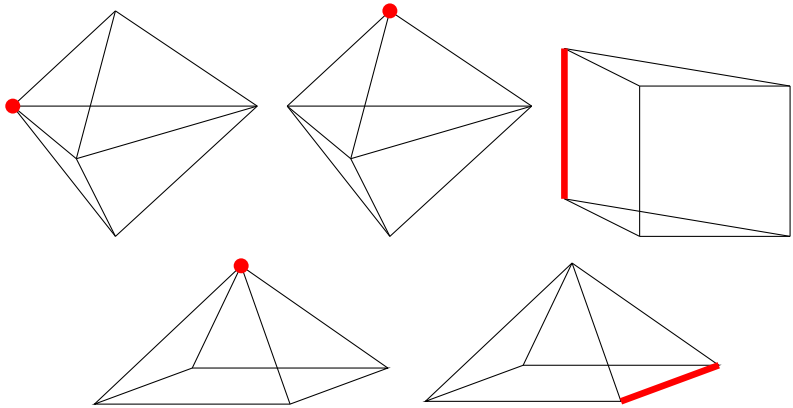
\begin{figure}[h]

 		    \begin{tikzpicture}[xscale=.65, yscale=0.5, x={(1cm,0cm)},y={(-.3cm,-0.35cm)}, z={(0cm,2cm)}]
			 \coordinate (v1) at (-2.5, -1, 0);
			 \coordinate (v2) at (2.5, -1, 0);
			 \coordinate (v3) at (0, 3, 0);
			 \coordinate (v4) at ($1/3*(v1)+1/3*(v2)+1/3*(v3)+(0,0,1.5)$);
   			 \coordinate (v5) at  ($1/3*(v1)+1/3*(v2)+1/3*(v3)-(0,0,1.5)$);
			 \draw (v1)--(v2);
			 \draw (v1)--(v3);
			 \draw (v2)--(v3);
			 \draw (v1)--(v4);
			 \draw (v2)--(v4);
			 \draw (v3)--(v4);
			 \draw (v1)--(v5);
			 \draw (v2)--(v5);
			 \draw (v3)--(v5);
			 \node[scale=0.6,circle,fill=red!] () at (v1){}; 
		   \end{tikzpicture} 		   
		   \hspace{0.15cm}
 		    \begin{tikzpicture}[xscale=.65, yscale=0.5, x={(1cm,0cm)},y={(-.3cm,-0.35cm)}, z={(0cm,2cm)}]
			 \coordinate (v1) at (-2.5, -1, 0);
			 \coordinate (v2) at (2.5, -1, 0);
			 \coordinate (v3) at (0, 3, 0);
			 \coordinate (v4) at ($1/3*(v1)+1/3*(v2)+1/3*(v3)+(0,0,1.5)$);
   			 \coordinate (v5) at  ($1/3*(v1)+1/3*(v2)+1/3*(v3)-(0,0,1.5)$);
			 \draw (v1)--(v2);
			 \draw (v1)--(v3);
			 \draw (v2)--(v3);
			 \draw (v1)--(v4);
			 \draw (v2)--(v4);
			 \draw (v3)--(v4);
			 \draw (v1)--(v5);
			 \draw (v2)--(v5);
			 \draw (v3)--(v5);
			 \node[scale=0.6,circle,fill=red!] () at (v4){}; 
		   \end{tikzpicture}  
		   		   \hspace{0.15cm}
		   \begin{tikzpicture}[yscale=1, xscale=2, x={(1cm,0cm)},y={(-.5cm,0.5cm)}, z={(0cm,2cm)}]
			 \coordinate (v3) at (0, 1, 0);
			 \coordinate (v2) at (1, 0, 0);
			 \coordinate (v1) at (0, 0, 0);
			 \coordinate (v6) at (0, 1, 1);
   			 \coordinate (v5) at (1, 0, 1);
			 \coordinate (v4) at (0, 0, 1);
			 \draw (v1)--(v2);
			 \draw (v4)--(v5);
			 \draw (v2)--(v3);
			 \draw (v5)--(v6);
			 \draw (v3)--(v1);
			 \draw (v6)--(v4);
			 \draw(v1)--(v4);
			 \draw (v2)--(v5);
			 \draw (v3)--(v6);
			 \draw[red, line width=2.5pt] (v3)--(v6);
		   \end{tikzpicture}
		   \hspace{0.15cm} 
       
\vspace{0.25cm}

		    \begin{tikzpicture}[yscale=.8, xscale=1.3, x={(1cm,0cm)},y={(-.5cm,-.3cm)}, z={(0cm,2cm)}]
			 \coordinate (v1) at (-1, -1, 0);
			 \coordinate (v2) at (1, -1, 0);
			 \coordinate (v3) at (1, 1, 0);
			 \coordinate (v4) at (-1, 1, 0);
   			 \coordinate (v5) at (0, 0, 1);
			 \draw (v1)--(v2)--(v3)--(v4)--(v1);
			 \draw (v1)--(v5);
			 \draw (v2)--(v5);
			 \draw (v3)--(v5);
			 \draw (v4)--(v5);
			 \node[scale=0.6,circle,fill=red!] () at (v5){}; 
		   \end{tikzpicture}       
		   \hspace{0.25cm}
		    \begin{tikzpicture}[yscale=.8, xscale=1.3, x={(1cm,0cm)},y={(-.5cm,-.3cm)}, z={(0cm,2cm)}]
			 \coordinate (v1) at (-1, -1, 0);
			 \coordinate (v2) at (1, -1, 0);
			 \coordinate (v3) at (1, 1, 0);
			 \coordinate (v4) at (-1, 1, 0);
   			 \coordinate (v5) at (0, 0, 1);
			 \draw (v1)--(v2)--(v3)--(v4)--(v1);
			 \draw (v1)--(v5);
			 \draw (v2)--(v5);
			 \draw (v3)--(v5);
			 \draw (v4)--(v5);
			 \draw[red, line width=2.5pt] (v3)--(v2);
		   \end{tikzpicture}
    \caption{All possible vertex or edge intersections among facets of a psd-minimal $4$-polytope (excluding the simplex).} \label{fig:intersections}

\end{figure}
 
We state one more result about the combinatorics of a psd-minimal $4$-polytope.
 
\begin{lemma}\label{lem:edge_in_three_four_facets}
An edge of a psd-minimal $4$-polytope is contained in at most four facets.
\end{lemma}
\begin{proof}
This follows dually from the fact that faces of dimension two of a psd-minimal $4$-polytope are themselves psd-minimal and hence have at most four vertices. 
\end{proof}

\section{Classification of combinatorial types} \label{sec:31 combinatorial classes}

In this section, we prove our main theorem.

\begin{theorem}\label{thm:31classes}
There are exactly $31$ combinatorial classes of psd-minimal $4$-polytopes.
\end{theorem}

\afterpage{ \clearpage %

\begin{landscape}
\begin{small}
\begin{table}[H]
\begin{tabular}{l|l|l|l|l|l|l|l} 
 $\sharp$   & Fig   & Construction                                         & Vertices of a psd-minimal embedding                                    & Facet Types             &Dual                     & $f$-vector       & Condition     \\ \hline
$1$ &       & $\Delta_4$                                                 & $\{-e_{1234},e_1,e_2,e_3,e_4\}$                                        & $5$S                    &Self                     & $(5,10,10,5)$   &Cor.~\ref{cor:classes 1-11}     \\ \hline
$2$ &       & $(\Delta_1 \times \Delta_1) \ast \Delta_1$                 & $\{\pm e_1, \pm e_2, e_3,e_4\}$                                        & $4$S,$2$Py              &Self                     & $(6,13,13,6)$   &Cor.~\ref{cor:classes 1-11}       \\ \hline
$3$ &\ref{fig:sec4_bisimplex_result2}&                                                            & $\{0,2e_1,2e_2,2e_3,e_{12}-e_3,e_4,e_{34}\}$                           & $3$S,$2$Py,$2$B         &Self                      & $(7,17,17,7)$  &Cor.~\ref{cor:classes 1-11}       \\ \hline
$4$ &\ref{fig:sec4_trprism_result3} & $\Delta_3 \times \Delta_1$                                 & $\{-e_{123},e_1,e_2,e_3\}+\{\pm e_4\}$                                 & $2$S,$4$Pr              &$5$                      & $(8,16,14,6)$   &Cor.~\ref{cor:classes 1-11}       \\ \hline
$5$ &       & $\Delta_3 \oplus \Delta_1$                                 & $\{-e_{123},e_1,e_2,e_3,\pm e_4\}$                                     & $8$S                    &$4$                      & $(6,14,16,8)$   &Cor.~\ref{cor:classes 1-11}       \\ \hline
$6$ &\ref{fig:sec4_trprism_result1} & $\Delta_2 \times \Delta_2$                                 & $\{-e_{12},e_1,e_2\}+\{-e_{34},e_3,e_4\}$                              & $6$Pr                   &$7$                       & $(9,18,15,6)$  &Cor.~\ref{cor:classes 1-11}       \\ \hline
$7$ &       & $\Delta_2 \oplus \Delta_2$                                 & $\{-e_{12},e_1,e_2,-e_{34},e_3,e_4\}$                                  & $9$S                    &$6$                      & $(6,15,18,9)$   &Cor.~\ref{cor:classes 1-11}       \\ \hline
$8$ &\ref{fig:sec4_trprism_result10} & $(\Delta_2 \times \Delta_1) \ast \Delta_0$                 & $\{e_4\} \cup(\{-e_{12},e_1,e_2\} + \{\pm e_3\})$                      & $2$S,$1$Pr,$3$Py        &$9$                     & $(7,15,14,6)$  &Cor.~\ref{cor:classes 1-11}         \\ \hline
$9$ && $(\Delta_2 \oplus \Delta_1) \ast \Delta_0$                 & $\{-e_{12},e_1,e_2,\pm e_3, e_4\}$                                     & $6$S,$1$B               &$8$                      & $(6,14,15,7)$          &Cor.~\ref{cor:classes 1-11}  \\ \hline
$10$&\ref{fig:sec4_trprism_result2} &                                                            & $\{0,e_1,e_2,e_3,e_{13},e_{23},e_4,e_{14}\}$                           & $1$S,$2$Pr,$4$Py        &$11$                     & $(8,18,17,7)$   &Cor.~\ref{cor:classes 1-11}       \\ \hline
$11$&       &                                                            & $\{e_1,e_2,e_3,e_4,-2e_1-e_{24},-e_{13}-2e_2,-2e_{12}\}$               & $4$S,$4$Py              &$10$                     & $(7,17,18,8)$   &Cor.~\ref{cor:classes 1-11}     \\ \hline
$12$&\ref{fig:sec4_trprism_result6}&                                                            & $\{0,e_1,e_2/2,e_3,e_4,e_{14},e_{12}/2,e_{13},e_2+4e_{34}\}$           & $3$Pr,$3$Py,$2$B        &$13$                     & $(9,22,21,8)$    &Prop.~\ref{prop:polytope 12 13}     \\ \hline
$13$&\ref{fig:sec4_bisimplex_result3}&                                                            & $\{e_1,e_2,9/4e_3,e_4,e_{124}/2,e_{13},e_2+e_3/4,e_{34}\}$             & $2$S,$6$Py,$1$B         &$12$                     & $(8,21,22,9)$  &Prop.~\ref{prop:polytope 12 13}       \\ \hline
$14$&\ref{fig:sec4_trprism_result4} & $(\Delta_2 \oplus \Delta_1)\times\Delta_1$                 & $\{0,e_1,e_2,e_3,e_4,e_{12},e_{23},e_{24},2e_{13}+e_4,2e_{13}+e_{24}\}$& $6$Pr,$2$B              &$15$                     & $(10,23,21,8)$   &Prop.~\ref{prop:polytope 14 15}     \\ \hline
$15$&       & $(\Delta_2 \times \Delta_1)\oplus\Delta_1$                 & $\{e_1,2e_2,e_3,2e_4,e_2+2e_3,e_2+4e_4,2e_1+e_2,e_{134}\}$             & $4$S,$6$Py              &$14$                     & $(8,21,23,10)$       &Prop.~\ref{prop:polytope 14 15} \\ \hline
$16$&\ref{fig:sec4_cube_case4} & $(\Delta_1 \times \Delta_1 \times \Delta_1) \ast \Delta_0$ & $(\{\pm e_1\}+ \{\pm e_2\} +\{\pm e_3\})\cup\{e_4\}$                   & $1$C,$6$Py              &$17$                     & $(9,20,18,7)$  & Prop.~\ref{prop:classes 16-29}     \\ \hline
$17$& & $(\Delta_1 \oplus \Delta_1 \oplus \Delta_1) \ast \Delta_0$ & $\{\pm e_1, \pm e_2, \pm e_3, e_4\}$                                   & $1$O,$8$S               &$16$                     & $(7,18,20,9)$      & Prop.~\ref{prop:classes 16-29}   \\ \hline
$18$&\ref{fig:sec4_trprism_result8}&                                                            & $\{0,e_1,e_2/2 e_4,e_{234},e_{23},e_{24}/2,e_{134},e_{13}\}$            & $2$Pr,$4$Py,$2$B        &$19$                     & $(9,22,21,8)$      & Prop.~\ref{prop:classes 16-29}    \\ \hline
$19$&\ref{fig:sec4_octahedron_result2}&                                                            & $\{0,e_1,e_3,e_4,e_{14},e_{23},e_{24},e_{234}\}$                       & $1$O,$4$S,$4$Py         &$18$                     & $(8,21,22,9)$    & Prop.~\ref{prop:classes 16-29}      \\ \hline
$20$&\ref{fig:sec4_cube_case3} & $((\Delta_1\times\Delta_1)\ast\Delta_0) \times \Delta_1$   & $\{\pm e_1, \pm e_2, e_3\} + \{\pm e_4\}$                              & $1$C,$4$Pr,$2$Py        &$21$                     & $(10,21,18,7)$  & Prop.~\ref{prop:classes 16-29}       \\ \hline
$21$&       & $((\Delta_1\times\Delta_1)\ast\Delta_0) \oplus \Delta_1$   & $\{\pm e_1, \pm e_2, e_3, e_3/2 \pm e_4\}$                             & $8$S,$2$Py              &$20$                     & $(7,18,21,10)$     & Prop.~\ref{prop:classes 16-29}    \\ \hline
$22$&       &                                                            & $\{0,2e_1,2e_3,2e_4,e_{12},e_{123},e_{1234},2e_{24},2e_{34}\}$         & $6$Py,$3$B              &$23$                     & $(9,24,24,9)$     & Prop.~\ref{prop:classes 16-29}     \\ \hline
$23$&\ref{fig:sec4_trprism_result12}&                                                            & $\{0,e_1,e_3,e_4,e_{12},e_{123},e_{23},e_{24},e_{234}\}$               & $2$O,$3$S,$1$Pr,$3$Py   &$22$                     & $(9,24,24,9)$     & Prop.~\ref{prop:classes 16-29}     \\ \hline
{$24$}&       &                                                            & $\{0,2e_1,2e_2,2e_3,2e_4,e_{123},e_{124},e_{134},e_{1234},e_{234}\}$   & $10$B                   &$25$                     & $(10,30,30,10)$     & Prop.~\ref{prop:classes 16-29}   \\ \hline
{$25$}&\ref{fig:sec4_octahedron_result1}&                                                            & $\{e_1,e_2,e_3,e_4,e_{12},e_{13},e_{14},e_{23},e_{24},e_{34}\}$        & $5$O,$5$S               &$24$                     & $(10,30,30,10)$   & Prop.~\ref{prop:classes 16-29}     \\ \hline
$26$&       & $(\Delta_1 \times \Delta_1 \times \Delta_1)\oplus \Delta_1$& $(\{\pm e_1\}+\{\pm e_2\}+\{\pm e_3\}) \cup \{\pm e_4\}$               & $12$Py                  &$27$                     & $(10,28,30,12)$     & Prop.~\ref{prop:classes 16-29}   \\ \hline
$27$&\ref{fig:sec4_trprism_result5}& $(\Delta_1 \oplus \Delta_1 \oplus \Delta_1)\times \Delta_1$& $\{\pm e_1,\pm e_2, \pm e_3\} + \{\pm e_4\}$                           & $2$O,$8$Pr              &$26$                     & $(12,30,28,10)$     & Prop.~\ref{prop:classes 16-29}   \\ \hline
{$28$}&\ref{fig:sec4_cube_case2} & $\Delta_1 \times \Delta_1 \times \Delta_2$                 & $\{\pm e_1\} + \{\pm e_2\} + \{-e_{34},e_3,e_4\}$                      & $3$C,$4$Pr              &$29$                     & $(12,24,19,7)$    & Prop.~\ref{prop:polys 28-29}    \\ \hline
{$29$}&       & $\Delta_1 \oplus \Delta_1 \oplus \Delta_2$                 & $\{\pm e_1, \pm e_2, -e_{34},e_3,e_4\}$                                & $12$S                   &$28$                     & $(7,19,24,12)$    & Prop.~\ref{prop:polys 28-29}     \\ \hline
$30$&\ref{fig:sec4_cube_case1} & $\Delta_1 \times \Delta_1 \times \Delta_1 \times \Delta_1$ & $\{\pm e_1\} + \{\pm e_2\} + \{\pm e_3\} + \{\pm e_4\}$                & $8$C                    &$31$                     & $(16,32,24,8)$    & Prop.~\ref{prop:polytope 30 31}    \\ \hline
$31$&       & $\Delta_1 \oplus \Delta_1 \oplus \Delta_1 \oplus \Delta_1$ & $\{\pm e_1,\pm e_2,\pm e_3\pm e_4\}$                                   & $16$S                   &$30$                     & $(8,24,32,16)$    & Prop.~\ref{prop:polytope 30 31}    \\ \hline
\end{tabular}~
\caption{Combinatorial classes of psd-minimal $4$-polytopes} \label{table:list}
\end{table}
\end{small}
\end{landscape}

}

In Table~\ref{table:list} we list all these classes with the following information in each column:
\begin{enumerate}
\item The number which we use to refer to the class.
\item The reference to its Schlegel diagram, when present in the paper.
\item A construction (if one is known) of a polytope in the class using standard operations on simplices. In the table, $\Delta_d$ denotes a simplex with $d+1$ vertices, and  $\times$, $\oplus$ and $\ast$ denote the product, free sum and join operations.
\item The vertices of a psd-minimal polytope in the combinatorial class, where psd-minimality of the polytope can be verified by constructing a slack matrix and
checking its square root rank. For $S\subseteq \{1,2,3,4\}$, $e_S$ denotes the zero-one vector with a one in position $i$ if and only if $i\in S$. 
\item Information about the combinatorial types of facets, i.e. how many facets of a polytope in this class are cubes~($C$), triangular prisms~($Pr$), octahedra~($O$), bisimplices~($B$), square pyramids~($Py$) and simplices~($S$).
\item The number of the dual of the combinatorial class, or \emph{Self} if it is self dual. Recall that psd-minimality is closed under polarity, and we list 
dual pairs of combinatorial classes consecutively.
\item The $f$-vector of the polytopes in the the combinatorial class.
\item The reference to the result characterizing psd-minimality in the combinatorial class. 
\end{enumerate}

Since we provide a psd-minimal polytope in each of the $31$ classes of Table~\ref{table:list}, to prove Theorem~\ref{thm:31classes}, it suffices to argue that 
there are no further combinatorial classes of psd-minimal $4$-polytopes.
To enumerate the possible combinatorial classes of psd-minimal $4$-polytopes we systematically study the combinatorics of their facets. For example, we start by listing all combinatorial classes of psd-minimal $4$-polytopes with a cubical facet. We assume that a psd-minimal $4$-polytope $P$ has a facet combinatorially equivalent to a cube, and then using combinatorial arguments we leverage the facet intersection results of Proposition~\ref{prop:intersections} to dramatically restrict the possible combinatorial classes of~$P$. After all classes of psd-minimal $4$-polytopes with a cubical facet are obtained, we add them and their dual classes to our list of possibilities, and move on to study polytopes with a different combinatorial type of facet, until we cover all cases. Naturally, when studying a new type of facet, we can assume that neither the polytope $P$ nor its dual has facets of the combinatorial types previously studied. Otherwise, we would be double counting. The combinatorial types of facets to consider are those of the psd-minimal $3$-polytopes, i.e., cubes, triangular prisms, octahedra, bisimplices, square pyramids and simplices in this order. We treat each case in a separate subsection. 

\subsection{Cube}

Let $P$ be a psd-minimal $4$-polytope with a cubical facet $F$. We provide a detailed 
discussion of this case, skipping similar arguments in later subsections.
By Proposition~\ref{prop:intersections}, any other facet of $P$ either has an empty intersection with $F$ or intersects $F$ in one of its facets. This fact is crucial for the analysis below.  Figure~\ref{fig:sec4_cube_marked_facets} shows the graph of $F$, along with some marked facets, $F'_1$, $F'_2$, $F'_3$ and $F'_4$ of $F$.

\begin{figure}[h]
  		    \begin{tikzpicture}[scale=0.45]
   			 \coordinate (v1) at (-1.5,-1.5);
			 \coordinate (v2) at (1.5,-1.5);
			 \coordinate (v3) at (1.5,1.2);
			 \coordinate (v4) at (-1.5,1.2);
   			 \coordinate (v5) at (-3,-2);
			 \coordinate (v6) at (3,-2);
			 \coordinate (v7) at (3,3);
			 \coordinate (v8) at (-3,3);
			 \draw[very thick] (v1)--(v2)--(v3)--(v4)--(v1);
			 \draw[very thick] (v5)--(v6)--(v7)--(v8)--(v5);
			 \draw[very thick] (v1)--(v5);
			 \draw[very thick] (v2)--(v6);
			 \draw[very thick] (v3)--(v7);
			 \draw[very thick] (v4)--(v8);
			 \node at (0,2) {$F'_1$};
			 \node at (0,0) {$F'_2$};
			 \node at (-2.2,0) {$F'_3$};
			 \node at (2.2,0) {$F'_4$};
   		    \end{tikzpicture} 
		    \caption{}
		    \label{fig:sec4_cube_marked_facets}
\end{figure}

 Let $F_1$ and $F_2$ be the facets of $P$ such that $F\cap F_1=F'_1$ and $F\cap F_2=F'_2$.  The facets $F_1$ and $F_2$ are cubes, triangular prisms or square pyramids since they must have the square facets $F'_1$ and $F'_2$, respectively. By Proposition~\ref{prop:intersections}, the edge $F'_1\cap F'_2$ is not the intersection of $F$ and some other facet of $P$. Hence $F_1\cap F_2$ is a two dimensional face of $P$, i.e. $F_1\cap F_2$ is a facet of both $F_1$ and $F_2$.  
 
 If there is a cube attached to $F$, we may assume that it is  $F_2$. Then $F_1$ must be either a cube or a triangular prism attached to $F_2$ by a square face.  If there are no cubes attached to $F$, but there are triangular prisms, then again we may assume that $F_2$ is a triangular prism and then all square faces of $F_2$ other than $F'_2$ must be attached to triangular prisms. 
 Finally, if there are no cubes or triangular prisms attached to $F$, then all facets attached to $F$ are square pyramids.
  These possibilities are represented by the diagrams in the upper row of Figure~\ref{fig:sec4_cube_cases} where we have drawn the Schlegel diagrams of $F_1$ and $F_2$ inside their faces $F_1'$ and $F_2'$.
 
\begin{figure}[h!]	
	  \begin{center}
	  	\begin{subfigure}{.23\linewidth}
  		    \begin{tikzpicture}[scale=0.45]
   			 \coordinate (v1) at (-1.5,-1.5);
			 \coordinate (v2) at (1.5,-1.5);
			 \coordinate (v3) at (1.5,1.2);
			 \coordinate (v4) at (-1.5,1.2);
   			 \coordinate (v5) at (-3,-2);
			 \coordinate (v6) at (3,-2);
			 \coordinate (v7) at (3,3);
			 \coordinate (v8) at (-3,3);
			 \draw[very thick] (v1)--(v2)--(v3)--(v4)--(v1);
			 \draw[very thick] (v5)--(v6)--(v7)--(v8)--(v5);
			 \draw[very thick] (v1)--(v5);
			 \draw[very thick] (v2)--(v6);
			 \draw[very thick] (v3)--(v7);
			 \draw[very thick] (v4)--(v8);
			 \coordinate (u1) at (-1,-1);
			 \coordinate (u2) at (1,-1);
			 \coordinate (u3) at (1,.7);
			 \coordinate (u4) at (-1,.7);	
			 \draw[red, very thin] (u1)--(u2)--(u3)--(u4)--(u1); 		 
			 \draw[blue, very thin] (v1)--(u1);
			 \draw[blue, very thin] (v2)--(u2);
			 \draw[blue, very thin] (v3)--(u3);
			 \draw[blue, very thin] (v4)--(u4);
		         \coordinate (w1) at (-1.4,1.7);
			 \coordinate (w2) at (1.4,1.7);
			 \coordinate (w3) at (2,2.5);
			 \coordinate (w4) at (-2,2.5);	
			 \draw[red, very thin] (w1)--(w2)--(w3)--(w4)--(w1); 
			 \draw[blue, very thin] (v4)--(w1);
			 \draw[blue, very thin] (v3)--(w2);
			 \draw[blue, very thin] (v7)--(w3);
			 \draw[blue, very thin] (v8)--(w4);
   		    \end{tikzpicture} 
		    
		    \vspace*{.5cm} 
		     		 
		   \begin{tikzpicture}[yscale=1., xscale=1.8, x={(1cm,0cm)},y={(-.5cm,-.5cm)}, z={(0cm,2cm)}]
			 \coordinate (v1) at (0, 0, 0);
			 \coordinate (v2) at (1, 0, 0);
			 \coordinate (v3) at (1, 1, 0);
			 \coordinate (v4) at (0, 1, 0);
			 \coordinate (v5) at (0, 0, 1);
			 \coordinate (v6) at (1, 0, 1);
			 \coordinate (v7) at (1, 1, 1);
			 \coordinate (v8) at (0, 1, 1);
			 \coordinate (center) at ($(v1)!0.5!(v7)$);
			 \draw[very thick] (v1)--(v2)--(v3)--(v4)--(v1);
			 \draw[very thick] (v5)--(v6)--(v7)--(v8)--(v5);
			 \draw[very thick] (v1)--(v2)--(v6)--(v5)--(v1);
			 \draw[very thick] (v3)--(v4)--(v8)--(v7)--(v3);
			 \foreach \x in {1,...,8} {\coordinate (w\x) at ($(v\x)!0.4!(center)$);
			  \draw[blue] (w\x)--(v\x);}
			 \draw[red] (w1)--(w2)--(w3)--(w4)--(w1);
			 \draw[red] (w5)--(w6)--(w7)--(w8)--(w5);
			 \draw[red] (w1)--(w2)--(w6)--(w5)--(w1);
			 \draw[red] (w3)--(w4)--(w8)--(w7)--(w3);
		   \end{tikzpicture}
		   \subcaption{}
		   \label{fig:sec4_cube_case1}
		\end{subfigure}
		\begin{subfigure}{.23\linewidth}
  		    \begin{tikzpicture}[scale=0.45]
   			 \coordinate (v1) at (-1.5,-1.5);
			 \coordinate (v2) at (1.5,-1.5);
			 \coordinate (v3) at (1.5,1.2);
			 \coordinate (v4) at (-1.5,1.2);
   			 \coordinate (v5) at (-3,-2);
			 \coordinate (v6) at (3,-2);
			 \coordinate (v7) at (3,3);
			 \coordinate (v8) at (-3,3);
			 \draw[very thick] (v1)--(v2)--(v3)--(v4)--(v1);
			 \draw[very thick] (v5)--(v6)--(v7)--(v8)--(v5);
			 \draw[very thick] (v1)--(v5);
			 \draw[very thick] (v2)--(v6);
			 \draw[very thick] (v3)--(v7);
			 \draw[very thick] (v4)--(v8);
			 \coordinate (u1) at (-1,-1);
			 \coordinate (u2) at (1,-1);
			 \coordinate (u3) at (1,.7);
			 \coordinate (u4) at (-1,.7);	
			 \draw[red, very thin] (u1)--(u2)--(u3)--(u4)--(u1); 		 
			 \draw[blue, very thin] (v1)--(u1);
			 \draw[blue, very thin] (v2)--(u2);
			 \draw[blue, very thin] (v3)--(u3);
			 \draw[blue, very thin] (v4)--(u4);
		         \coordinate (w1) at (-1.2,2);
			 \coordinate (w2) at (1.2,2);	
			 \draw[red, very thin] (w1)--(w2); 
			 \draw[blue, very thin] (v4)--(w1);
			 \draw[blue, very thin] (v3)--(w2);
			 \draw[blue, very thin] (v7)--(w2);
			 \draw[blue, very thin] (v8)--(w1);
   		    \end{tikzpicture} 
		    
		    \vspace*{.5cm} 
		     
		   \begin{tikzpicture}[yscale=1., xscale=1.8, x={(1cm,0cm)},y={(-.5cm,-.5cm)}, z={(0cm,2cm)}]
			 \coordinate (v1) at (0, 0, 0);
			 \coordinate (v2) at (1, 0, 0);
			 \coordinate (v3) at (1, 1, 0);
			 \coordinate (v4) at (0, 1, 0);
			 \coordinate (v5) at (0, 0, 1);
			 \coordinate (v6) at (1, 0, 1);
			 \coordinate (v7) at (1, 1, 1);
			 \coordinate (v8) at (0, 1, 1);
			 \coordinate (center) at ($(v1)!0.5!(v7)$);
			 \draw[very thick] (v1)--(v2)--(v3)--(v4)--(v1);
			 \draw[very thick] (v5)--(v6)--(v7)--(v8)--(v5);
			 \draw[very thick] (v1)--(v2)--(v6)--(v5)--(v1);
			 \draw[very thick] (v3)--(v4)--(v8)--(v7)--(v3);
			 \foreach \x/\y in {1/4,3/2,5/8,7/6} {\coordinate (w\x) at ($.15*(v\x)+.15*(v\y)+.7*(center)$);
			  \draw[blue] (w\x)--(v\x);
			  \draw[blue] (w\x)--(v\y);}
			 \draw[red] (w5)--(w7)--(w3)--(w1)--(w5);
		   \end{tikzpicture}
		   \subcaption{}
		   \label{fig:sec4_cube_case2}
		 \end{subfigure}				
		\begin{subfigure}{.23\linewidth}
  		    \begin{tikzpicture}[scale=0.45]
   			 \coordinate (v1) at (-1.5,-1.5);
			 \coordinate (v2) at (1.5,-1.5);
			 \coordinate (v3) at (1.5,1.2);
			 \coordinate (v4) at (-1.5,1.2);
   			 \coordinate (v5) at (-3,-2);
			 \coordinate (v6) at (3,-2);
			 \coordinate (v7) at (3,3);
			 \coordinate (v8) at (-3,3);
			 \draw[very thick] (v1)--(v2)--(v3)--(v4)--(v1);
			 \draw[very thick] (v5)--(v6)--(v7)--(v8)--(v5);
			 \draw[very thick] (v1)--(v5);
			 \draw[very thick] (v2)--(v6);
			 \draw[very thick] (v3)--(v7);
			 \draw[very thick] (v4)--(v8);
			 \coordinate (u1) at (-1,0);
			 \coordinate (u2) at (1,0);
			 \draw[red, very thin] (u1)--(u2); 		 
			 \draw[blue, very thin] (v1)--(u1);
			 \draw[blue, very thin] (v2)--(u2);
			 \draw[blue, very thin] (v3)--(u2);
			 \draw[blue, very thin] (v4)--(u1);
		         \coordinate (w1) at (-1.2,2);
			 \coordinate (w2) at (1.2,2);	
			 \draw[red, very thin] (w1)--(w2); 
			 \draw[blue, very thin] (v4)--(w1);
			 \draw[blue, very thin] (v3)--(w2);
			 \draw[blue, very thin] (v7)--(w2);
			 \draw[blue, very thin] (v8)--(w1);
   		    \end{tikzpicture}		    
		     \vspace*{.5cm} 
		     
		   \begin{tikzpicture}[yscale=1., xscale=1.8, x={(1cm,0cm)},y={(-.5cm,-.5cm)}, z={(0cm,2cm)}]
			 \coordinate (v1) at (0, 0, 0);
			 \coordinate (v2) at (1, 0, 0);
			 \coordinate (v3) at (1, 1, 0);
			 \coordinate (v4) at (0, 1, 0);
			 \coordinate (v5) at (0, 0, 1);
			 \coordinate (v6) at (1, 0, 1);
			 \coordinate (v7) at (1, 1, 1);
			 \coordinate (v8) at (0, 1, 1);
			 \coordinate (center) at ($(v1)!0.5!(v7)$);
			 \draw[very thick] (v1)--(v2)--(v3)--(v4)--(v1);
			 \draw[very thick] (v5)--(v6)--(v7)--(v8)--(v5);
			 \draw[very thick] (v1)--(v2)--(v6)--(v5)--(v1);
			 \draw[very thick] (v3)--(v4)--(v8)--(v7)--(v3);
			 \foreach \x/\y/\s/\t in {1/4/5/8, 3/2/7/6} {\coordinate (w\x) at ($.05*(v\x)+.05*(v\y)+.05*(v\s)+.05*(v\t)+.8*(center)$);
			  \draw[blue] (v\y)--(w\x)--(v\x);
			  \draw[blue] (v\s)--(w\x)--(v\t);}
			 \draw[red] (w1)--(w3);
		   \end{tikzpicture}
		   \subcaption{}
		   \label{fig:sec4_cube_case3}
		 \end{subfigure} 			
		\begin{subfigure}{.23\linewidth}
  		    \begin{tikzpicture}[scale=0.45]
   			 \coordinate (v1) at (-1.5,-1.5);
			 \coordinate (v2) at (1.5,-1.5);
			 \coordinate (v3) at (1.5,1.2);
			 \coordinate (v4) at (-1.5,1.2);
   			 \coordinate (v5) at (-3,-2);
			 \coordinate (v6) at (3,-2);
			 \coordinate (v7) at (3,3);
			 \coordinate (v8) at (-3,3);
			 \draw[very thick] (v1)--(v2)--(v3)--(v4)--(v1);
			 \draw[very thick] (v5)--(v6)--(v7)--(v8)--(v5);
			 \draw[very thick] (v1)--(v5);
			 \draw[very thick] (v2)--(v6);
			 \draw[very thick] (v3)--(v7);
			 \draw[very thick] (v4)--(v8);
			 \coordinate (u1) at (0,0);	 
			 \draw[blue, very thin] (v1)--(u1);
			 \draw[blue, very thin] (v2)--(u1);
			 \draw[blue, very thin] (v3)--(u1);
			 \draw[blue,very thin] (v4)--(u1);
		         \coordinate (w1) at (0,2);	
			 \draw[blue, very thin] (v4)--(w1);
			 \draw[blue, very thin] (v3)--(w1);
			 \draw[blue, very thin] (v7)--(w1);
			 \draw[blue, very thin] (v8)--(w1);
   		    \end{tikzpicture}
		    \vspace*{.5cm}
		    
		   \begin{tikzpicture}[yscale=1., xscale=1.8, x={(1cm,0cm)},y={(-.5cm,-.5cm)}, z={(0cm,2cm)}]
			 \coordinate (v1) at (0, 0, 0);
			 \coordinate (v2) at (1, 0, 0);
			 \coordinate (v3) at (1, 1, 0);
			 \coordinate (v4) at (0, 1, 0);
			 \coordinate (v5) at (0, 0, 1);
			 \coordinate (v6) at (1, 0, 1);
			 \coordinate (v7) at (1, 1, 1);
			 \coordinate (v8) at (0, 1, 1);
			 \coordinate (center) at ($(v1)!0.5!(v7)$);
			 \draw[very thick] (v1)--(v2)--(v3)--(v4)--(v1);
			 \draw[very thick] (v5)--(v6)--(v7)--(v8)--(v5);
			 \draw[very thick] (v1)--(v2)--(v6)--(v5)--(v1);
			 \draw[very thick] (v3)--(v4)--(v8)--(v7)--(v3);
			 \foreach \x in {1,...,8} {
			  \draw[blue] (center)--(v\x);}
		   \end{tikzpicture}
		   \subcaption{} 
		   \label{fig:sec4_cube_case4}
		 \end{subfigure} 	 
		 \caption{}
	 \label{fig:sec4_cube_cases}
	\end{center}
\end{figure}

The analysis of the four cases depicted in the upper row of Figure~\ref{fig:sec4_cube_cases} leads us to unique ways of completing them to potentially psd-minimal polytopes. The corresponding Schlegel diagrams are  shown in the lower row of Figure~\ref{fig:sec4_cube_cases}. We elaborate on 
Figure~\ref{fig:sec4_cube_case1}, when both $F_1$ and $F_2$ are cubes. 

Observe that the cube is the only psd-minimal $3$-polytope with a vertex contained in three square facets. This observation applied to the vertices of $F_1\cap F_2\cap F$, shows that $F_3$ and $F_4$ are cubes, where $F_3$ and $F_4$ are the facets of $P$ such that  $F'_3=F_3\cap F$ and $F'_4=F_4\cap F$ (see Figure~\ref{fig:sec4_cube_marked_facets}). Repeating this argument, we deduce that all facets of $P$ having a nonempty intersection with $F$ are cubes. 

Consider the vertex $v$ of $P$ such that  $v\in (F_1\cap F_2\cap F_3)\smallsetminus F$. There exists a facet $G$ of $P$ different from $F_1$, $F_2$ and $F_3$ such that $v\in G$, because every vertex lies in at least four facets of $P$. 
Since $F_1$, $F_2$ and $F_3$ are cubes and they intersect $G$, by Proposition~\ref{prop:intersections}, 
each of the intersections $G \cap F_i$, $i=1,\ldots,3$ must be a square face of $G$ containing $v$. Therefore, $G$ must be a cube as it is the only psd-minimal $3$-polytope with three square facets meeting at a vertex. 
This implies that the four vertices of $F_i, i=1,\ldots,3$ that are not in $F$ are in $G$. Applying the same argument to different triples of facets of $P$ intersecting $F$ we conclude that they all meet $G$ in a square face, giving rise to the Schlegel diagram in~\ref{fig:sec4_cube_case1}.
 The other cases of Figure~\ref{fig:sec4_cube_cases} are obtained by a similar and simpler application of Proposition~\ref{prop:intersections}.

We conclude that there are four distinct combinatorial classes of psd-minimal polytopes with cubical facets, namely classes $16$, $20$, $28$ and $30$. 
Since none of them are self dual, we obtain eight classes in Table~\ref{table:list}.

\subsection{Triangular Prism} 
Suppose $P$ has a facet $F$ that is a triangular prism. We may assume that $P$ has no cubical facets since otherwise we would be in the previous case. Consider the facets 
$F_1$, $F_2$ and $F_3$ of $P$ such that $F'_1=F_1\cap F$, $F'_2=F_2\cap F$ and $F'_3=F_3\cap F$ are the square faces of $F$(see Figure~\ref{fig:sec4_trprism_marked_facets}).
\begin{figure}[h!]	
	  \begin{center}
  		    \begin{tikzpicture}[scale=0.37]
   			 \coordinate (v1) at (0,4);
			 \coordinate (v2) at (4,-4);
			 \coordinate (v3) at (-4,-4);
			 \coordinate (v4) at ($(v1)!0.75!(0,-1)$);
   			 \coordinate (v5) at ($(v2)!0.75!(0,-1)$);
			 \coordinate (v6) at ($(v3)!0.75!(0,-1)$);
			 \draw[very thick] (v1)--(v2)--(v3)--(v1);
			 \draw[very thick] (v5)--(v6)--(v4)--(v5);
			 \draw (v1)--(v4);
			 \draw (v2)--(v5);
			 \draw (v3)--(v6);
			 \node at (0,-3) {$F'_1$};
			 \node at (-1.4,-.5) {$F'_2$};
			 \node at (1.4,-.5) {$F'_3$};
   		    \end{tikzpicture}
	\caption{}	     
	 \label{fig:sec4_trprism_marked_facets}
	\end{center}
\end{figure}

The facet $F_1$ is a triangular prism or a square pyramid, since its facet $F'_1$ is a square.
Thus, we have three possibilities for 
$F_1$, shown in Figure~\ref{fig:sec4_trprism_lower_facet_embeddings}: two ways in which $F_1$ can be a triangular prism and one way in which $F_1$ can be a square pyramid.

\begin{figure}[h!]	
	  \begin{center}
	  	\begin{subfigure}{.3\linewidth}
  		    \begin{tikzpicture}[scale=0.4]
   			 \coordinate (v1) at (0,4);
			 \coordinate (v2) at (4,-4);
			 \coordinate (v3) at (-4,-4);
			 \coordinate (v4) at ($(v1)!0.75!(0,-1)$);
   			 \coordinate (v5) at ($(v2)!0.75!(0,-1)$);
			 \coordinate (v6) at ($(v3)!0.75!(0,-1)$);
			 \draw[very thick] (v1)--(v2)--(v3)--(v1);
			 \draw[very thick] (v5)--(v6)--(v4)--(v5);
			 \draw (v1)--(v4);
			 \draw (v2)--(v5);
			 \draw (v3)--(v6);
			 \coordinate (w1) at (-1.5,-3);
			 \coordinate (w2) at (1.5,-3);
			 \draw[very thin] (w1)--(v3);
			 \draw[very thin] (w1)--(v6);
			 \draw[very thin] (w2)--(v2);
			 \draw[very thin] (w2)--(v5);	
			 \draw[very thin] (w2)--(w1);		 
   		    \end{tikzpicture}
		   \end{subfigure} 
		   \begin{subfigure}{.3\linewidth}
  		    \begin{tikzpicture}[scale=0.4]
   			 \coordinate (v1) at (0,4);
			 \coordinate (v2) at (4,-4);
			 \coordinate (v3) at (-4,-4);
			 \coordinate (v4) at ($(v1)!0.75!(0,-1)$);
   			 \coordinate (v5) at ($(v2)!0.75!(0,-1)$);
			 \coordinate (v6) at ($(v3)!0.75!(0,-1)$);
			 \draw[very thick] (v1)--(v2)--(v3)--(v1);
			 \draw[very thick] (v5)--(v6)--(v4)--(v5);
			 \draw (v1)--(v4);
			 \draw (v2)--(v5);
			 \draw (v3)--(v6);
			 \coordinate (w1) at (0,-3.5);
			 \coordinate (w2) at (0,-2.4);
			 \draw[very thin] (w1)--(v3);
			 \draw[very thin] (w2)--(v6);
			 \draw[very thin] (w1)--(v2);
			 \draw[very thin] (w2)--(v5);	
			 \draw[very thin] (w2)--(w1);
   		    \end{tikzpicture}
		\end{subfigure}
		\begin{subfigure}{.3\linewidth}
  		    \begin{tikzpicture}[scale=0.4]
   			 \coordinate (v1) at (0,4);
			 \coordinate (v2) at (4,-4);
			 \coordinate (v3) at (-4,-4);
			 \coordinate (v4) at ($(v1)!0.75!(0,-1)$);
   			 \coordinate (v5) at ($(v2)!0.75!(0,-1)$);
			 \coordinate (v6) at ($(v3)!0.75!(0,-1)$);
			 \draw[very thick] (v1)--(v2)--(v3)--(v1);
			 \draw[very thick] (v5)--(v6)--(v4)--(v5);
			 \draw (v1)--(v4);
			 \draw (v2)--(v5);
			 \draw (v3)--(v6);
			 \coordinate (w1) at (0,-2.7);
			 \draw[very thin] (w1)--(v3);
			 \draw[very thin] (w1)--(v6);
			 \draw[very thin] (w1)--(v2);
			 \draw[very thin] (w1)--(v5);
   		    \end{tikzpicture}
		\end{subfigure} 
		\caption{}	     
	 \label{fig:sec4_trprism_lower_facet_embeddings}
	\end{center}
\end{figure}

We start by considering all the possibilities when $F_1$ is a triangular prism. To do this we have to list all possible combinatorial types and positions of $F_1$, $F_2$ and $F_3$, and draw the corresponding diagrams as in the case of the cube. By Lemma~\ref{lem:edge_in_three_four_facets} we know that every edge of $F$ is in at most four facets of $P$, while by Proposition \ref{prop:intersections} we know that certain edges of prisms and pyramids can only appear in three facets of $P$. Therefore, besides combinatorial types, we also consider all the possibilities for edges of $F$ to be contained in exactly three facets of $P$ (in which case we draw the edge thick) or in exactly four facets of $P$ (in which case we draw the edge dashed). The diagrams obtained are shown in Figure~\ref{fig:sec4_triangular_prism_trprism_cases}. 

The rules to obtain these diagrams are simply that if $F_i$ and $F_j$ share a thick edge, their facets that contain this edge, and are not facets of $F$, must be combinatorially the same, as they are identified. If these face identifications force distinct vertices from the same facet to be identified, we  discard the underlying diagram, as it does not correspond to a polytope. The diagrams in the upper rows in Figure~\ref{fig:sec4_triangular_prism_trprism_cases} are all the cases that result from this process. 
We explain in detail how we obtain the diagrams in Figure~\ref{fig:sec4_triangular_prism_trprism_cases} that come from the first diagram in Figure~\ref{fig:sec4_trprism_lower_facet_embeddings}. 

Since $F_2$ and $F_3$ must share a triangular face with $F_1$, there are three cases to consider. Both $F_2$ and $F_3$ can be triangular prisms sharing a triangular face which gives immediately the case in Figure~\ref{fig:sec4_trprism_result1}.
The second possibility is that both $F_2$ and $F_3$ are square pyramids. In this case, either they intersect in a triangular face or an edge. 
If they intersected in a triangular face, then two vertices of $F_1$ would be identified with each other which cannot happen. Therefore, we can obtain only the case in Figure~\ref{fig:sec4_trprism_result2}. Lastly, suppose $F_2$ is a triangular prism and $F_3$ is a square pyramid. Then Proposition~\ref{prop:intersections} implies that each pair of $F_1,F_2, F_3$ intersect in a triangular face which would force the identification of two vertices of $F_1$ and $F_2$ not belonging to their already common triangular face which is a contradiction. The remaining seven cases in Figure~\ref{fig:sec4_triangular_prism_trprism_cases} arise similarly from the second 
case in Figure~\ref{fig:sec4_trprism_lower_facet_embeddings}.

As in the case of cubical facets, it is easy to show that for each diagram there is a unique way to complete it to a potentially psd-minimal combinatorial polytope. Below each case in the upper rows of Figure~\ref{fig:sec4_triangular_prism_trprism_cases} we present the corresponding Schlegel diagram. Next we assume that none of the facets $F_1$, $F_2$ and $F_3$ are triangular prisms. All possible diagrams and the corresponding Schlegel diagrams are shown in Figure~\ref{fig:sec4_triangular_prism_pyramid_cases}.

\begin{figure}	
	  \begin{center}
	  	\begin{subfigure}{.3\linewidth}
  		    \begin{tikzpicture}[scale=0.4]
   			 \coordinate (v1) at (0,4);
			 \coordinate (v2) at (4,-4);
			 \coordinate (v3) at (-4,-4);
			 \coordinate (v4) at ($(v1)!0.75!(0,-1)$);
   			 \coordinate (v5) at ($(v2)!0.75!(0,-1)$);
			 \coordinate (v6) at ($(v3)!0.75!(0,-1)$);
			 \draw[very thick] (v1)--(v2)--(v3)--(v1);
			 \draw[very thick] (v5)--(v6)--(v4)--(v5);
			 \draw[very thick] (v1)--(v4);
			 \draw[very thick] (v2)--(v5);
			 \draw[very thick] (v3)--(v6);
			 \coordinate (w1) at (-1.5,-3);
			 \coordinate (w2) at (1.5,-3);
			 \draw[very thin, blue] (w1)--(v3);
			 \draw[very thin, blue] (w1)--(v6);
			 \draw[very thin, blue] (w2)--(v2);
			 \draw[very thin, blue] (w2)--(v5);	
			 \draw[very thin, red] (w2)--(w1);
			 \coordinate(t1) at (-1,.5);
			 \coordinate(t2) at (-2,-1.5);
			 \draw[very thin, red] (t1)--(t2);	
			 \draw[very thin, blue] (t1)--(v1);
			 \draw[very thin, blue] (t1)--(v4);	
			 \draw[very thin, blue] (t2)--(v3);
			 \draw[very thin, blue] (t2)--(v6);	
			 \coordinate(u1) at (1,.5);
			 \coordinate(u2) at (2,-1.5);
			 \draw[very thin,red] (u1)--(u2);	
			 \draw[very thin, blue] (u1)--(v1);
			 \draw[very thin, blue] (u1)--(v4);	
			 \draw[very thin, blue] (u2)--(v2);
			 \draw[very thin, blue] (u2)--(v5);			 
   		    \end{tikzpicture}
		    
		    \vspace*{.5cm}
		    		    
		    \begin{tikzpicture}[yscale=1, xscale=2, x={(1cm,0cm)},y={(-.5cm,0.5cm)}, z={(0cm,2cm)}]
			 \coordinate (v1) at (0, 1, 0);
			 \coordinate (v2) at (1, 0, 0);
			 \coordinate (v3) at (0, 0, 0);
			 \coordinate (v4) at (0, 1, 1);
   			 \coordinate (v5) at (1, 0, 1);
			 \coordinate (v6) at (0, 0, 1);
			 \draw[very thick] (v1)--(v2)--(v3)--(v1);
			 \draw[very thick] (v4)--(v5)--(v6)--(v4);
			 \draw[very thick] (v1)--(v4);
			 \draw[very thick] (v2)--(v5);
			 \draw[very thick] (v3)--(v6);
			 \coordinate (u1) at ($.5*(v1)+.25*(v2)+.25*(v3)+(0,0,.5)$);
			 \coordinate (u2) at ($.25*(v1)+.5*(v2)+.25*(v3)+(0,0,.5)$);
			 \coordinate (u3) at ($.25*(v1)+.25*(v2)+.5*(v3)+(0,0,.5)$);
			 \draw[red] (u1)--(u2)--(u3)--(u1);
			 \draw[blue] (u1)--(v1);
			 \draw[blue] (u2)--(v2);
			 \draw[blue] (u3)--(v3);
			 \draw[blue] (u1)--(v4);
			 \draw[blue] (u2)--(v5);
			 \draw[blue] (u3)--(v6);
		   \end{tikzpicture}
		   \subcaption{}	 
		   \label{fig:sec4_trprism_result1}	 
		 \end{subfigure}   		   	   
		  \begin{subfigure}{.3\linewidth}
  		    \begin{tikzpicture}[scale=0.4]
   			 \coordinate (v1) at (0,4);
			 \coordinate (v2) at (4,-4);
			 \coordinate (v3) at (-4,-4);
			 \coordinate (v4) at ($(v1)!0.75!(0,-1)$);
   			 \coordinate (v5) at ($(v2)!0.75!(0,-1)$);
			 \coordinate (v6) at ($(v3)!0.75!(0,-1)$);
			 \draw[very thick] (v1)--(v2)--(v3)--(v1);
			 \draw[very thick] (v5)--(v6)--(v4)--(v5);
			 \draw[dashed] (v1)--(v4);
			 \draw[very thick] (v2)--(v5);
			 \draw[very thick] (v3)--(v6);
			 \coordinate (w1) at (-1.5,-3);
			 \coordinate (w2) at (1.5,-3);
			 \draw[very thin, blue] (w1)--(v3);
			 \draw[very thin, blue] (w1)--(v6);
			 \draw[very thin, blue] (w2)--(v2);
			 \draw[very thin, blue] (w2)--(v5);	
			 \draw[very thin, red] (w2)--(w1);
			 \coordinate(u1) at (-1.4,-.5);
			 \draw[very thin, blue] (u1)--(v3);
			 \draw[very thin, blue] (u1)--(v6);
			 \draw[very thin, blue] (u1)--(v1);
			 \draw[very thin, blue] (u1)--(v4);
			 \coordinate (t1) at (1.4,-.5);
			 \draw[very thin, blue] (t1)--(v2);
			 \draw[very thin, blue] (t1)--(v5);
			 \draw[very thin, blue] (t1)--(v1);
			 \draw[very thin, blue] (t1)--(v4);		 
   		    \end{tikzpicture}
		    
		    \vspace*{.5cm}
		    
		    \begin{tikzpicture}[yscale=1, xscale=2, x={(1cm,0cm)},y={(-.5cm,0.5cm)}, z={(0cm,2cm)}]
			 \coordinate (v3) at (0, 1, 0);
			 \coordinate (v2) at (1, 0, 0);
			 \coordinate (v1) at (0, 0, 0);
			 \coordinate (v6) at (0, 1, 1);
   			 \coordinate (v5) at (1, 0, 1);
			 \coordinate (v4) at (0, 0, 1);
			 \draw[very thick] (v1)--(v2);
			 \draw[very thick] (v4)--(v5);
			 \draw[very thick] (v2)--(v3);
			 \draw[very thick] (v5)--(v6);
			 \draw[very thick] (v3)--(v1);
			 \draw[very thick] (v6)--(v4);
			 \draw[very thick] (v1)--(v4);
			 \draw[very thick] (v2)--(v5);
			 \draw[dashed] (v3)--(v6);
			 \coordinate (u1) at ($.5*(v1)+.25*(v2)+.25*(v3)+(0,0,.5)$);
			 \coordinate (u2) at ($.25*(v1)+.5*(v2)+.25*(v3)+(0,0,.5)$);
			 \draw[red]  (u1)--(u2);
			 \draw[blue] (u1)--(v1);
			 \draw[blue] (u2)--(v2);
			 \draw[blue] (u1)--(v4);
			 \draw[blue] (u2)--(v5);
			 \draw[blue] (u1)--(v3);
			 \draw[blue] (u2)--(v3);
			 \draw[blue] (u1)--(v6);
			 \draw[blue] (u2)--(v6);
		   \end{tikzpicture}
		   \subcaption{}
		   \label{fig:sec4_trprism_result2}	
		\end{subfigure}
	  	\begin{subfigure}{.27\linewidth}
  		    \begin{tikzpicture}[scale=0.4]
   			 \coordinate (v1) at (0,4);
			 \coordinate (v2) at (4,-4);
			 \coordinate (v3) at (-4,-4);
			 \coordinate (v4) at ($(v1)!0.75!(0,-1)$);
   			 \coordinate (v5) at ($(v2)!0.75!(0,-1)$);
			 \coordinate (v6) at ($(v3)!0.75!(0,-1)$);
			 \draw[very thick] (v1)--(v2)--(v3)--(v1);
			 \draw[very thick] (v5)--(v6)--(v4)--(v5);
			 \draw[very thick] (v1)--(v4);
			 \draw[very thick] (v2)--(v5);
			 \draw[very thick] (v3)--(v6);
			 \coordinate (w1) at (0,-3.5);
			 \coordinate (w2) at (0,-2.4);
			 \draw[very thin, blue] (w1)--(v3);
			 \draw[very thin, blue] (w2)--(v6);
			 \draw[very thin, blue] (w1)--(v2);
			 \draw[very thin, blue] (w2)--(v5);	
			 \draw[very thin, red] (w2)--(w1);	
			 \coordinate (s1) at ($0.25*(v1)+0.25*(v3)+0.5*(0,-1)$);
   			 \coordinate (s2) at ($0.4*(v1)+0.4*(v3)+0.2*(0,-1)$);
			 \draw[very thin, red] (s1)--(s2);
			 \draw[very thin, blue] (s2)--(v1);
			 \draw[very thin, blue] (s1)--(v4);
			 \draw[very thin, blue] (s2)--(v3);
			 \draw[very thin, blue] (s1)--(v6);
			 \coordinate (t1) at ($0.25*(v1)+0.25*(v2)+0.5*(0,-1)$);
   			 \coordinate (t2) at ($0.4*(v1)+0.4*(v2)+0.2*(0,-1)$);
			 \draw[very thin, red] (t1)--(t2);
			 \draw[very thin, blue] (t2)--(v1);
			 \draw[very thin, blue] (t1)--(v4);
			 \draw[very thin, blue] (t2)--(v2);
			 \draw[very thin, blue] (t1)--(v5);		 
   		    \end{tikzpicture}
		    
		    \vspace*{.5cm}
		    		    
		    \begin{tikzpicture}[yscale=1, xscale=2, x={(1cm,0cm)},y={(-.5cm,0.5cm)}, z={(0cm,2cm)}]
			 \coordinate (v1) at (0, 1, 0);
			 \coordinate (v2) at (1, 0, 0);
			 \coordinate (v3) at (0, 0, 0);
			 \coordinate (v4) at (0, 1, 1);
   			 \coordinate (v5) at (1, 0, 1);
			 \coordinate (v6) at (0, 0, 1);
			 \draw[very thick] (v1)--(v2)--(v3)--(v1);
			 \draw[very thick] (v4)--(v5)--(v6)--(v4);
			 \draw[very thick] (v1)--(v4);
			 \draw[very thick] (v2)--(v5);
			 \draw[very thick] (v3)--(v6);
			 \coordinate (u1) at ($1/3*(v1)+1/3*(v2)+1/3*(v3)+(0,0,.4)$);
			 \coordinate (u2) at ($1/3*(v1)+1/3*(v2)+1/3*(v3)+(0,0,.6)$);
			 \draw[red] (u1)--(u2);
			 \draw[blue] (u1)--(v1);
			 \draw[blue] (u1)--(v2);
			 \draw[blue] (u1)--(v3);
			 \draw[blue] (u2)--(v4);
			 \draw[blue] (u2)--(v5);
			 \draw[blue] (u2)--(v6);
		   \end{tikzpicture}	
		   \subcaption{}
		   \label{fig:sec4_trprism_result3}	
		  \end{subfigure}
 		   
	  	\begin{subfigure}{.3\linewidth}
  		    \begin{tikzpicture}[scale=0.4]
   			 \coordinate (v1) at (0,4);
			 \coordinate (v2) at (4,-4);
			 \coordinate (v3) at (-4,-4);
			 \coordinate (v4) at ($(v1)!0.75!(0,-1)$);
   			 \coordinate (v5) at ($(v2)!0.75!(0,-1)$);
			 \coordinate (v6) at ($(v3)!0.75!(0,-1)$);
			 \draw[very thick] (v1)--(v2)--(v3)--(v1);
			 \draw[very thick] (v5)--(v6)--(v4)--(v5);
			 \draw[very thick] (v1)--(v4);
			 \draw[dashed] (v2)--(v5);
			 \draw[dashed] (v3)--(v6);
			 \coordinate (w1) at (0,-3.5);
			 \coordinate (w2) at (0,-2.4);
			 \draw[very thin, blue] (w1)--(v3);
			 \draw[very thin, blue] (w2)--(v6);
			 \draw[very thin, blue] (w1)--(v2);
			 \draw[very thin, blue] (w2)--(v5);	
			 \draw[very thin, red] (w2)--(w1);	
			 \coordinate (s1) at ($0.25*(v1)+0.25*(v3)+0.5*(0,-1)$);
   			 \coordinate (s2) at ($0.4*(v1)+0.4*(v3)+0.2*(0,-1)$);
			 \draw[very thin, red] (s1)--(s2);
			 \draw[very thin, blue] (s2)--(v1);
			 \draw[very thin, blue] (s1)--(v4);
			 \draw[very thin, blue] (s2)--(v3);
			 \draw[very thin, blue] (s1)--(v6);
			 \coordinate (t1) at ($0.25*(v1)+0.25*(v2)+0.5*(0,-1)$);
   			 \coordinate (t2) at ($0.4*(v1)+0.4*(v2)+0.2*(0,-1)$);
			 \draw[very thin, red] (t1)--(t2);
			 \draw[very thin, blue] (t2)--(v1);
			 \draw[very thin, blue] (t1)--(v4);
			 \draw[very thin, blue] (t2)--(v2);
			 \draw[very thin, blue] (t1)--(v5);				 
   		    \end{tikzpicture}
		    
		    \vspace*{.5cm}
		    
		    \begin{tikzpicture}[yscale=1, xscale=2, x={(1cm,0cm)},y={(-.5cm,0.5cm)}, z={(0cm,2cm)}]
			 \coordinate (v1) at (0, 1, 0);
			 \coordinate (v2) at (1, 0, 0);
			 \coordinate (v3) at (0, 0, 0);
			 \coordinate (v4) at (0, 1, 1);
   			 \coordinate (v5) at (1, 0, 1);
			 \coordinate (v6) at (0, 0, 1);
			 \draw[very thick] (v1)--(v2)--(v3)--(v1);
			 \draw[very thick] (v4)--(v5)--(v6)--(v4);
			 \draw[very thick] (v1)--(v4);
			 \draw[dashed] (v2)--(v5);
			 \draw[dashed] (v3)--(v6);
			 \coordinate (u1) at ($.2*(v1)+.4*(v2)+.4*(v3)+(0,0,.4)$);
			 \coordinate (u2) at ($.2*(v1)+.4*(v2)+.4*(v3)+(0,0,.6)$);
			 \coordinate (t1) at ($.6*(v1)+.2*(v2)+.2*(v3)+(0,0,.4)$);
			 \coordinate (t2) at ($.6*(v1)+.2*(v2)+.2*(v3)+(0,0,.6)$);
			 \draw[red] (u1)--(u2);
			 \draw[red] (t1)--(t2);
			 \draw[blue] (t1)--(v1);
			 \draw[blue] (t1)--(v2);
			 \draw[blue] (t1)--(v3);
			 \draw[blue] (t2)--(v4);
			 \draw[blue] (t2)--(v5);
			 \draw[blue] (t2)--(v6);
			 \draw[blue] (u1)--(v2);
			 \draw[blue] (u1)--(v3);
			 \draw[blue] (u2)--(v5);
			 \draw[blue] (u2)--(v6);
			 \draw[orange] (u2)--(t2);
			 \draw[orange] (u1)--(t1); 
		   \end{tikzpicture}
		   \subcaption{}	
		   \label{fig:sec4_trprism_result4}		    
	\end{subfigure} 		   
	\begin{subfigure}{.3\linewidth}
  		    \begin{tikzpicture}[scale=0.4]
   			 \coordinate (v1) at (0,4);
			 \coordinate (v2) at (4,-4);
			 \coordinate (v3) at (-4,-4);
			 \coordinate (v4) at ($(v1)!0.75!(0,-1)$);
   			 \coordinate (v5) at ($(v2)!0.75!(0,-1)$);
			 \coordinate (v6) at ($(v3)!0.75!(0,-1)$);
			 \draw[very thick] (v1)--(v2)--(v3)--(v1);
			 \draw[very thick] (v5)--(v6)--(v4)--(v5);
			 \draw[dashed] (v1)--(v4);
			 \draw[dashed] (v2)--(v5);
			 \draw[dashed] (v3)--(v6);
			 \coordinate (w1) at (0,-3.5);
			 \coordinate (w2) at (0,-2.4);
			 \draw[very thin, blue] (w1)--(v3);
			 \draw[very thin, blue] (w2)--(v6);
			 \draw[very thin, blue] (w1)--(v2);
			 \draw[very thin, blue] (w2)--(v5);	
			 \draw[very thin, red] (w2)--(w1);	
			 \coordinate (s1) at ($0.25*(v1)+0.25*(v3)+0.5*(0,-1)$);
   			 \coordinate (s2) at ($0.4*(v1)+0.4*(v3)+0.2*(0,-1)$);
			 \draw[very thin, red] (s1)--(s2);
			 \draw[very thin, blue] (s2)--(v1);
			 \draw[very thin, blue] (s1)--(v4);
			 \draw[very thin, blue] (s2)--(v3);
			 \draw[very thin, blue] (s1)--(v6);
			 \coordinate (t1) at ($0.25*(v1)+0.25*(v2)+0.5*(0,-1)$);
   			 \coordinate (t2) at ($0.4*(v1)+0.4*(v2)+0.2*(0,-1)$);
			 \draw[very thin, red] (t1)--(t2);
			 \draw[very thin, blue] (t2)--(v1);
			 \draw[very thin, blue] (t1)--(v4);
			 \draw[very thin, blue] (t2)--(v2);
			 \draw[very thin, blue] (t1)--(v5);				 
   		    \end{tikzpicture}
		    
		    \vspace*{.5cm}
		    
		    \begin{tikzpicture}[yscale=1, xscale=2, x={(1cm,0cm)},y={(-.5cm,0.5cm)}, z={(0cm,2cm)}]
			 \coordinate (v1) at (0, 1, 0);
			 \coordinate (v2) at (1, 0, 0);
			 \coordinate (v3) at (0, 0, 0);
			 \coordinate (v4) at (0, 1, 1);
   			 \coordinate (v5) at (1, 0, 1);
			 \coordinate (v6) at (0, 0, 1);
			 \draw[very thick] (v1)--(v2)--(v3)--(v1);
			 \draw[very thick] (v4)--(v5)--(v6)--(v4);
			 \draw[dashed] (v1)--(v4);
			 \draw[dashed] (v2)--(v5);
			 \draw[dashed] (v3)--(v6);
			 \coordinate (u1) at ($.1*(v1)+.45*(v2)+.45*(v3)+(0,0,.4)$);
			 \coordinate (u2) at ($.1*(v1)+.45*(v2)+.45*(v3)+(0,0,.6)$);
			 \coordinate (t1) at ($.45*(v1)+.1*(v2)+.45*(v3)+(0,0,.4)$);
			 \coordinate (t2) at ($.45*(v1)+.1*(v2)+.45*(v3)+(0,0,.6)$);
			 \coordinate (s1) at ($.45*(v1)+.45*(v2)+.1*(v3)+(0,0,.4)$);
			 \coordinate (s2) at ($.45*(v1)+.45*(v2)+.1*(v3)+(0,0,.6)$); 
			 \draw[red] (u1)--(u2);
			 \draw[red] (t1)--(t2);
			 \draw[red] (s1)--(s2);
			 \draw[blue] (u1)--(v2);
			 \draw[blue] (u1)--(v3);
			 \draw[blue] (u2)--(v5);
			 \draw[blue] (u2)--(v6);
			 \draw[blue] (t1)--(v1);
			 \draw[blue] (t1)--(v3);
			 \draw[blue] (t2)--(v4);
			 \draw[blue] (t2)--(v6);
			 \draw[blue] (s1)--(v2);
			 \draw[blue] (s1)--(v1);
			 \draw[blue] (s2)--(v5);
			 \draw[blue] (s2)--(v4);
			 \draw[orange] (u2)--(t2)--(s2)--(u2);
			 \draw[orange] (u1)--(t1)--(s1)--(u1); 
		   \end{tikzpicture}	
		   \subcaption{}
		   \label{fig:sec4_trprism_result5}	
		\end{subfigure} 
	  	\begin{subfigure}{.27\linewidth}
  		    \begin{tikzpicture}[scale=0.4]
   			 \coordinate (v1) at (0,4);
			 \coordinate (v2) at (4,-4);
			 \coordinate (v3) at (-4,-4);
			 \coordinate (v4) at ($(v1)!0.75!(0,-1)$);
   			 \coordinate (v5) at ($(v2)!0.75!(0,-1)$);
			 \coordinate (v6) at ($(v3)!0.75!(0,-1)$);
			 \draw[very thick] (v1)--(v2)--(v3)--(v1);
			 \draw[very thick] (v5)--(v6)--(v4)--(v5);
			 \draw[dashed] (v1)--(v4);
			 \draw[dashed] (v2)--(v5);
			 \draw[very thick] (v3)--(v6);
			 \coordinate (w1) at (0,-3.5);
			 \coordinate (w2) at (0,-2.4);
			 \draw[very thin, blue] (w1)--(v3);
			 \draw[very thin, blue] (w2)--(v6);
			 \draw[very thin, blue] (w1)--(v2);
			 \draw[very thin, blue] (w2)--(v5);	
			 \draw[very thin, red] (w2)--(w1);	
			 \coordinate (s1) at ($0.25*(v1)+0.25*(v3)+0.5*(0,-1)$);
   			 \coordinate (s2) at ($0.4*(v1)+0.4*(v3)+0.2*(0,-1)$);
			 \draw[very thin, red] (s1)--(s2);
			 \draw[very thin, blue] (s2)--(v1);
			 \draw[very thin, blue] (s1)--(v4);
			 \draw[very thin, blue] (s2)--(v3);
			 \draw[very thin, blue] (s1)--(v6);
			 \coordinate (t1) at (1.4,-.5);
			 \draw[very thin, blue] (t1)--(v2);
			 \draw[very thin, blue] (t1)--(v5);
			 \draw[very thin, blue] (t1)--(v1);
			 \draw[very thin, blue] (t1)--(v4);		 
   		    \end{tikzpicture}
		    
		    \vspace*{.5cm}
		    		    
		    \begin{tikzpicture}[yscale=1, xscale=2, x={(1cm,0cm)},y={(-.5cm,0.5cm)}, z={(0cm,2cm)}]
			 \coordinate (v1) at (0, 1, 0);
			 \coordinate (v2) at (1, 0, 0);
			 \coordinate (v3) at (0, 0, 0);
			 \coordinate (v4) at (0, 1, 1);
   			 \coordinate (v5) at (1, 0, 1);
			 \coordinate (v6) at (0, 0, 1);
			 \draw[very thick] (v1)--(v2)--(v3)--(v1);
			 \draw[very thick] (v4)--(v5)--(v6)--(v4);
			 \draw[dashed] (v1)--(v4);
			 \draw[dashed] (v2)--(v5);
			 \draw[very thick] (v3)--(v6);
			 \coordinate (u1) at ($.35*(v1)+.3*(v2)+.35*(v3)+(0,0,.4)$);
			 \coordinate (u2) at ($.35*(v1)+.3*(v2)+.35*(v3)+(0,0,.6)$);
			 \coordinate (t) at ($.2*(v1)+.7*(v2)+.1*(v3)+(0,0,.5)$);
			 \draw[red] (u1)--(u2);
			 \draw[blue] (u1)--(v1);
			 \draw[blue] (u1)--(v2);
			 \draw[blue] (u1)--(v3);
			 \draw[blue] (u2)--(v4);
			 \draw[blue] (u2)--(v5);
			 \draw[blue] (u2)--(v6);
			 \draw[blue] (t)--(v1);
			 \draw[blue] (t)--(v4);
			 \draw[blue] (t)--(v2);
			 \draw[blue] (t)--(v5);
			 \draw[orange] (u1)--(t)--(u2);
		   \end{tikzpicture}	
		   \subcaption{}
		   \label{fig:sec4_trprism_result6}	
		  \end{subfigure}   		   
	  	\begin{subfigure}{.3\linewidth}
  		    \begin{tikzpicture}[scale=0.4]
   			 \coordinate (v1) at (0,4);
			 \coordinate (v2) at (4,-4);
			 \coordinate (v3) at (-4,-4);
			 \coordinate (v4) at ($(v1)!0.75!(0,-1)$);
   			 \coordinate (v5) at ($(v2)!0.75!(0,-1)$);
			 \coordinate (v6) at ($(v3)!0.75!(0,-1)$);
			 \draw[very thick] (v1)--(v2)--(v3)--(v1);
			 \draw[very thick] (v5)--(v6)--(v4)--(v5);
			 \draw[dashed] (v1)--(v4);
			 \draw[dashed] (v2)--(v5);
			 \draw[dashed] (v3)--(v6);
			 \coordinate (w1) at (0,-3.5);
			 \coordinate (w2) at (0,-2.4);
			 \draw[very thin, blue] (w1)--(v3);
			 \draw[very thin, blue] (w2)--(v6);
			 \draw[very thin, blue] (w1)--(v2);
			 \draw[very thin, blue] (w2)--(v5);	
			 \draw[very thin, red] (w2)--(w1);	
			 \coordinate (s1) at ($0.25*(v1)+0.25*(v3)+0.5*(0,-1)$);
   			 \coordinate (s2) at ($0.4*(v1)+0.4*(v3)+0.2*(0,-1)$);
			 \draw[very thin, red] (s1)--(s2);
			 \draw[very thin, blue] (s2)--(v1);
			 \draw[very thin, blue] (s1)--(v4);
			 \draw[very thin, blue] (s2)--(v3);
			 \draw[very thin, blue] (s1)--(v6);
			 \coordinate (t1) at (1.4,-.5);
			 \draw[very thin, blue] (t1)--(v2);
			 \draw[very thin, blue] (t1)--(v5);
			 \draw[very thin, blue] (t1)--(v1);
			 \draw[very thin, blue] (t1)--(v4);		 
   		    \end{tikzpicture}
		    
		    \vspace*{.5cm}
		    
		    \begin{tikzpicture}[yscale=1, xscale=2, x={(1cm,0cm)},y={(-.5cm,0.5cm)}, z={(0cm,2cm)}]
			 \coordinate (v1) at (0, 1, 0);
			 \coordinate (v2) at (1, 0, 0);
			 \coordinate (v3) at (0, 0, 0);
			 \coordinate (v4) at (0, 1, 1);
   			 \coordinate (v5) at (1, 0, 1);
			 \coordinate (v6) at (0, 0, 1);
			 \draw[very thick] (v1)--(v2)--(v3)--(v1);
			 \draw[very thick] (v4)--(v5)--(v6)--(v4);
			 \draw[dashed] (v1)--(v4);
			 \draw[dashed] (v2)--(v5);
			 \draw[dashed] (v3)--(v6);
			 \coordinate (u1) at ($.7*(v1)+.25*(v2)+.25*(v3)+(0,0,.4)$);
			 \coordinate (u2) at ($.7*(v1)+.25*(v2)+.25*(v3)+(0,0,.6)$);
			 \coordinate (t1) at ($.3*(v1)+.55*(v2)+.15*(v3)+(0,0,.4)$);
			 \coordinate (t2) at ($.3*(v1)+.55*(v2)+.15*(v3)+(0,0,.6)$);
			 \coordinate (t) at ($.1*(v1)+.8*(v2)+.1*(v3)+(0,0,.5)$); 
			 \draw[red] (u1)--(u2);
			 \draw[red] (t1)--(t2); 
			 \draw[blue] (u1)--(v1);
			 \draw[blue] (u1)--(v3);
			 \draw[blue] (u2)--(v4);
			 \draw[blue] (u2)--(v6);
			 \draw[blue] (t1)--(v2);
			 \draw[blue] (t1)--(v3);
			 \draw[blue] (t2)--(v5);
			 \draw[blue] (t2)--(v6);
			 \draw[orange] (u1)--(t1);
			 \draw[orange] (u2)--(t2);
			 \draw[orange] (u1)--(t)--(t1);
			 \draw[orange] (u2)--(t)--(t2);
			 \draw[blue] (t)--(v1);
			 \draw[blue] (t)--(v4);
			 \draw[blue] (t)--(v2);
			 \draw[blue] (t)--(v5); 
		   \end{tikzpicture}	
		   \subcaption{}
		   \label{fig:sec4_trprism_result7}		    
		   \end{subfigure} 		   
		   \begin{subfigure}{.3\linewidth}
  		    \begin{tikzpicture}[scale=0.4]
   			 \coordinate (v1) at (0,4);
			 \coordinate (v2) at (4,-4);
			 \coordinate (v3) at (-4,-4);
			 \coordinate (v4) at ($(v1)!0.75!(0,-1)$);
   			 \coordinate (v5) at ($(v2)!0.75!(0,-1)$);
			 \coordinate (v6) at ($(v3)!0.75!(0,-1)$);
			 \draw[very thick] (v1)--(v2)--(v3)--(v1);
			 \draw[very thick] (v5)--(v6)--(v4)--(v5);
			 \draw[very thick] (v1)--(v4);
			 \draw[dashed] (v2)--(v5);
			 \draw[dashed] (v3)--(v6);
			 \coordinate (w1) at (0,-3.5);
			 \coordinate (w2) at (0,-2.4);
			 \draw[very thin, blue] (w1)--(v3);
			 \draw[very thin, blue] (w2)--(v6);
			 \draw[very thin, blue] (w1)--(v2);
			 \draw[very thin, blue] (w2)--(v5);	
			 \draw[very thin, red] (w2)--(w1);	
			 \coordinate(u1) at (-1.4,-.5);
			 \draw[very thin, blue] (u1)--(v3);
			 \draw[very thin, blue] (u1)--(v6);
			 \draw[very thin, blue] (u1)--(v1);
			 \draw[very thin, blue] (u1)--(v4);
			 \coordinate (t1) at (1.4,-.5);
			 \draw[very thin, blue] (t1)--(v2);
			 \draw[very thin, blue] (t1)--(v5);
			 \draw[very thin, blue] (t1)--(v1);
			 \draw[very thin, blue] (t1)--(v4);	
   		    \end{tikzpicture}
		    
		    \vspace*{.5cm}
		    
		    \begin{tikzpicture}[yscale=1, xscale=2, x={(1cm,0cm)},y={(-.5cm,0.5cm)}, z={(0cm,2cm)}]
			 \coordinate (v1) at (0, 1, 0);
			 \coordinate (v2) at (1, 0, 0);
			 \coordinate (v3) at (0, 0, 0);
			 \coordinate (v4) at (0, 1, 1);
   			 \coordinate (v5) at (1, 0, 1);
			 \coordinate (v6) at (0, 0, 1);
			 \draw[very thick] (v1)--(v2)--(v3)--(v1);
			 \draw[very thick] (v4)--(v5)--(v6)--(v4);
			 \draw[very thick] (v1)--(v4);
			 \draw[dashed] (v2)--(v5);
			 \draw[dashed] (v3)--(v6);
			 \coordinate (u1) at ($.2*(v1)+.7*(v2)+.1*(v3)+(0,0,.4)$);
			 \coordinate (u2) at ($.2*(v1)+.7*(v2)+.1*(v3)+(0,0,.6)$);
			 \coordinate (t) at ($.35*(v1)+.3*(v2)+.35*(v3)+(0,0,.5)$);
			 \draw[red] (u1)--(u2);
			 \draw[blue] (u1)--(v2);
			 \draw[blue] (u1)--(v3);
			 \draw[blue] (u2)--(v5);
			 \draw[blue] (u2)--(v6);
			 \draw[blue] (t)--(v1);
			 \draw[blue] (t)--(v2);
			 \draw[blue] (t)--(v3);
			 \draw[blue] (t)--(v4);
			 \draw[blue] (t)--(v5);
			 \draw[blue] (t)--(v6);
			 \draw[orange] (u1)--(t)--(u2);
		   \end{tikzpicture}
		   \subcaption{}
		   \label{fig:sec4_trprism_result8}	
		 \end{subfigure} 		   
		 \begin{subfigure}{.3\linewidth}
  		    \begin{tikzpicture}[scale=0.4]
   			 \coordinate (v1) at (0,4);
			 \coordinate (v2) at (4,-4);
			 \coordinate (v3) at (-4,-4);
			 \coordinate (v4) at ($(v1)!0.75!(0,-1)$);
   			 \coordinate (v5) at ($(v2)!0.75!(0,-1)$);
			 \coordinate (v6) at ($(v3)!0.75!(0,-1)$);
			 \draw[very thick] (v1)--(v2)--(v3)--(v1);
			 \draw[very thick] (v5)--(v6)--(v4)--(v5);
			 \draw[dashed] (v1)--(v4);
			 \draw[dashed] (v2)--(v5);
			 \draw[dashed] (v3)--(v6);
			 \coordinate (w1) at (0,-3.5);
			 \coordinate (w2) at (0,-2.4);
			 \draw[very thin, blue] (w1)--(v3);
			 \draw[very thin, blue] (w2)--(v6);
			 \draw[very thin, blue] (w1)--(v2);
			 \draw[very thin, blue] (w2)--(v5);	
			 \draw[very thin, red] (w2)--(w1);	
			 \coordinate(u1) at (-1.4,-.5);
			 \draw[very thin, blue] (u1)--(v3);
			 \draw[very thin, blue] (u1)--(v6);
			 \draw[very thin, blue] (u1)--(v1);
			 \draw[very thin, blue] (u1)--(v4);
			 \coordinate (t1) at (1.4,-.5);
			 \draw[very thin, blue] (t1)--(v2);
			 \draw[very thin, blue] (t1)--(v5);
			 \draw[very thin, blue] (t1)--(v1);
			 \draw[very thin, blue] (t1)--(v4);	
   		    \end{tikzpicture}
		    
		    \vspace*{.5cm}
		    
		    \begin{tikzpicture}[yscale=1, xscale=2, x={(1cm,0cm)},y={(-.5cm,0.5cm)}, z={(0cm,2cm)}]
			 \coordinate (v1) at (0, 1, 0);
			 \coordinate (v2) at (1, 0, 0);
			 \coordinate (v3) at (0, 0, 0);
			 \coordinate (v4) at (0, 1, 1);
   			 \coordinate (v5) at (1, 0, 1);
			 \coordinate (v6) at (0, 0, 1);
			 \draw[very thick] (v1)--(v2)--(v3)--(v1);
			 \draw[very thick] (v4)--(v5)--(v6)--(v4);
			 \draw[dashed] (v1)--(v4);
			 \draw[dashed] (v2)--(v5);
			 \draw[dashed] (v3)--(v6);
			 \coordinate (u) at ($.7*(v1)+.25*(v2)+.25*(v3)+(0,0,.5)$);
			 \coordinate (t1) at ($.3*(v1)+.55*(v2)+.15*(v3)+(0,0,.4)$);
			 \coordinate (t2) at ($.3*(v1)+.55*(v2)+.15*(v3)+(0,0,.6)$);
			 \coordinate (t) at ($.1*(v1)+.8*(v2)+.1*(v3)+(0,0,.5)$); 
			 \draw[red] (t1)--(t2); 
			 \draw[blue] (u)--(v1);
			 \draw[blue] (u)--(v3);
			 \draw[blue] (u)--(v4);
			 \draw[blue] (u)--(v6);
			 \draw[blue] (t1)--(v2);
			 \draw[blue] (t1)--(v3);
			 \draw[blue] (t2)--(v5);
			 \draw[blue] (t2)--(v6);
			 \draw[orange] (u)--(t1);
			 \draw[orange] (u)--(t2);
			 \draw[orange] (u)--(t)--(t1);
			 \draw[orange] (u)--(t)--(t2);
			 \draw[blue] (t)--(v1);
			 \draw[blue] (t)--(v4);
			 \draw[blue] (t)--(v2);
			 \draw[blue] (t)--(v5); 
		   \end{tikzpicture}
		   \subcaption{}
		   \label{fig:sec4_trprism_result9}	
		   \end{subfigure} 
	\end{center}
	\caption{}
	\label{fig:sec4_triangular_prism_trprism_cases}
\end{figure}

\begin{figure}[h!]	
	  \begin{center}  		   
		   \begin{subfigure}{.3\linewidth}
  		    \begin{tikzpicture}[scale=0.4]
   			 \coordinate (v1) at (0,4);
			 \coordinate (v2) at (4,-4);
			 \coordinate (v3) at (-4,-4);
			 \coordinate (v4) at ($(v1)!0.75!(0,-1)$);
   			 \coordinate (v5) at ($(v2)!0.75!(0,-1)$);
			 \coordinate (v6) at ($(v3)!0.75!(0,-1)$);
			 \draw[very thick] (v1)--(v2)--(v3)--(v1);
			 \draw[very thick] (v5)--(v6)--(v4)--(v5);
			 \draw[very thick] (v1)--(v4);
			 \draw[very thick] (v2)--(v5);
			 \draw[very thick] (v3)--(v6);
			 \coordinate (w1) at (0,-2.7);
			 \draw[very thin, blue] (w1)--(v3);
			 \draw[very thin, blue] (w1)--(v6);
			 \draw[very thin, blue] (w1)--(v2);
			 \draw[very thin, blue] (w1)--(v5);	
			 \coordinate(u1) at (-1.4,-.5);
			 \draw[very thin, blue] (u1)--(v3);
			 \draw[very thin, blue] (u1)--(v6);
			 \draw[very thin, blue] (u1)--(v1);
			 \draw[very thin, blue] (u1)--(v4);
			 \coordinate (t1) at (1.4,-.5);
			 \draw[very thin, blue] (t1)--(v2);
			 \draw[very thin, blue] (t1)--(v5);
			 \draw[very thin, blue] (t1)--(v1);
			 \draw[very thin, blue] (t1)--(v4);	
   		    \end{tikzpicture}
		    
		    \vspace*{.5cm}
		    
		    \begin{tikzpicture}[yscale=1, xscale=2, x={(1cm,0cm)},y={(-.5cm,0.5cm)}, z={(0cm,2cm)}]
			 \coordinate (v1) at (0, 1, 0);
			 \coordinate (v2) at (1, 0, 0);
			 \coordinate (v3) at (0, 0, 0);
			 \coordinate (v4) at (0, 1, 1);
   			 \coordinate (v5) at (1, 0, 1);
			 \coordinate (v6) at (0, 0, 1);
			 \draw[very thick] (v1)--(v2)--(v3)--(v1);
			 \draw[very thick] (v4)--(v5)--(v6)--(v4);
			 \draw[very thick] (v1)--(v4);
			 \draw[very thick] (v2)--(v5);
			 \draw[very thick] (v3)--(v6);
			 \coordinate (t) at ($.35*(v1)+.3*(v2)+.35*(v3)+(0,0,.5)$);
			 \draw[blue] (t)--(v1);
			 \draw[blue] (t)--(v2);
			 \draw[blue] (t)--(v3);
			 \draw[blue] (t)--(v4);
			 \draw[blue] (t)--(v5);
			 \draw[blue] (t)--(v6);
		   \end{tikzpicture}
		   \subcaption{}
		   \label{fig:sec4_trprism_result10}	
		   \end{subfigure} 
		   \begin{subfigure}{.3\linewidth}
  		    \begin{tikzpicture}[scale=0.4]
   			 \coordinate (v1) at (0,4);
			 \coordinate (v2) at (4,-4);
			 \coordinate (v3) at (-4,-4);
			 \coordinate (v4) at ($(v1)!0.75!(0,-1)$);
   			 \coordinate (v5) at ($(v2)!0.75!(0,-1)$);
			 \coordinate (v6) at ($(v3)!0.75!(0,-1)$);
			 \draw[very thick] (v1)--(v2)--(v3)--(v1);
			 \draw[very thick] (v5)--(v6)--(v4)--(v5);
			 \draw[very thick] (v1)--(v4);
			 \draw[dashed] (v2)--(v5);
			 \draw[dashed] (v3)--(v6);
			 \coordinate (w1) at (0,-2.7);
			 \draw[very thin, blue] (w1)--(v3);
			 \draw[very thin, blue] (w1)--(v6);
			 \draw[very thin, blue] (w1)--(v2);
			 \draw[very thin, blue] (w1)--(v5);	
			 \coordinate(u1) at (-1.4,-.5);
			 \draw[very thin, blue] (u1)--(v3);
			 \draw[very thin, blue] (u1)--(v6);
			 \draw[very thin, blue] (u1)--(v1);
			 \draw[very thin, blue] (u1)--(v4);
			 \coordinate (t1) at (1.4,-.5);
			 \draw[very thin, blue] (t1)--(v2);
			 \draw[very thin, blue] (t1)--(v5);
			 \draw[very thin, blue] (t1)--(v1);
			 \draw[very thin, blue] (t1)--(v4);	
   		    \end{tikzpicture}
		    
		    \vspace*{.5cm}
		    
		    \begin{tikzpicture}[yscale=1, xscale=2, x={(1cm,0cm)},y={(-.5cm,0.5cm)}, z={(0cm,2cm)}]
			 \coordinate (v1) at (0, 1, 0);
			 \coordinate (v2) at (1, 0, 0);
			 \coordinate (v3) at (0, 0, 0);
			 \coordinate (v4) at (0, 1, 1);
   			 \coordinate (v5) at (1, 0, 1);
			 \coordinate (v6) at (0, 0, 1);
			 \draw[very thick] (v1)--(v2)--(v3)--(v1);
			 \draw[very thick] (v4)--(v5)--(v6)--(v4);
			 \draw[very thick] (v1)--(v4);
			 \draw[dashed] (v2)--(v5);
			 \draw[dashed] (v3)--(v6);
			 \coordinate (u) at ($.1*(v1)+.45*(v2)+.45*(v3)+(0,0,.5)$);
			 \coordinate (t) at ($.45*(v1)+.1*(v2)+.45*(v3)+(0,0,.5)$);
			 \draw[blue] (t)--(v1);
			 \draw[blue] (t)--(v2);
			 \draw[blue] (t)--(v3);
			 \draw[blue] (t)--(v4);
			 \draw[blue] (t)--(v5);
			 \draw[blue] (t)--(v6);
			 \draw[blue] (u)--(v2);
			 \draw[blue] (u)--(v3);
			 \draw[blue] (u)--(v5);
			 \draw[blue] (u)--(v6);
			 \draw[orange] (t)--(u);
		   \end{tikzpicture}
		   \subcaption{}
		   \label{fig:sec4_trprism_result11}	
		   \end{subfigure}
		   \begin{subfigure}{.3\linewidth}
  		    \begin{tikzpicture}[scale=0.4]
   			 \coordinate (v1) at (0,4);
			 \coordinate (v2) at (4,-4);
			 \coordinate (v3) at (-4,-4);
			 \coordinate (v4) at ($(v1)!0.75!(0,-1)$);
   			 \coordinate (v5) at ($(v2)!0.75!(0,-1)$);
			 \coordinate (v6) at ($(v3)!0.75!(0,-1)$);
			 \draw[very thick] (v1)--(v2)--(v3)--(v1);
			 \draw[very thick] (v5)--(v6)--(v4)--(v5);
			 \draw[dashed] (v1)--(v4);
			 \draw[dashed] (v2)--(v5);
			 \draw[dashed] (v3)--(v6);
			 \coordinate (w1) at (0,-2.7);
			 \draw[very thin, blue] (w1)--(v3);
			 \draw[very thin, blue] (w1)--(v6);
			 \draw[very thin, blue] (w1)--(v2);
			 \draw[very thin, blue] (w1)--(v5);	
			 \coordinate(u1) at (-1.4,-.5);
			 \draw[very thin, blue] (u1)--(v3);
			 \draw[very thin, blue] (u1)--(v6);
			 \draw[very thin, blue] (u1)--(v1);
			 \draw[very thin, blue] (u1)--(v4);
			 \coordinate (t1) at (1.4,-.5);
			 \draw[very thin, blue] (t1)--(v2);
			 \draw[very thin, blue] (t1)--(v5);
			 \draw[very thin, blue] (t1)--(v1);
			 \draw[very thin, blue] (t1)--(v4);	
   		    \end{tikzpicture}
		    
		    \vspace*{.5cm}
		    
		    \begin{tikzpicture}[yscale=1, xscale=2, x={(1cm,0cm)},y={(-.5cm,0.5cm)}, z={(0cm,2cm)}]
			 \coordinate (v1) at (0, 1, 0);
			 \coordinate (v2) at (1, 0, 0);
			 \coordinate (v3) at (0, 0, 0);
			 \coordinate (v4) at (0, 1, 1);
   			 \coordinate (v5) at (1, 0, 1);
			 \coordinate (v6) at (0, 0, 1);
			 \draw[very thick] (v1)--(v2)--(v3)--(v1);
			 \draw[very thick] (v4)--(v5)--(v6)--(v4);
			 \draw[dashed] (v1)--(v4);
			 \draw[dashed] (v2)--(v5);
			 \draw[dashed] (v3)--(v6);
			 \coordinate (u) at ($.1*(v1)+.45*(v2)+.45*(v3)+(0,0,.5)$);
			 \coordinate (t) at ($.45*(v1)+.1*(v2)+.45*(v3)+(0,0,.5)$); 
			 \coordinate (s) at ($.45*(v1)+.45*(v2)+.1*(v3)+(0,0,.5)$);
			 \draw[blue] (u)--(v2);
			 \draw[blue] (u)--(v3);
			 \draw[blue] (u)--(v5);
			 \draw[blue] (u)--(v6);
			 \draw[blue] (t)--(v1);
			 \draw[blue] (t)--(v3);
			 \draw[blue] (t)--(v4);
			 \draw[blue] (t)--(v6);
			 \draw[blue] (s)--(v1);
			 \draw[blue] (s)--(v2);
			 \draw[blue] (s)--(v4);
			 \draw[blue] (s)--(v5);
			 \draw[orange] (s)--(t)--(u)--(s);
		   \end{tikzpicture}
		   \subcaption{}
		   \label{fig:sec4_trprism_result12}	
		   \end{subfigure}

		\caption{}	
		\label{fig:sec4_triangular_prism_pyramid_cases}     
	\end{center}
\end{figure}

Further analysis allows us to rule out some of these diagrams. In Figures~\ref{fig:sec4_trprism_result7}  and \ref{fig:sec4_trprism_result9}  two octahedra intersect in a vertex and in an edge respectively, contradicting Proposition \ref{prop:intersections}.
We can exclude the case in Figure~\ref{fig:sec4_trprism_result11} as well because there two bisimplices intersect in an edge. What remains are $9$ distinct combinatorial classes, namely classes $4$, $6$, $8$, $10$, $12$, $14$, $18$, $23$ and $27$ in Table~\ref{table:list}. Along with their duals we get $18$ new classes.

\subsection{Octahedron}
We may now assume that no facet of $P$ or its dual is a 
cube or a triangular prism. Let $F$ be an octahedral facet of $P$ and let $F'_1$ be one of its facets (see Figure~\ref{fig:sec4_case_octahedron_marked}). Take $F_1$ to be the facet of $P$ such that $F'_1=F_1\cap F$. 

\begin{figure}[h!]	
	  \begin{center}
  		    \begin{tikzpicture}[scale=0.5]
   			 \coordinate (v1) at (0,4);
			 \coordinate (v2) at (4,-4);
			 \coordinate (v3) at (-4,-4);
			 \coordinate (v4) at ($-.15*(v1)+1.15*(0,-1)-(0,.8)$);
   			 \coordinate (v5) at ($-.25*(v2)+1.25*(0,-1)-(0,.5)$);
			 \coordinate (v6) at ($-.25*(v3)+1.25*(0,-1)-(0,.5)$);
			 \draw[very thick] (v1)--(v2)--(v3)--(v1);
			 \draw[very thick] (v5)--(v6)--(v4)--(v5);
			 \draw[very thick] (v1)--(v5);
			 \draw[very thick] (v1)--(v6);
			 \draw[very thick] (v2)--(v4);
			 \draw[very thick] (v2)--(v6);
			 \draw[very thick] (v3)--(v4);
			 \draw[very thick] (v3)--(v5);
			 \node at (0,-1.4) {$F'_1$};
   		    \end{tikzpicture}
	\caption{}
	\label{fig:sec4_case_octahedron_marked}
	\end{center}
\end{figure}

Every vertex of $F'_1$ is contained in three or four facets of $F_1$, because $F_1$ is a psd-minimal $3$-polytope. We use this fact for the classification and introduce a new type of diagram:  the pair consisting of $F'_1$ and its vertex $v$ is marked by a blue angle when $v$ is contained in three facets of $F_1$ and by a red cut angle if $v$ is contained in four facets of $F_1$ (see Figure~\ref{fig:sec4_case_octahedron_triangle_labelings} for all possible markings of $F'_1$). Moreover, given the labeling of $F'_1$, we can now uniquely determine the combinatorial type of $F_1$, since the triangular faces of the four remaining psd-minimal $3$-polytopes available to use as facets all have different vertex degree distributions.

\begin{figure}[h!]	
	  \begin{center}
	   \begin{subfigure}{0.23\linewidth}
  		    \begin{tikzpicture}[scale=0.3]
   			 \coordinate (v1) at (0,4);
			 \coordinate (v2) at (4,-4);
			 \coordinate (v3) at (-4,-4);
			 \draw[very thick] (v1)--(v2)--(v3)--(v1);
			 
			 \coordinate (v1_1) at ($.7*(v1)+.05*(v2)+.25*(v3)$);
			 \coordinate (v1_2) at ($.8*(v1)+.05*(v2)+.15*(v3)$);
			 \coordinate (v1_3) at ($.8*(v1)+.15*(v2)+.05*(v3)$);
			 \coordinate (v1_4) at ($.7*(v1)+.25*(v2)+.05*(v3)$);
		 	 \coordinate (w1) at (intersection of v1_1--v1_2 and v1_3--v1_4);		 
			 \draw[blue,thick] (v1_1) --(w1) --(v1_4) ;

			 \coordinate (v2_1) at ($.7*(v2)+.05*(v1)+.25*(v3)$);
			 \coordinate (v2_2) at ($.8*(v2)+.05*(v1)+.15*(v3)$);
			 \coordinate (v2_3) at ($.8*(v2)+.15*(v1)+.05*(v3)$);
			 \coordinate (v2_4) at ($.7*(v2)+.25*(v1)+.05*(v3)$);
		 	 \coordinate (w2) at (intersection of v2_1--v2_2 and v2_3--v2_4);		 
			 \draw[blue,thick] (v2_1) --(w2) --(v2_4) ;
			 
			 \coordinate (v3_1) at ($.7*(v3)+.05*(v2)+.25*(v1)$);
			 \coordinate (v3_2) at ($.8*(v3)+.05*(v2)+.15*(v1)$);
			 \coordinate (v3_3) at ($.8*(v3)+.15*(v2)+.05*(v1)$);
			 \coordinate (v3_4) at ($.7*(v3)+.25*(v2)+.05*(v1)$);
		 	 \coordinate (w3) at (intersection of v3_1--v3_2 and v3_3--v3_4);		 
			 \draw[blue,thick] (v3_1) --(w3) --(v3_4) ;		 			 	
   		    \end{tikzpicture}
		    \subcaption{Simplex}
	\end{subfigure}
	   \begin{subfigure}{0.23\linewidth}
  		    \begin{tikzpicture}[scale=0.3]
   			 \coordinate (v1) at (0,4);
			 \coordinate (v2) at (4,-4);
			 \coordinate (v3) at (-4,-4);
			 \draw[very thick] (v1)--(v2)--(v3)--(v1);
			 
			 \coordinate (v1_1) at ($.7*(v1)+.05*(v2)+.25*(v3)$);
			 \coordinate (v1_2) at ($.8*(v1)+.05*(v2)+.15*(v3)$);
			 \coordinate (v1_3) at ($.8*(v1)+.15*(v2)+.05*(v3)$);
			 \coordinate (v1_4) at ($.7*(v1)+.25*(v2)+.05*(v3)$);
		 	 \coordinate (w1) at (intersection of v1_1--v1_2 and v1_3--v1_4);		 
			 \draw[red,thick] (v1_1)--(v1_2) --(v1_3)  --(v1_4) ;

			 \coordinate (v2_1) at ($.7*(v2)+.05*(v1)+.25*(v3)$);
			 \coordinate (v2_2) at ($.8*(v2)+.05*(v1)+.15*(v3)$);
			 \coordinate (v2_3) at ($.8*(v2)+.15*(v1)+.05*(v3)$);
			 \coordinate (v2_4) at ($.7*(v2)+.25*(v1)+.05*(v3)$);
		 	 \coordinate (w2) at (intersection of v2_1--v2_2 and v2_3--v2_4);		 
			 \draw[blue,thick] (v2_1) --(w2) --(v2_4) ;
			 
			 \coordinate (v3_1) at ($.7*(v3)+.05*(v2)+.25*(v1)$);
			 \coordinate (v3_2) at ($.8*(v3)+.05*(v2)+.15*(v1)$);
			 \coordinate (v3_3) at ($.8*(v3)+.15*(v2)+.05*(v1)$);
			 \coordinate (v3_4) at ($.7*(v3)+.25*(v2)+.05*(v1)$);
		 	 \coordinate (w3) at (intersection of v3_1--v3_2 and v3_3--v3_4);		 
			 \draw[blue,thick] (v3_1) --(w3) --(v3_4) ;		 			 	
   		    \end{tikzpicture}
		    \subcaption{ Square pyramid}
	\end{subfigure}
	   \begin{subfigure}{0.23\linewidth}
  		    \begin{tikzpicture}[scale=0.3]
   			 \coordinate (v1) at (0,4);
			 \coordinate (v2) at (4,-4);
			 \coordinate (v3) at (-4,-4);
			 \draw[very thick] (v1)--(v2)--(v3)--(v1);
			 
			 \coordinate (v1_1) at ($.7*(v1)+.05*(v2)+.25*(v3)$);
			 \coordinate (v1_2) at ($.8*(v1)+.05*(v2)+.15*(v3)$);
			 \coordinate (v1_3) at ($.8*(v1)+.15*(v2)+.05*(v3)$);
			 \coordinate (v1_4) at ($.7*(v1)+.25*(v2)+.05*(v3)$);
		 	 \coordinate (w1) at (intersection of v1_1--v1_2 and v1_3--v1_4);		 
			 \draw[blue,thick] (v1_1) --(w1) --(v1_4) ;

			 \coordinate (v2_1) at ($.7*(v2)+.05*(v1)+.25*(v3)$);
			 \coordinate (v2_2) at ($.8*(v2)+.05*(v1)+.15*(v3)$);
			 \coordinate (v2_3) at ($.8*(v2)+.15*(v1)+.05*(v3)$);
			 \coordinate (v2_4) at ($.7*(v2)+.25*(v1)+.05*(v3)$);
		 	 \coordinate (w2) at (intersection of v2_1--v2_2 and v2_3--v2_4);		 
			 \draw[red,thick] (v2_1)--(v2_2)--(v2_3) --(v2_4) ;
			 
			 \coordinate (v3_1) at ($.7*(v3)+.05*(v2)+.25*(v1)$);
			 \coordinate (v3_2) at ($.8*(v3)+.05*(v2)+.15*(v1)$);
			 \coordinate (v3_3) at ($.8*(v3)+.15*(v2)+.05*(v1)$);
			 \coordinate (v3_4) at ($.7*(v3)+.25*(v2)+.05*(v1)$);
		 	 \coordinate (w3) at (intersection of v3_1--v3_2 and v3_3--v3_4);		 
			 \draw[red,thick] (v3_1)--(v3_2)--(v3_3) --(v3_4) ;		 			 	
   		    \end{tikzpicture}
		    \subcaption{Bisimplex}
	\end{subfigure}
	   \begin{subfigure}{0.23\linewidth}
  		    \begin{tikzpicture}[scale=0.3]
   			 \coordinate (v1) at (0,4);
			 \coordinate (v2) at (4,-4);
			 \coordinate (v3) at (-4,-4);
			 \draw[very thick] (v1)--(v2)--(v3)--(v1);
			 
			 \coordinate (v1_1) at ($.7*(v1)+.05*(v2)+.25*(v3)$);
			 \coordinate (v1_2) at ($.8*(v1)+.05*(v2)+.15*(v3)$);
			 \coordinate (v1_3) at ($.8*(v1)+.15*(v2)+.05*(v3)$);
			 \coordinate (v1_4) at ($.7*(v1)+.25*(v2)+.05*(v3)$);
		 	 \coordinate (w1) at (intersection of v1_1--v1_2 and v1_3--v1_4);		 
			 \draw[red,thick] (v1_1)--(v1_2) --(v1_3)  --(v1_4) ;

			 \coordinate (v2_1) at ($.7*(v2)+.05*(v1)+.25*(v3)$);
			 \coordinate (v2_2) at ($.8*(v2)+.05*(v1)+.15*(v3)$);
			 \coordinate (v2_3) at ($.8*(v2)+.15*(v1)+.05*(v3)$);
			 \coordinate (v2_4) at ($.7*(v2)+.25*(v1)+.05*(v3)$);
		 	 \coordinate (w2) at (intersection of v2_1--v2_2 and v2_3--v2_4);		 
			 \draw[red,thick] (v2_1)--(v2_2)--(v2_3) --(v2_4) ;
			 
			 \coordinate (v3_1) at ($.7*(v3)+.05*(v2)+.25*(v1)$);
			 \coordinate (v3_2) at ($.8*(v3)+.05*(v2)+.15*(v1)$);
			 \coordinate (v3_3) at ($.8*(v3)+.15*(v2)+.05*(v1)$);
			 \coordinate (v3_4) at ($.7*(v3)+.25*(v2)+.05*(v1)$);
		 	 \coordinate (w3) at (intersection of v3_1--v3_2 and v3_3--v3_4);		 
			 \draw[red,thick] (v3_1)--(v3_2)--(v3_3) --(v3_4) ;		 			 	
   		    \end{tikzpicture}
		    \subcaption{Octahedron}
	\end{subfigure}			    
	\caption{}
	\label{fig:sec4_case_octahedron_triangle_labelings}
	\end{center}
\end{figure}

To enumerate all possible labelings of the facets of $F$ we use the three labeling rules in Figure~\ref{fig:sec4_case_octahedron_rules}. Rule~\ref{subfig:sec4_labeling_rule1} comes from the fact that by Proposition~\ref{prop:intersections} no facet of $P$ intersects $F$ in a vertex or an edge. 
To elaborate, let $(v,F_1')$ be the pair marked with a red cut angle and $F_2',F_3',F_4'$ be the facets of $F$ that touch $v$ numbered clockwise from $F_1'$ (see Figure~\ref{subfig:sec4_labeling_rule1_exp}). As usual, $F_i$ denotes the facet of $P$ that intersects $F$ in $F_i'$. 
Let $G'$ be the facet of $F_1$ that intersects $F$ in only $v$ and $G$ be the facet of $P$ that intersects $F_1$ in $G'$.  Since the intersection of $G$ and $F$ is nonempty, $G$ intersects $F$ in one of its facets. Hence $G$ intersects $F$ in $F'_2$, $F'_3$ or $F'_4$, i.e. $G$ equals $F_2$, $F_3$ or $F_4$. Note that $G$ cannot be either $F_2$ or $F_4$ as they already contain facets of $F_1$ distinct from $G'$. Thus, $G$ equals $F_3$ and $v$ has degree four in $F_3$, showing that 
$(v,F_3')$ is marked by a red cut angle. It is now easy to see that $v$ must have degree three in $F_2$ as its intersections with 
$F_3$ and $F_1$ share an edge of $G'$. Similarly for $F_4$ and we get that $(v,F_2')$ and $(v,F_4')$ must be blue angles.

Let us prove that rule~\ref{subfig:sec4_labeling_rule2} is valid. Indeed, let $v,F_1,F_2,F_3,F_4$ be as in Figure~\ref{subfig:sec4_labeling_rule2_exp} with $F_1$ being the bisimplicial facet.
By Rule~\ref{subfig:sec4_labeling_rule1} we know that $(v,F_3')$ must be a blue angle and $(v,F_2')$ and $(v,F_4')$ are either both blue angles or both red cut angles. If all are blue, then $F_1,\ldots,F_4$ all share a common edge. { In particular, the bisimplicial facet $F_1$ and facet $F_3$ intersect in an edge, contradicting Proposition~\ref{prop:intersections}.

For Rule~\ref{subfig:sec4_labeling_rule3} we use the fact that only square pyramids among facets of $P$ can have a square face. Thus, if $F_1$ is a square pyramid, the triangular facet of $F$ not containing the apex of $F_1$ but containing an edge of $F'_1$ must also be a facet of a square pyramid.

\begin{figure}	
	  \begin{center}
	   \begin{subfigure}{0.32\linewidth}
  		    \begin{tikzpicture}[scale=0.17]
   			 \coordinate (v1) at (-4,-4);
			 \coordinate (v2) at (4,-4);
			 \coordinate (v3) at (4,4);
			 \coordinate (v4) at (-4,4);
			 \coordinate (w) at (0,0);
			 \draw[very thick] (v1)--(v3);
			 \draw[very thick] (v2)--(v4);
			 
			 \coordinate (v34_1) at ($.5*(w)+.1*(v4)+.4*(v3)$);
			 \coordinate (v34_2) at ($.7*(w)+.1*(v4)+.2*(v3)$);
			 \coordinate (v34_3) at ($.7*(w)+.2*(v4)+.1*(v3)$);
			 \coordinate (v34_4) at ($.5*(w)+.4*(v4)+.1*(v3)$);
		 	 \coordinate (w34) at (intersection of v34_1--v34_2 and v34_3--v34_4);		 
			 \draw[red,thick] (v34_1)--(v34_2)--(v34_3)--(v34_4) ;

           		 \draw[->,double, >=angle 90] (5.5,0) --(7.5,0);

			\begin{scope}[xshift=13cm]
   			 \coordinate (v1) at (-4,-4);
			 \coordinate (v2) at (4,-4);
			 \coordinate (v3) at (4,4);
			 \coordinate (v4) at (-4,4);
			 \coordinate (w) at (0,0);
			 \draw[very thick] (v1)--(v3);
			 \draw[very thick] (v2)--(v4);
			
			 \coordinate (v34_1) at ($.5*(w)+.1*(v4)+.4*(v3)$);
			 \coordinate (v34_2) at ($.7*(w)+.1*(v4)+.2*(v3)$);
			 \coordinate (v34_3) at ($.7*(w)+.2*(v4)+.1*(v3)$);
			 \coordinate (v34_4) at ($.5*(w)+.4*(v4)+.1*(v3)$);
		 	 \coordinate (w34) at (intersection of v34_1--v34_2 and v34_3--v34_4);		 
			 \draw[red,thick] (v34_1)--(v34_2)--(v34_3)--(v34_4) ;
			 
			 \coordinate (v14_1) at ($.5*(w)+.1*(v4)+.4*(v1)$);
			 \coordinate (v14_2) at ($.7*(w)+.1*(v4)+.2*(v1)$);
			 \coordinate (v14_3) at ($.7*(w)+.2*(v4)+.1*(v1)$);
			 \coordinate (v14_4) at ($.5*(w)+.4*(v4)+.1*(v1)$);
		 	 \coordinate (w14) at (intersection of v14_1--v14_2 and v14_3--v14_4);
			 \draw[blue,thick] (v14_1)--(w14)--(v14_4) ; 	
			 
			 \coordinate (v32_1) at ($.5*(w)+.1*(v2)+.4*(v3)$);
			 \coordinate (v32_2) at ($.7*(w)+.1*(v2)+.2*(v3)$);
			 \coordinate (v32_3) at ($.7*(w)+.2*(v2)+.1*(v3)$);
			 \coordinate (v32_4) at ($.5*(w)+.4*(v2)+.1*(v3)$);
		 	 \coordinate (w32) at (intersection of v32_1--v32_2 and v32_3--v32_4);
			 \draw[blue,thick] (v32_1)--(w32)--(v32_4) ;	
			 
			 \coordinate (v12_1) at ($.5*(w)+.1*(v1)+.4*(v2)$);
			 \coordinate (v12_2) at ($.7*(w)+.1*(v1)+.2*(v2)$);
			 \coordinate (v12_3) at ($.7*(w)+.2*(v1)+.1*(v2)$);
			 \coordinate (v12_4) at ($.5*(w)+.4*(v1)+.1*(v2)$);
		 	 \coordinate (w12) at (intersection of v12_1--v12_2 and v12_3--v12_4);	
			 \draw[red,thick] (v12_1)--(v12_2)--(v12_3)--(v12_4) ;	   
			 
			\end{scope}
	 			 	
   		    \end{tikzpicture}
		    \subcaption{}
		    \label{subfig:sec4_labeling_rule1}
	\end{subfigure}
	   \begin{subfigure}{0.32\linewidth}
  		    \begin{tikzpicture}[scale=0.17]
   			 \coordinate (v1) at (-3,-3);
			 \coordinate (v2) at (3,-3);
			 \coordinate (v3) at (5,5);
			 \coordinate (v4) at (-5,5);
			 \coordinate (v5) at (v3);
			 \coordinate (v6) at (v4);
			 \coordinate (w) at (0,0);
			 \draw[very thick] (v1)--(v3);
			 \draw[very thick] (v2)--(v4);
			 \draw[very thick] (-5,5)--(5,5);
			 
			 \coordinate (v56_1) at ($.5*(w)+.1*(v5)+.4*(v6)$);
			 \coordinate (v56_2) at ($.7*(w)+.1*(v5)+.2*(v6)$);
			 \coordinate (v56_3) at ($.7*(w)+.2*(v5)+.1*(v6)$);
			 \coordinate (v56_4) at ($.5*(w)+.4*(v5)+.1*(v6)$);
		 	 \coordinate (w56) at (intersection of v56_1--v56_2 and v56_3--v56_4);		 
			 \draw[blue,thick] (v56_1)--(w56)--(v56_4) ; 
			 
			 \coordinate (v5_1) at ($.5*(v5)+.1*(w)+.4*(v6)$);
			 \coordinate (v5_2) at ($.7*(v5)+.1*(w)+.2*(v6)$);
			 \coordinate (v5_3) at ($.7*(v5)+.2*(w)+.1*(v6)$);
			 \coordinate (v5_4) at ($.5*(v5)+.4*(w)+.1*(v6)$);
		 	 \coordinate (w5) at (intersection of v5_1--v5_2 and v5_3--v5_4);		 
			 \draw[red,thick] (v5_1)--(v5_2)--(v5_3)--(v5_4) ; 
			 
			 \coordinate (v6_1) at ($.5*(v6)+.1*(v5)+.4*(w)$);
			 \coordinate (v6_2) at ($.7*(v6)+.1*(v5)+.2*(w)$);
			 \coordinate (v6_3) at ($.7*(v6)+.2*(v5)+.1*(w)$);
			 \coordinate (v6_4) at ($.5*(v6)+.4*(v5)+.1*(w)$);
		 	 \coordinate (w6) at (intersection of v6_1--v6_2 and v6_3--v6_4);		 
			 \draw[red,thick] (v6_1)--(v6_2)--(v6_3)--(v6_4) ; 

           		 \draw[->,double, >=angle 90] (5.5,0) --(7.5,0);			 

			\begin{scope}[xshift=13cm]
			
    			 \coordinate (v1) at (-3,-3);
			 \coordinate (v2) at (3,-3);
			 \coordinate (v3) at (5,5);
			 \coordinate (v4) at (-5,5);
			 \coordinate (v5) at (v3);
			 \coordinate (v6) at (v4);
			 \coordinate (w) at (0,0);
			 \draw[very thick] (v1)--(v3);
			 \draw[very thick] (v2)--(v4);
			 \draw[very thick] (-5,5)--(5,5);
			 
			 \coordinate (v56_1) at ($.5*(w)+.1*(v5)+.4*(v6)$);
			 \coordinate (v56_2) at ($.7*(w)+.1*(v5)+.2*(v6)$);
			 \coordinate (v56_3) at ($.7*(w)+.2*(v5)+.1*(v6)$);
			 \coordinate (v56_4) at ($.5*(w)+.4*(v5)+.1*(v6)$);
		 	 \coordinate (w56) at (intersection of v56_1--v56_2 and v56_3--v56_4);		 
			 \draw[blue,thick] (v56_1)--(w56)--(v56_4) ; 
			 
			 \coordinate (v5_1) at ($.5*(v5)+.1*(w)+.4*(v6)$);
			 \coordinate (v5_2) at ($.7*(v5)+.1*(w)+.2*(v6)$);
			 \coordinate (v5_3) at ($.7*(v5)+.2*(w)+.1*(v6)$);
			 \coordinate (v5_4) at ($.5*(v5)+.4*(w)+.1*(v6)$);
		 	 \coordinate (w5) at (intersection of v5_1--v5_2 and v5_3--v5_4);		 
			 \draw[red,thick] (v5_1)--(v5_2)--(v5_3)--(v5_4) ; 
			 
			 \coordinate (v6_1) at ($.5*(v6)+.1*(v5)+.4*(w)$);
			 \coordinate (v6_2) at ($.7*(v6)+.1*(v5)+.2*(w)$);
			 \coordinate (v6_3) at ($.7*(v6)+.2*(v5)+.1*(w)$);
			 \coordinate (v6_4) at ($.5*(v6)+.4*(v5)+.1*(w)$);
		 	 \coordinate (w6) at (intersection of v6_1--v6_2 and v6_3--v6_4);		 
			 \draw[red,thick] (v6_1)--(v6_2)--(v6_3)--(v6_4) ; 
			 
			 \coordinate (v14_1) at ($.4*(w)+.1*(v4)+.5*(v1)$);
			 \coordinate (v14_2) at ($.7*(w)+.1*(v4)+.2*(v1)$);
			 \coordinate (v14_3) at ($.7*(w)+.2*(v4)+.1*(v1)$);
			 \coordinate (v14_4) at ($.5*(w)+.4*(v4)+.1*(v1)$);
		 	 \coordinate (w14) at (intersection of v14_1--v14_2 and v14_3--v14_4);
			 \draw[red,thick] (v14_1)--(v14_2)--(v14_3)--(v14_4) ; 	
			 
			 \coordinate (v32_1) at ($.5*(w)+.1*(v2)+.4*(v3)$);
			 \coordinate (v32_2) at ($.7*(w)+.1*(v2)+.2*(v3)$);
			 \coordinate (v32_3) at ($.7*(w)+.2*(v2)+.1*(v3)$);
			 \coordinate (v32_4) at ($.4*(w)+.5*(v2)+.1*(v3)$);
		 	 \coordinate (w32) at (intersection of v32_1--v32_2 and v32_3--v32_4);
			 \draw[red,thick] (v32_1)--(v32_2)--(v32_3)--(v32_4) ;	
			 
			 \coordinate (v12_1) at ($.4*(w)+.1*(v1)+.5*(v2)$);
			 \coordinate (v12_2) at ($.7*(w)+.1*(v1)+.2*(v2)$);
			 \coordinate (v12_3) at ($.7*(w)+.2*(v1)+.1*(v2)$);
			 \coordinate (v12_4) at ($.4*(w)+.5*(v1)+.1*(v2)$);
		 	 \coordinate (w12) at (intersection of v12_1--v12_2 and v12_3--v12_4);	
			 \draw[blue,thick] (v12_1)--(w12)--(v12_4) ;	   			
		    \end{scope}
	 			 	
   		    \end{tikzpicture}
		    \subcaption{}
		    \label{subfig:sec4_labeling_rule2}
	\end{subfigure}	
	\begin{subfigure}{0.32\linewidth}
  		    \begin{tikzpicture}[scale=0.17]
    			 \coordinate (v1) at (0, 4);
			 \coordinate (v2) at (0,-4);
			 \coordinate (v3) at (4,0);
			 \coordinate (v4) at (-4,0);
			 \draw[very thick] (v1)--(v3)--(v4)--(v1);
			 \draw[very thick] (v2)--(v3)--(v4)--(v2);
			 
			 \coordinate (v43_1) at ($.5*(v1)+.1*(v3)+.4*(v4)$);
			 \coordinate (v43_2) at ($.7*(v1)+.1*(v3)+.2*(v4)$);
			 \coordinate (v43_3) at ($.7*(v1)+.2*(v3)+.1*(v4)$);
			 \coordinate (v43_4) at ($.5*(v1)+.4*(v3)+.1*(v4)$); 
			 \draw[red,thick] (v43_1)--(v43_2)--(v43_3)--(v43_4) ; 
			 
			 \coordinate (v13_1) at ($.5*(v4)+.1*(v3)+.4*(v1)$);
			 \coordinate (v13_2) at ($.6*(v4)+.1*(v3)+.3*(v1)$);
			 \coordinate (v13_3) at ($.6*(v4)+.2*(v3)+.2*(v1)$);
			 \coordinate (v13_4) at ($.5*(v4)+.3*(v3)+.2*(v1)$);
		 	 \coordinate (w13) at (intersection of v13_1--v13_2 and v13_3--v13_4);		 
			 \draw[blue,thick] (v13_1)--(w13)--(v13_4) ; 
			 
			 \coordinate (v41_1) at ($.5*(v3)+.1*(v4)+.4*(v1)$);
			 \coordinate (v41_2) at ($.6*(v3)+.1*(v4)+.3*(v1)$);
			 \coordinate (v41_3) at ($.6*(v3)+.2*(v4)+.2*(v1)$);
			 \coordinate (v41_4) at ($.5*(v3)+.3*(v4)+.2*(v1)$);
		 	 \coordinate (w41) at (intersection of v41_1--v41_2 and v41_3--v41_4);		 
			 \draw[blue,thick] (v41_1)--(w41)--(v41_4) ; 

           		 \draw[->,double, >=angle 90] (5.5,0) --(7.5,0);			 

			\begin{scope}[xshift=13cm]
			
    			 \coordinate (v1) at (0, 4);
			 \coordinate (v2) at (0,-4);
			 \coordinate (v3) at (4,0);
			 \coordinate (v4) at (-4,0);
			 \draw[very thick] (v1)--(v3)--(v4)--(v1);
			 \draw[very thick] (v2)--(v3)--(v4)--(v2);
			 
			\coordinate (v43_1) at ($.5*(v1)+.1*(v3)+.4*(v4)$);
			 \coordinate (v43_2) at ($.7*(v1)+.1*(v3)+.2*(v4)$);
			 \coordinate (v43_3) at ($.7*(v1)+.2*(v3)+.1*(v4)$);
			 \coordinate (v43_4) at ($.5*(v1)+.4*(v3)+.1*(v4)$); 
			 \draw[red,thick] (v43_1)--(v43_2)--(v43_3)--(v43_4) ; 
			 
			 \coordinate (v13_1) at ($.5*(v4)+.1*(v3)+.4*(v1)$);
			 \coordinate (v13_2) at ($.6*(v4)+.1*(v3)+.3*(v1)$);
			 \coordinate (v13_3) at ($.6*(v4)+.2*(v3)+.2*(v1)$);
			 \coordinate (v13_4) at ($.5*(v4)+.3*(v3)+.2*(v1)$);
		 	 \coordinate (w13) at (intersection of v13_1--v13_2 and v13_3--v13_4);		 
			 \draw[blue,thick] (v13_1)--(w13)--(v13_4) ; 
			 
			 \coordinate (v41_1) at ($.5*(v3)+.1*(v4)+.4*(v1)$);
			 \coordinate (v41_2) at ($.6*(v3)+.1*(v4)+.3*(v1)$);
			 \coordinate (v41_3) at ($.6*(v3)+.2*(v4)+.2*(v1)$);
			 \coordinate (v41_4) at ($.5*(v3)+.3*(v4)+.2*(v1)$);
		 	 \coordinate (w41) at (intersection of v41_1--v41_2 and v41_3--v41_4);		 
			 \draw[blue,thick] (v41_1)--(w41)--(v41_4) ; 
			 
			 \coordinate (v43_1) at ($.5*(v2)+.1*(v3)+.4*(v4)$);
			 \coordinate (v43_2) at ($.7*(v2)+.1*(v3)+.2*(v4)$);
			 \coordinate (v43_3) at ($.7*(v2)+.2*(v3)+.1*(v4)$);
			 \coordinate (v43_4) at ($.5*(v2)+.4*(v3)+.1*(v4)$); 
			 \draw[red,thick] (v43_1)--(v43_2)--(v43_3)--(v43_4) ; 
			 
			 \coordinate (v23_1) at ($.5*(v4)+.1*(v3)+.4*(v2)$);
			 \coordinate (v23_2) at ($.6*(v4)+.1*(v3)+.3*(v2)$);
			 \coordinate (v23_3) at ($.6*(v4)+.2*(v3)+.2*(v2)$);
			 \coordinate (v23_4) at ($.5*(v4)+.3*(v3)+.2*(v2)$);
		 	 \coordinate (w23) at (intersection of v23_1--v23_2 and v23_3--v23_4);		 
			 \draw[blue,thick] (v23_1)--(w23)--(v23_4) ; 
			 
			 \coordinate (v42_1) at ($.5*(v3)+.1*(v4)+.4*(v2)$);
			 \coordinate (v42_2) at ($.6*(v3)+.1*(v4)+.3*(v2)$);
			 \coordinate (v42_3) at ($.6*(v3)+.2*(v4)+.2*(v2)$);
			 \coordinate (v42_4) at ($.5*(v3)+.3*(v4)+.2*(v2)$);
		 	 \coordinate (w42) at (intersection of v42_1--v42_2 and v42_3--v42_4);		 
			 \draw[blue,thick] (v42_1)--(w42)--(v42_4) ; 
			    			
		    \end{scope}
	 			 	
   		    \end{tikzpicture}
		    \subcaption{}
		    \label{subfig:sec4_labeling_rule3}
	\end{subfigure}			    
			    
	\caption{}
	\label{fig:sec4_case_octahedron_rules}
	\end{center}
\end{figure}

\begin{figure}	
	  \begin{center}
	   \begin{subfigure}{0.32\linewidth}
  		    \begin{tikzpicture}[scale=0.2]
   			 \coordinate (v1) at (-4,-4);
			 \coordinate (v2) at (4,-4);
			 \coordinate (v3) at (4,4);
			 \coordinate (v4) at (-4,4);
			 \coordinate (w) at (0,0);
			 \draw[very thick] (v1)--(v3);
			 \draw[very thick] (v2)--(v4);
			 
			 \coordinate (v34_1) at ($.5*(w)+.1*(v4)+.4*(v3)$);
			 \coordinate (v34_2) at ($.7*(w)+.1*(v4)+.2*(v3)$);
			 \coordinate (v34_3) at ($.7*(w)+.2*(v4)+.1*(v3)$);
			 \coordinate (v34_4) at ($.5*(w)+.4*(v4)+.1*(v3)$);
		 	 \coordinate (w34) at (intersection of v34_1--v34_2 and v34_3--v34_4);		 
			 \draw[red,thick] (v34_1)--(v34_2)--(v34_3)--(v34_4) ; 
			 \node at (1.5,0) {$v$};
			 \node at (0,3) {$F'_1$};
			 \node at (5,0) {$F'_2$};
			 \node at (0,-3) {$F'_3$};
			 \node at (-5,0) {$F'_4$};	 			 	
   		    \end{tikzpicture}
		    \subcaption{}
		    \label{subfig:sec4_labeling_rule1_exp}
	\end{subfigure}
	 \begin{subfigure}{0.4\linewidth}
  		    \begin{tikzpicture}[scale=0.2]
   			 \coordinate (v1) at (-3,-3);
			 \coordinate (v2) at (3,-3);
			 \coordinate (v3) at (5,5);
			 \coordinate (v4) at (-5,5);
			 \coordinate (v5) at (v3);
			 \coordinate (v6) at (v4);
			 \coordinate (w) at (0,0);
			 \draw[very thick] (v1)--(v3);
			 \draw[very thick] (v2)--(v4);
			 \draw[very thick] (-5,5)--(5,5);
			 
			 \coordinate (v56_1) at ($.5*(w)+.1*(v5)+.4*(v6)$);
			 \coordinate (v56_2) at ($.7*(w)+.1*(v5)+.2*(v6)$);
			 \coordinate (v56_3) at ($.7*(w)+.2*(v5)+.1*(v6)$);
			 \coordinate (v56_4) at ($.5*(w)+.4*(v5)+.1*(v6)$);
		 	 \coordinate (w56) at (intersection of v56_1--v56_2 and v56_3--v56_4);		 
			 \draw[blue,thick] (v56_1)--(w56)--(v56_4) ; 
			 
			 \coordinate (v5_1) at ($.5*(v5)+.1*(w)+.4*(v6)$);
			 \coordinate (v5_2) at ($.7*(v5)+.1*(w)+.2*(v6)$);
			 \coordinate (v5_3) at ($.7*(v5)+.2*(w)+.1*(v6)$);
			 \coordinate (v5_4) at ($.5*(v5)+.4*(w)+.1*(v6)$);
		 	 \coordinate (w5) at (intersection of v5_1--v5_2 and v5_3--v5_4);		 
			 \draw[red,thick] (v5_1)--(v5_2)--(v5_3)--(v5_4) ; 
			 
			 \coordinate (v6_1) at ($.5*(v6)+.1*(v5)+.4*(w)$);
			 \coordinate (v6_2) at ($.7*(v6)+.1*(v5)+.2*(w)$);
			 \coordinate (v6_3) at ($.7*(v6)+.2*(v5)+.1*(w)$);
			 \coordinate (v6_4) at ($.5*(v6)+.4*(v5)+.1*(w)$);
		 	 \coordinate (w6) at (intersection of v6_1--v6_2 and v6_3--v6_4);		 
			 \draw[red,thick] (v6_1)--(v6_2)--(v6_3)--(v6_4) ; 
			 
			 \node at (1.5,0) {$v$};
			 \node at (0,3) {$F'_1$};
			 \node at (5,0) {$F'_2$};
			 \node at (0,-3) {$F'_3$};
			 \node at (-5,0) {$F'_4$};		 			 	
   		    \end{tikzpicture}
		    \subcaption{}
		    \label{subfig:sec4_labeling_rule2_exp}
	\end{subfigure}	
	\caption{}
	\end{center}
\end{figure}

We now use these rules to produce all valid diagrams. If we start by assuming $F_1$ is an octahedron we get a unique valid diagram, shown in Figure~\ref{subfig:sec4_octahedron_case_octahedron}.  If we assume $F_1$ to be a square pyramid we only obtain the diagram in Figure~\ref{subfig:sec4_octahedron_case_square_pyramid}. The assumption that $F_1$ is a bisimplex does not lead to any valid diagram. As before, there is a unique way to produce a polytope that complies with each diagram and respects the facet intersection properties of psd-minimal $4$-polytopes. The Schlegel diagrams of polytopes in these classes are presented in Figure~\ref{fig:sec4_octahedron_result}. It turns out that the class shown in Figure \ref{fig:sec4_octahedron_result1} is dual to class $18$, so it is already accounted for.

\begin{figure}[h!]	
	  \begin{center}
	  \begin{subfigure}{0.4\linewidth}
  		    \begin{tikzpicture}[scale=0.4]
   			 \coordinate (v1) at (0,4);
			 \coordinate (v2) at (4,-4);
			 \coordinate (v3) at (-4,-4);
			 \coordinate (v4) at ($-.2*(v1)+1.2*(0,-1)-(0,.5)$);
   			 \coordinate (v5) at ($-.2*(v2)+1.2*(0,-1)-(0,.5)$);
			 \coordinate (v6) at ($-.2*(v3)+1.2*(0,-1)-(0,.5)$);
			 \draw[very thick] (v1)--(v2)--(v3)--(v1);
			 \draw[very thick] (v5)--(v6)--(v4)--(v5);
			 \draw[very thick] (v1)--(v5);
			 \draw[very thick] (v1)--(v6);
			 \draw[very thick] (v2)--(v4);
			 \draw[very thick] (v2)--(v6);
			 \draw[very thick] (v3)--(v4);
			 \draw[very thick] (v3)--(v5);

			 \coordinate (v1_1_1) at ($.6*(v1)+.1*(v5)+.3*(v6)$);
			 \coordinate (v1_1_2) at ($.7*(v1)+.1*(v5)+.2*(v6)$);
			 \coordinate (v1_1_3) at ($.7*(v1)+.2*(v5)+.1*(v6)$);
			 \coordinate (v1_1_4) at ($.6*(v1)+.3*(v5)+.1*(v6)$);
		 	 \coordinate (v1_1) at (intersection of v1_1_1--v1_1_2 and v1_1_3--v1_1_4);		 
			 \draw[blue,thick] (v1_1_1) --(v1_1) --(v1_1_4) ;
			 
 			 \coordinate (v1_2_1) at ($.6*(v1)+.1*(v6)+.3*(v5)+.2*(v5)-.2*(v6)$);
			 \coordinate (v1_2_2) at ($.7*(v1)+.1*(v6)+.2*(v5)+.2*(v5)-.2*(v6)$);
			 \coordinate (v1_2_3) at ($.7*(v1)+.1*(v6)+.2*(v5)+.3*(v5)-.3*(v6)$);
			 \coordinate (v1_2_4) at ($.7*(v1)+.1*(v6)+.2*(v5)+.3*(v5)-.3*(v6)+.07*(v3)-.07*(v1)$);
			 \draw[red,thick] (v1_2_1) --(v1_2_2) --(v1_2_3) --(v1_2_4) ;

 			 \coordinate (v1_3_1) at ($.6*(v1)+.1*(v5)+.3*(v6)+.2*(v6)-.2*(v5)$);
			 \coordinate (v1_3_2) at ($.7*(v1)+.1*(v5)+.2*(v6)+.2*(v6)-.2*(v5)$);
			 \coordinate (v1_3_3) at ($.7*(v1)+.1*(v5)+.2*(v6)+.3*(v6)-.3*(v5)$);
			 \coordinate (v1_3_4) at ($.7*(v1)+.1*(v5)+.2*(v6)+.3*(v6)-.3*(v5)+.07*(v2)-.07*(v1)$);
			 \draw[red,thick] (v1_3_1) --(v1_3_2) --(v1_3_3) --(v1_3_4) ;			 
			 
			 \coordinate (v1_4_1) at ($1.1*(v1)+.1*(v5)+.2*(v6)+.1*(v2)-.1*(v1)-.1*(v5)-.1*(v6)$);
			 \coordinate (v1_4_2) at ($1.1*(v1)+.1*(v5)+.2*(v6)-.1*(v5)-.1*(v6)$);
			 \coordinate (v1_4_3) at ($1.1*(v1)+.2*(v5)+.1*(v6)-.1*(v5)-.1*(v6)$);
			 \coordinate (v1_4_4) at ($1.1*(v1)+.2*(v5)+.1*(v6)+.1*(v3)-.1*(v1)-.1*(v5)-.1*(v6)$);
		 	 \coordinate (v1_4) at (intersection of v1_4_1--v1_4_2 and v1_4_3--v1_4_4);		 
			 \draw[blue,thick] (v1_4_1) --(v1_4) --(v1_4_4) ;
			 
			 \coordinate (v2_1_1) at ($.6*(v2)+.1*(v4)+.3*(v6)$);
			 \coordinate (v2_1_2) at ($.7*(v2)+.1*(v4)+.2*(v6)$);
			 \coordinate (v2_1_3) at ($.7*(v2)+.2*(v4)+.1*(v6)$);
			 \coordinate (v2_1_4) at ($.6*(v2)+.3*(v4)+.1*(v6)$);
			 \coordinate (v2_1) at (intersection of v2_1_1--v2_1_2 and v2_1_3--v2_1_4);	
			 \draw[blue,thick] (v2_1_1) --(v2_1) --(v2_1_4) ;
			 
 			 \coordinate (v2_2_1) at ($.6*(v2)+.1*(v6)+.3*(v4)+.2*(v4)-.2*(v6)$);
			 \coordinate (v2_2_2) at ($.7*(v2)+.1*(v6)+.2*(v4)+.2*(v4)-.2*(v6)$);
			 \coordinate (v2_2_3) at ($.7*(v2)+.1*(v6)+.2*(v4)+.3*(v4)-.3*(v6)$);
			 \coordinate (v2_2_4) at ($.7*(v2)+.1*(v6)+.2*(v4)+.3*(v4)-.3*(v6)+.07*(v3)-.07*(v2)$);
			 \draw[red,thick] (v2_2_1) --(v2_2_2) --(v2_2_3) --(v2_2_4) ;

 			 \coordinate (v2_3_1) at ($.6*(v2)+.1*(v4)+.3*(v6)+.2*(v6)-.2*(v4)$);
			 \coordinate (v2_3_2) at ($.7*(v2)+.1*(v4)+.2*(v6)+.2*(v6)-.2*(v4)$);
			 \coordinate (v2_3_3) at ($.7*(v2)+.1*(v4)+.2*(v6)+.3*(v6)-.3*(v4)$);
			 \coordinate (v2_3_4) at ($.7*(v2)+.1*(v4)+.2*(v6)+.3*(v6)-.3*(v4)+.07*(v1)-.07*(v2)$);
			 \draw[red,thick] (v2_3_1) --(v2_3_2) --(v2_3_3) --(v2_3_4) ;			 
			 
			 \coordinate (v2_4_1) at ($.7*(v2)+.1*(v6)+.2*(v4)-.1*(v4)+.1*(v6)+.4*(v2)-.2*(v6)-0.2*(v4)+.1*(v1)-.1*(v2)$);
			 \coordinate (v2_4_2) at ($.7*(v2)+.1*(v6)+.2*(v4)-.1*(v4)+.1*(v6)+.4*(v2)-.2*(v6)-0.2*(v4)$);
			 \coordinate (v2_4_3) at ($.7*(v2)+.1*(v6)+.2*(v4)+.01*(v4)-.01*(v6)+.4*(v2)-.2*(v6)-0.2*(v4)$);
			 \coordinate (v2_4_4) at ($.7*(v2)+.1*(v6)+.2*(v4)+.01*(v4)-.01*(v6)+.4*(v2)-.2*(v6)-0.2*(v4)+.1*(v3)-.1*(v2)$);
			 \coordinate (v2_4) at (intersection of v2_4_1--v2_4_2 and v2_4_3--v2_4_4);	
			 \draw[blue,thick] (v2_4_1) --(v2_4) --(v2_4_4) ;
			 
			 \coordinate (v3_1_1) at ($.6*(v3)+.1*(v4)+.3*(v5)$);
			 \coordinate (v3_1_2) at ($.7*(v3)+.1*(v4)+.2*(v5)$);
			 \coordinate (v3_1_3) at ($.7*(v3)+.2*(v4)+.1*(v5)$);
			 \coordinate (v3_1_4) at ($.6*(v3)+.3*(v4)+.1*(v5)$);
			 \coordinate (v3_1) at (intersection of v3_1_1--v3_1_2 and v3_1_3--v3_1_4);	
			 \draw[blue,thick] (v3_1_1) --(v3_1) --(v3_1_4) ;
			 
 			 \coordinate (v3_2_1) at ($.6*(v3)+.1*(v5)+.3*(v4)+.2*(v4)-.2*(v5)$);
			 \coordinate (v3_2_2) at ($.7*(v3)+.1*(v5)+.2*(v4)+.2*(v4)-.2*(v5)$);
			 \coordinate (v3_2_3) at ($.7*(v3)+.1*(v5)+.2*(v4)+.3*(v4)-.3*(v5)$);
			 \coordinate (v3_2_4) at ($.7*(v3)+.1*(v5)+.2*(v4)+.3*(v4)-.3*(v5)+.07*(v2)-.07*(v3)$);
			 \draw[red,thick] (v3_2_1) --(v3_2_2) --(v3_2_3) --(v3_2_4) ;

 			 \coordinate (v3_3_1) at ($.6*(v3)+.1*(v4)+.3*(v5)+.2*(v5)-.2*(v4)$);
			 \coordinate (v3_3_2) at ($.7*(v3)+.1*(v4)+.2*(v5)+.2*(v5)-.2*(v4)$);
			 \coordinate (v3_3_3) at ($.7*(v3)+.1*(v4)+.2*(v5)+.3*(v5)-.3*(v4)$);
			 \coordinate (v3_3_4) at ($.7*(v3)+.1*(v4)+.2*(v5)+.3*(v5)-.3*(v4)+.07*(v1)-.07*(v3)$);
			 \draw[red,thick] (v3_3_1) --(v3_3_2) --(v3_3_3) --(v3_3_4) ;			 
			 
			 \coordinate (v3_4_1) at ($.7*(v3)+.1*(v5)+.2*(v4)-.1*(v4)+.1*(v5)+.4*(v3)-.2*(v5)-0.2*(v4)+.1*(v1)-.1*(v3)$);
			 \coordinate (v3_4_2) at ($.7*(v3)+.1*(v5)+.2*(v4)-.1*(v4)+.1*(v5)+.4*(v3)-.2*(v5)-0.2*(v4)$);
			 \coordinate (v3_4_3) at ($.7*(v3)+.1*(v5)+.2*(v4)+.01*(v4)-.01*(v5)+.4*(v3)-.2*(v5)-0.2*(v4)$);
			 \coordinate (v3_4_4) at ($.7*(v3)+.1*(v5)+.2*(v4)+.01*(v4)-.01*(v5)+.4*(v3)-.2*(v5)-0.2*(v4)+.1*(v2)-.1*(v3)$);
			 \coordinate (v3_4) at (intersection of v3_4_1--v3_4_2 and v3_4_3--v3_4_4);	
			 \draw[blue,thick] (v3_4_1) --(v3_4) --(v3_4_4) ;

			 \coordinate (v4_1_1) at ($.5*(v4)+.1*(v5)+.4*(v6)$);
			 \coordinate (v4_1_2) at ($.7*(v4)+.1*(v5)+.2*(v6)$);
			 \coordinate (v4_1_3) at ($.7*(v4)+.2*(v5)+.1*(v6)$);
			 \coordinate (v4_1_4) at ($.5*(v4)+.4*(v5)+.1*(v6)$);
			 \draw[red,thick] (v4_1_1) --(v4_1_2) --(v4_1_3) --(v4_1_4) ;
			 
			 \coordinate (v4_2_1) at ($.7*(v4)+.1*(v6)+.2*(v2)$);
			 \coordinate (v4_2_2) at ($.8*(v4)+.1*(v6)+.1*(v2)$);
			 \coordinate (v4_2_3) at ($.7*(v4)+.25*(v6)+.05*(v2)$);
			 \coordinate (v4_2_4) at ($.5*(v4)+.45*(v6)+.05*(v2)$);
			 \coordinate (v4_2) at (intersection of v4_2_1--v4_2_2 and v4_2_3--v4_2_4);		 
			 \draw[blue,thick] (v4_2_1) --(v4_2) --(v4_2_4) ;

			 \coordinate (v4_3_1) at ($.7*(v4)+.1*(v5)+.2*(v3)$);
			 \coordinate (v4_3_2) at ($.8*(v4)+.1*(v5)+.1*(v3)$);
			 \coordinate (v4_3_3) at ($.7*(v4)+.25*(v5)+.05*(v3)$);
			 \coordinate (v4_3_4) at ($.5*(v4)+.45*(v5)+.05*(v3)$);
			 \coordinate (v4_3) at (intersection of v4_3_1--v4_3_2 and v4_3_3--v4_3_4);		 
			 \draw[blue,thick] (v4_3_1) --(v4_3) --(v4_3_4) ;
			 
			 \coordinate (v4_4_1) at ($.72*(v4)+.06*(v2)+.23*(v3)$);
			 \coordinate (v4_4_2) at ($.84*(v4)+.06*(v2)+.11*(v3)$);
			 \coordinate (v4_4_3) at ($.84*(v4)+.11*(v2)+.06*(v3)$);
			 \coordinate (v4_4_4) at ($.72*(v4)+.23*(v2)+.06*(v3)$);
			 \draw[red,thick] (v4_4_1) --(v4_4_2) --(v4_4_3) --(v4_4_4) ;
			 
			 \coordinate (v5_1_1) at ($.5*(v5)+.1*(v4)+.4*(v6)$);
			 \coordinate (v5_1_2) at ($.7*(v5)+.1*(v4)+.2*(v6)$);
			 \coordinate (v5_1_3) at ($.7*(v5)+.2*(v4)+.1*(v6)$);
			 \coordinate (v5_1_4) at ($.5*(v5)+.4*(v4)+.1*(v6)$);
			 \draw[red,thick] (v5_1_1) --(v5_1_2) --(v5_1_3) --(v5_1_4) ;
			 
			 \coordinate (v5_2_1) at ($.7*(v5)+.1*(v4)+.2*(v3)$);
			 \coordinate (v5_2_2) at ($.8*(v5)+.1*(v4)+.1*(v3)$);
			 \coordinate (v5_2_3) at ($.6*(v5)+.23*(v4)+.07*(v3)$);
			 \coordinate (v5_2_4) at ($.4*(v5)+.43*(v4)+.07*(v3)$);
			 \coordinate (v5_2) at (intersection of v5_2_1--v5_2_2 and v5_2_3--v5_2_4);
			 \draw[blue,thick] (v5_2_1) --(v5_2) --(v5_2_4) ;
			 
			 \coordinate (v5_3_1) at ($.72*(v5)+.06*(v1)+.23*(v3)$);
			 \coordinate (v5_3_2) at ($.84*(v5)+.06*(v1)+.11*(v3)$);
			 \coordinate (v5_3_3) at ($.84*(v5)+.11*(v1)+.06*(v3)$);
			 \coordinate (v5_3_4) at ($.72*(v5)+.23*(v1)+.06*(v3)$);
			 \draw[red,thick] (v5_3_1) --(v5_3_2) --(v5_3_3) --(v5_3_4) ;
			 
			 \coordinate (v5_4_1) at ($.7*(v5)+.1*(v6)+.2*(v1)$);
			 \coordinate (v5_4_2) at ($.8*(v5)+.1*(v6)+.1*(v1)$);
			 \coordinate (v5_4_3) at ($.7*(v5)+.25*(v6)+.05*(v1)$);
			 \coordinate (v5_4_4) at ($.5*(v5)+.45*(v6)+.05*(v1)$);
			 \coordinate (v5_4) at (intersection of v5_4_1--v5_4_2 and v5_4_3--v5_4_4);
			 \draw[blue,thick] (v5_4_1) --(v5_4) --(v5_4_4) ;
			 
			 \coordinate (v6_1_1) at ($.5*(v6)+.1*(v5)+.4*(v4)$);
			 \coordinate (v6_1_2) at ($.7*(v6)+.1*(v5)+.2*(v4)$);
			 \coordinate (v6_1_3) at ($.7*(v6)+.2*(v5)+.1*(v4)$);
			 \coordinate (v6_1_4) at ($.5*(v6)+.4*(v5)+.1*(v4)$);
			 \draw[red,thick] (v6_1_1) --(v6_1_2) --(v6_1_3) --(v6_1_4) ;
			 
			 \coordinate (v6_2_1) at ($.7*(v6)+.1*(v4)+.2*(v2)$);
			 \coordinate (v6_2_2) at ($.8*(v6)+.1*(v4)+.1*(v2)$);
			 \coordinate (v6_2_3) at ($.6*(v6)+.23*(v4)+.07*(v2)$);
			 \coordinate (v6_2_4) at ($.4*(v6)+.43*(v4)+.07*(v2)$);
			\coordinate (v6_2) at (intersection of v6_2_1--v6_2_2 and v6_2_3--v6_2_4);
			 \draw[blue,thick] (v6_2_1) --(v6_2)--(v6_2_4) ;
			 
			 \coordinate (v6_3_1) at ($.72*(v6)+.06*(v1)+.23*(v2)$);
			 \coordinate (v6_3_2) at ($.84*(v6)+.06*(v1)+.11*(v2)$);
			 \coordinate (v6_3_3) at ($.84*(v6)+.11*(v1)+.06*(v2)$);
			 \coordinate (v6_3_4) at ($.72*(v6)+.23*(v1)+.06*(v2)$);
			 \draw[red,thick] (v6_3_1) --(v6_3_2) --(v6_3_3) --(v6_3_4) ;
			 
			 \coordinate (v6_4_1) at ($.7*(v6)+.1*(v5)+.2*(v1)$);
			 \coordinate (v6_4_2) at ($.8*(v6)+.1*(v5)+.1*(v1)$);
			 \coordinate (v6_4_3) at ($.7*(v6)+.25*(v5)+.05*(v1)$);
			 \coordinate (v6_4_4) at ($.5*(v6)+.45*(v5)+.05*(v1)$);
			 \coordinate (v6_4) at (intersection of v6_4_1--v6_4_2 and v6_4_3--v6_4_4);
			 \draw[blue,thick] (v6_4_1) --(v6_4) --(v6_4_4) ;
   		    \end{tikzpicture}
		    \subcaption{}
		     \label{subfig:sec4_octahedron_case_octahedron}
	\end{subfigure}
	  \begin{subfigure}{0.45\linewidth}
 		    \begin{tikzpicture}[scale=0.4]
   			 \coordinate (v1) at (0,4);
			 \coordinate (v2) at (4,-4);
			 \coordinate (v3) at (-4,-4);
			 \coordinate (v4) at ($-.2*(v1)+1.2*(0,-1)-(0,.5)$);
   			 \coordinate (v5) at ($-.2*(v2)+1.2*(0,-1)-(0,.5)$);
			 \coordinate (v6) at ($-.2*(v3)+1.2*(0,-1)-(0,.5)$);
			 \draw[very thick] (v1)--(v2)--(v3)--(v1);
			 \draw[very thick] (v5)--(v6)--(v4)--(v5);
			 \draw[very thick] (v1)--(v5);
			 \draw[very thick] (v1)--(v6);
			 \draw[very thick] (v2)--(v4);
			 \draw[very thick] (v2)--(v6);
			 \draw[very thick] (v3)--(v4);
			 \draw[very thick] (v3)--(v5);

			 \coordinate (v1_1_1) at ($.6*(v1)+.1*(v5)+.3*(v6)$);
			 \coordinate (v1_1_2) at ($.7*(v1)+.1*(v5)+.2*(v6)$);
			 \coordinate (v1_1_3) at ($.7*(v1)+.2*(v5)+.1*(v6)$);
			 \coordinate (v1_1_4) at ($.6*(v1)+.3*(v5)+.1*(v6)$);	 
			 \draw[red,thick] (v1_1_1)--(v1_1_2)--(v1_1_3)--(v1_1_4) ;
			 
 			 \coordinate (v1_2_1) at ($.6*(v1)+.1*(v6)+.3*(v5)+.2*(v5)-.2*(v6)$);
			 \coordinate (v1_2_2) at ($.7*(v1)+.1*(v6)+.2*(v5)+.2*(v5)-.2*(v6)$);
			 \coordinate (v1_2_3) at ($.7*(v1)+.1*(v6)+.2*(v5)+.3*(v5)-.3*(v6)$);
			 \coordinate (v1_2_4) at ($.7*(v1)+.1*(v6)+.2*(v5)+.3*(v5)-.3*(v6)+.07*(v3)-.07*(v1)$);
		 	 \coordinate (v1_2) at (intersection of v1_2_1--v1_2_2 and v1_2_3--v1_2_4);		 
			 \draw[blue,thick] (v1_2_1) --(v1_2) --(v1_2_4) ;

 			 \coordinate (v1_3_1) at ($.6*(v1)+.1*(v5)+.3*(v6)+.2*(v6)-.2*(v5)$);
			 \coordinate (v1_3_2) at ($.7*(v1)+.1*(v5)+.2*(v6)+.2*(v6)-.2*(v5)$);
			 \coordinate (v1_3_3) at ($.7*(v1)+.1*(v5)+.2*(v6)+.3*(v6)-.3*(v5)$);
			 \coordinate (v1_3_4) at ($.7*(v1)+.1*(v5)+.2*(v6)+.3*(v6)-.3*(v5)+.07*(v2)-.07*(v1)$);
		 	 \coordinate (v1_3) at (intersection of v1_3_1--v1_3_2 and v1_3_3--v1_3_4);		 
			 \draw[blue,thick] (v1_3_1) --(v1_3) --(v1_3_4) ;			 
			 
			 \coordinate (v1_4_1) at ($1.1*(v1)+.1*(v5)+.2*(v6)+.1*(v2)-.1*(v1)-.1*(v5)-.1*(v6)$);
			 \coordinate (v1_4_2) at ($1.1*(v1)+.1*(v5)+.2*(v6)-.1*(v5)-.1*(v6)$);
			 \coordinate (v1_4_3) at ($1.1*(v1)+.2*(v5)+.1*(v6)-.1*(v5)-.1*(v6)$);
			 \coordinate (v1_4_4) at ($1.1*(v1)+.2*(v5)+.1*(v6)+.1*(v3)-.1*(v1)-.1*(v5)-.1*(v6)$);	 
			 \draw[red,thick] (v1_4_1)--(v1_4_2)--(v1_4_3)--(v1_4_4);
			 
			 \coordinate (v2_1_1) at ($.6*(v2)+.1*(v4)+.3*(v6)$);
			 \coordinate (v2_1_2) at ($.7*(v2)+.1*(v4)+.2*(v6)$);
			 \coordinate (v2_1_3) at ($.7*(v2)+.2*(v4)+.1*(v6)$);
			 \coordinate (v2_1_4) at ($.6*(v2)+.3*(v4)+.1*(v6)$);
			 \coordinate (v2_1) at (intersection of v2_1_1--v2_1_2 and v2_1_3--v2_1_4);	
			 \draw[blue,thick] (v2_1_1)--(v2_1)--(v2_1_4) ;
			 
 			 \coordinate (v2_2_1) at ($.6*(v2)+.1*(v6)+.3*(v4)+.2*(v4)-.2*(v6)$);
			 \coordinate (v2_2_2) at ($.7*(v2)+.1*(v6)+.2*(v4)+.2*(v4)-.2*(v6)$);
			 \coordinate (v2_2_3) at ($.7*(v2)+.1*(v6)+.2*(v4)+.3*(v4)-.3*(v6)$);
			 \coordinate (v2_2_4) at ($.7*(v2)+.1*(v6)+.2*(v4)+.3*(v4)-.3*(v6)+.07*(v3)-.07*(v2)$);
			 \coordinate (v2_2) at (intersection of v2_2_1--v2_2_2 and v2_2_3--v2_2_4);	
			 \draw[blue,thick] (v2_2_1)--(v2_2)--(v2_2_4) ;

 			 \coordinate (v2_3_1) at ($.6*(v2)+.1*(v4)+.3*(v6)+.2*(v6)-.2*(v4)$);
			 \coordinate (v2_3_2) at ($.7*(v2)+.1*(v4)+.2*(v6)+.2*(v6)-.2*(v4)$);
			 \coordinate (v2_3_3) at ($.7*(v2)+.1*(v4)+.2*(v6)+.3*(v6)-.3*(v4)$);
			 \coordinate (v2_3_4) at ($.7*(v2)+.1*(v4)+.2*(v6)+.3*(v6)-.3*(v4)+.07*(v1)-.07*(v2)$);
			 \coordinate (v2_3) at (intersection of v2_3_1--v2_3_2 and v2_3_3--v2_3_4);	
			 \draw[blue,thick] (v2_3_1)--(v2_3)--(v2_3_4) ;			 
			 
			 \coordinate (v2_4_1) at ($.7*(v2)+.1*(v6)+.2*(v4)-.1*(v4)+.1*(v6)+.4*(v2)-.2*(v6)-0.2*(v4)+.1*(v1)-.1*(v2)$);
			 \coordinate (v2_4_2) at ($.7*(v2)+.1*(v6)+.2*(v4)-.1*(v4)+.1*(v6)+.4*(v2)-.2*(v6)-0.2*(v4)$);
			 \coordinate (v2_4_3) at ($.7*(v2)+.1*(v6)+.2*(v4)+.01*(v4)-.01*(v6)+.4*(v2)-.2*(v6)-0.2*(v4)$);
			 \coordinate (v2_4_4) at ($.7*(v2)+.1*(v6)+.2*(v4)+.01*(v4)-.01*(v6)+.4*(v2)-.2*(v6)-0.2*(v4)+.1*(v3)-.1*(v2)$);
			 \coordinate (v2_4) at (intersection of v2_4_1--v2_4_2 and v2_4_3--v2_4_4);	
			 \draw[blue,thick] (v2_4_1)--(v2_4)--(v2_4_4) ;
			 
			 \coordinate (v3_1_1) at ($.6*(v3)+.1*(v4)+.3*(v5)$);
			 \coordinate (v3_1_2) at ($.7*(v3)+.1*(v4)+.2*(v5)$);
			 \coordinate (v3_1_3) at ($.7*(v3)+.2*(v4)+.1*(v5)$);
			 \coordinate (v3_1_4) at ($.6*(v3)+.3*(v4)+.1*(v5)$);
			 \coordinate (v3_1) at (intersection of v3_1_1--v3_1_2 and v3_1_3--v3_1_4);	
			 \draw[blue,thick] (v3_1_1)--(v3_1)--(v3_1_4) ;
			 
 			 \coordinate (v3_2_1) at ($.6*(v3)+.1*(v5)+.3*(v4)+.2*(v4)-.2*(v5)$);
			 \coordinate (v3_2_2) at ($.7*(v3)+.1*(v5)+.2*(v4)+.2*(v4)-.2*(v5)$);
			 \coordinate (v3_2_3) at ($.7*(v3)+.1*(v5)+.2*(v4)+.3*(v4)-.3*(v5)$);
			 \coordinate (v3_2_4) at ($.7*(v3)+.1*(v5)+.2*(v4)+.3*(v4)-.3*(v5)+.07*(v2)-.07*(v3)$);
			 \coordinate (v3_2) at (intersection of v3_2_1--v3_2_2 and v3_2_3--v3_2_4);	
			 \draw[blue,thick] (v3_2_1)--(v3_2)--(v3_2_4) ;

 			 \coordinate (v3_3_1) at ($.6*(v3)+.1*(v4)+.3*(v5)+.2*(v5)-.2*(v4)$);
			 \coordinate (v3_3_2) at ($.7*(v3)+.1*(v4)+.2*(v5)+.2*(v5)-.2*(v4)$);
			 \coordinate (v3_3_3) at ($.7*(v3)+.1*(v4)+.2*(v5)+.3*(v5)-.3*(v4)$);
			 \coordinate (v3_3_4) at ($.7*(v3)+.1*(v4)+.2*(v5)+.3*(v5)-.3*(v4)+.07*(v1)-.07*(v3)$);
			 \coordinate (v3_3) at (intersection of v3_3_1--v3_3_2 and v3_3_3--v3_3_4);	
			 \draw[blue,thick] (v3_3_1)--(v3_3)--(v3_3_4) ;			 
			 
			 \coordinate (v3_4_1) at ($.7*(v3)+.1*(v5)+.2*(v4)-.1*(v4)+.1*(v5)+.4*(v3)-.2*(v5)-0.2*(v4)+.1*(v1)-.1*(v3)$);
			 \coordinate (v3_4_2) at ($.7*(v3)+.1*(v5)+.2*(v4)-.1*(v4)+.1*(v5)+.4*(v3)-.2*(v5)-0.2*(v4)$);
			 \coordinate (v3_4_3) at ($.7*(v3)+.1*(v5)+.2*(v4)+.01*(v4)-.01*(v5)+.4*(v3)-.2*(v5)-0.2*(v4)$);
			 \coordinate (v3_4_4) at ($.7*(v3)+.1*(v5)+.2*(v4)+.01*(v4)-.01*(v5)+.4*(v3)-.2*(v5)-0.2*(v4)+.1*(v2)-.1*(v3)$);
			 \coordinate (v3_4) at (intersection of v3_4_1--v3_4_2 and v3_4_3--v3_4_4);	
			 \draw[blue,thick] (v3_4_1)--(v3_4)--(v3_4_4) ;			 
			 
			 \coordinate (v4_1_1) at ($.5*(v4)+.1*(v5)+.4*(v6)$);
			 \coordinate (v4_1_2) at ($.7*(v4)+.1*(v5)+.2*(v6)$);
			 \coordinate (v4_1_3) at ($.7*(v4)+.2*(v5)+.1*(v6)$);
			 \coordinate (v4_1_4) at ($.5*(v4)+.4*(v5)+.1*(v6)$);
			 \draw[red,thick] (v4_1_1) --(v4_1_2) --(v4_1_3) --(v4_1_4) ;

			 \coordinate (v4_2_1) at ($.7*(v4)+.1*(v6)+.2*(v2)$);
			 \coordinate (v4_2_2) at ($.8*(v4)+.1*(v6)+.1*(v2)$);
			 \coordinate (v4_2_3) at ($.7*(v4)+.25*(v6)+.05*(v2)$);
			 \coordinate (v4_2_4) at ($.5*(v4)+.45*(v6)+.05*(v2)$);
			 \coordinate (v4_2) at (intersection of v4_2_1--v4_2_2 and v4_2_3--v4_2_4);		 
			 \draw[blue,thick] (v4_2_1) --(v4_2) --(v4_2_4) ;

			 \coordinate (v4_3_1) at ($.7*(v4)+.1*(v5)+.2*(v3)$);
			 \coordinate (v4_3_2) at ($.8*(v4)+.1*(v5)+.1*(v3)$);
			 \coordinate (v4_3_3) at ($.7*(v4)+.25*(v5)+.05*(v3)$);
			 \coordinate (v4_3_4) at ($.5*(v4)+.45*(v5)+.05*(v3)$);
			 \coordinate (v4_3) at (intersection of v4_3_1--v4_3_2 and v4_3_3--v4_3_4);		 
			 \draw[blue,thick] (v4_3_1) --(v4_3) --(v4_3_4) ;

			 \coordinate (v4_4_1) at ($.72*(v4)+.06*(v2)+.23*(v3)$);
			 \coordinate (v4_4_2) at ($.84*(v4)+.06*(v2)+.11*(v3)$);
			 \coordinate (v4_4_3) at ($.84*(v4)+.11*(v2)+.06*(v3)$);
			 \coordinate (v4_4_4) at ($.72*(v4)+.23*(v2)+.06*(v3)$);
			 \draw[red,thick] (v4_4_1) --(v4_4_2) --(v4_4_3) --(v4_4_4) ;

			 \coordinate (v5_1_1) at ($.5*(v5)+.1*(v4)+.4*(v6)$);
			 \coordinate (v5_1_2) at ($.7*(v5)+.1*(v4)+.2*(v6)$);
			 \coordinate (v5_1_3) at ($.7*(v5)+.2*(v4)+.1*(v6)$);
			 \coordinate (v5_1_4) at ($.5*(v5)+.4*(v4)+.1*(v6)$);
			 \coordinate (v5_1) at (intersection of v5_1_1--v5_1_2 and v5_1_3--v5_1_4);
			 \draw[blue,thick] (v5_1_1) --(v5_1)--(v5_1_4);
			 
			 \coordinate (v5_2_1) at ($.7*(v5)+.1*(v4)+.2*(v3)$);
			 \coordinate (v5_2_2) at ($.8*(v5)+.1*(v4)+.1*(v3)$);
			 \coordinate (v5_2_3) at ($.6*(v5)+.23*(v4)+.07*(v3)$);
			 \coordinate (v5_2_4) at ($.4*(v5)+.43*(v4)+.07*(v3)$);
			 \coordinate (v5_2) at (intersection of v5_2_1--v5_2_2 and v5_2_3--v5_2_4);
			 \draw[blue,thick] (v5_2_1) --(v5_2)--(v5_2_4);
			 
			 \coordinate (v5_3_1) at ($.72*(v5)+.06*(v1)+.23*(v3)$);
			 \coordinate (v5_3_2) at ($.84*(v5)+.06*(v1)+.11*(v3)$);
			 \coordinate (v5_3_3) at ($.84*(v5)+.11*(v1)+.06*(v3)$);
			 \coordinate (v5_3_4) at ($.72*(v5)+.23*(v1)+.06*(v3)$);
			 \coordinate (v5_3) at (intersection of v5_3_1--v5_3_2 and v5_3_3--v5_3_4);
			 \draw[blue,thick] (v5_3_1) --(v5_3)--(v5_3_4);
			 
			 \coordinate (v5_4_1) at ($.7*(v5)+.1*(v6)+.2*(v1)$);
			 \coordinate (v5_4_2) at ($.8*(v5)+.1*(v6)+.1*(v1)$);
			 \coordinate (v5_4_3) at ($.7*(v5)+.25*(v6)+.05*(v1)$);
			 \coordinate (v5_4_4) at ($.5*(v5)+.45*(v6)+.05*(v1)$);
			 \coordinate (v5_4) at (intersection of v5_4_1--v5_4_2 and v5_4_3--v5_4_4);
			 \draw[blue,thick] (v5_4_1) --(v5_4)--(v5_4_4);
			 
			 \coordinate (v6_1_1) at ($.5*(v6)+.1*(v5)+.4*(v4)$);
			 \coordinate (v6_1_2) at ($.7*(v6)+.1*(v5)+.2*(v4)$);
			 \coordinate (v6_1_3) at ($.7*(v6)+.2*(v5)+.1*(v4)$);
			 \coordinate (v6_1_4) at ($.5*(v6)+.4*(v5)+.1*(v4)$);
			 \coordinate (v6_1) at (intersection of v6_1_1--v6_1_2 and v6_1_3--v6_1_4);
			 \draw[blue,thick] (v6_1_1) --(v6_1)--(v6_1_4) ;
			 
			 \coordinate (v6_2_1) at ($.7*(v6)+.1*(v4)+.2*(v2)$);
			 \coordinate (v6_2_2) at ($.8*(v6)+.1*(v4)+.1*(v2)$);
			 \coordinate (v6_2_3) at ($.6*(v6)+.23*(v4)+.07*(v2)$);
			 \coordinate (v6_2_4) at ($.4*(v6)+.43*(v4)+.07*(v2)$);
			 \coordinate (v6_2) at (intersection of v6_2_1--v6_2_2 and v6_2_3--v6_2_4);
			 \draw[blue,thick] (v6_2_1) --(v6_2)--(v6_2_4) ;
			 
			 \coordinate (v6_3_1) at ($.72*(v6)+.06*(v1)+.23*(v2)$);
			 \coordinate (v6_3_2) at ($.84*(v6)+.06*(v1)+.11*(v2)$);
			 \coordinate (v6_3_3) at ($.84*(v6)+.11*(v1)+.06*(v2)$);
			 \coordinate (v6_3_4) at ($.72*(v6)+.23*(v1)+.06*(v2)$);
			 \coordinate (v6_3) at (intersection of v6_3_1--v6_3_2 and v6_3_3--v6_3_4);
			 \draw[blue,thick] (v6_3_1) --(v6_3)--(v6_3_4) ;
			 
			 \coordinate (v6_4_1) at ($.7*(v6)+.1*(v5)+.2*(v1)$);
			 \coordinate (v6_4_2) at ($.8*(v6)+.1*(v5)+.1*(v1)$);
			 \coordinate (v6_4_3) at ($.7*(v6)+.25*(v5)+.05*(v1)$);
			 \coordinate (v6_4_4) at ($.5*(v6)+.45*(v5)+.05*(v1)$);
			 \coordinate (v6_4) at (intersection of v6_4_1--v6_4_2 and v6_4_3--v6_4_4);
			 \draw[blue,thick] (v6_4_1) --(v6_4) --(v6_4_4) ;

   		    \end{tikzpicture}
		    \subcaption{}
		    \label{subfig:sec4_octahedron_case_square_pyramid}
	\end{subfigure}
	\end{center}
	\caption{}
	\label{fig:sec4_octahedron_label}
\end{figure}

\begin{figure}[h!]	
	  \begin{center} 	  	
	\begin{subfigure}{.4\linewidth}
		    \begin{tikzpicture}[xscale=.8, yscale=.5, x={(1cm,0cm)},y={(-.4cm,-0.6cm)}, z={(0cm,2cm)}]
			 \coordinate (v1) at (0, 0, 0);
			 \coordinate (v2) at (0, 3, 0);
			 \coordinate (v3) at (3, 3, 0);
			 \coordinate (v4) at (3, 0, 0);
			 \coordinate (v5) at ($.5*(v1)+.5*(v3)+(0,0,1.5)$);
   			 \coordinate (v6) at  ($.5*(v1)+.5*(v3)-(0,0,1.5)$);
			 \draw[very thick] (v1)--(v2)--(v3)--(v4)--(v1);
			 \draw[very thick] (v5)--(v1);
			 \draw[very thick] (v5)--(v2);
			 \draw[very thick] (v5)--(v3);
			 \draw[very thick] (v5)--(v4);
			 \draw[very thick] (v6)--(v1);
			 \draw[very thick] (v6)--(v2);
			 \draw[very thick] (v6)--(v3);
			 \draw[very thick] (v6)--(v4);
			 \coordinate (u) at ($.25*(v1)+.25*(v2)+.15*(v3)+.15*(v4)+.2*(v5)$);
   			 \coordinate (w) at  ($.25*(v3)+.25*(v4)+.15*(v1)+.15*(v2)+.2*(v5)$);
			 \coordinate (s) at ($.3*(v1)+.3*(v4)+.05*(v2)+.05*(v3)+.3*(v6)$);
   			 \coordinate (t) at  ($.3*(v2)+.3*(v3)+.05*(v1)+.05*(v4)+.3*(v6)$);
			 \draw[red] (u)--(w);
			 \draw[red] (u)--(s);
			 \draw[red] (u)--(t);
			 \draw[red] (w)--(s);
			 \draw[red] (w)--(t);
			 \draw[red] (s)--(t);
			 \draw[blue] (u)--(v5);
			 \draw[blue] (w)--(v5);
			 \draw[blue] (s)--(v6);
			 \draw[blue] (t)--(v6);
			 \draw[blue] (u)--(v1);
			 \draw[blue] (u)--(v2);
			 \draw[blue] (w)--(v3);
			 \draw[blue] (w)--(v4);
			 \draw[blue] (s)--(v1);
			 \draw[blue] (s)--(v4); 
			 \draw[blue] (t)--(v3);
			 \draw[blue] (t)--(v2); 
		   \end{tikzpicture}
		    \subcaption{}
		    \label{fig:sec4_octahedron_result1}
	\end{subfigure}  
	  	\begin{subfigure}{.4\linewidth}
 		    \begin{tikzpicture}[xscale=.8, yscale=.5, x={(1cm,0cm)},y={(-.4cm,-0.6cm)}, z={(0cm,2cm)}]
			 \coordinate (v1) at (0, 0, 0);
			 \coordinate (v2) at (0, 3, 0);
			 \coordinate (v3) at (3, 3, 0);
			 \coordinate (v4) at (3, 0, 0);
			 \coordinate (v5) at ($.5*(v1)+.5*(v3)+(0,0,1.5)$);
   			 \coordinate (v6) at  ($.5*(v1)+.5*(v3)-(0,0,1.5)$);
			 \draw[very thick] (v1)--(v2)--(v3)--(v4)--(v1);
			 \draw[very thick] (v5)--(v1);
			 \draw[very thick] (v5)--(v2);
			 \draw[very thick] (v5)--(v3);
			 \draw[very thick] (v5)--(v4);
			 \draw[very thick] (v6)--(v1);
			 \draw[very thick] (v6)--(v2);
			 \draw[very thick] (v6)--(v3);
			 \draw[very thick] (v6)--(v4);
			 \coordinate (u) at ($.3*(v1)+.3*(v2)+.2*(v3)+.2*(v4)$);
   			 \coordinate (w) at  ($.3*(v3)+.3*(v4)+.2*(v1)+.2*(v2)$);
			 \draw[red] (u)--(w);
			 \draw[blue] (u)--(v5);
			 \draw[blue] (u)--(v6);
			 \draw[blue] (u)--(v1);
			 \draw[blue] (u)--(v2);
			 \draw[blue] (w)--(v5);
			 \draw[blue] (w)--(v6);
			 \draw[blue] (w)--(v3);
			 \draw[blue] (w)--(v4);
		   \end{tikzpicture}   
		    \subcaption{}
		    \label{fig:sec4_octahedron_result2}
	\end{subfigure}  
	\caption{}
	\label{fig:sec4_octahedron_result}
	\end{center}
\end{figure}

The other option left to explore is if $F_1$ is not an octahedron, square pyramid or a bisimplex. Without loss of generality we may then assume that all facets of $P$ that intersect $F$ in a facet are simplices, otherwise we would consider a non-simplex facet as the new $F_1$ and be in one of the previous cases. However, this case immediately gives a pyramid over an octahedron, a case that is already accounted for, as it is the dual of the pyramid over the cube, namely class $16$. Therefore we only get one new class of psd-minimal polytopes which is class $25$, and its dual.
	
\subsection{Bisimplex}

Suppose now that $P$ has a bisimplicial facet $F$. We can assume that neither $P$ nor its dual has a cube, triangular prism or octahedron as a facet. In particular, every vertex of $P$ is contained in at most five facets.
\begin{figure}
  		    \begin{tikzpicture}[yscale=0.4, xscale =0.7]
   			 \coordinate (v1) at (0,4.5);
			 \coordinate (v2) at (4,-4);
			 \coordinate (v3) at (-4,-4);
			 \coordinate (v5) at (0, 1);
   			 \coordinate (v4) at (0,-2);
			 \coordinate (v6) at ($-.2*(v3)+1.2*(0,-1)-(0,.5)$);
			 \draw[very thick] (v1)--(v2)--(v3)--(v1);
			 \draw[very thick] (v5)--(v1);
			 \draw[very thick] (v4)--(v2);
			 \draw[very thick] (v4)--(v3);
			 \draw[very thick] (v5)--(v2);
			 \draw[very thick] (v5)--(v3);
			 \draw[very thick] (v5)--(v4);
			 \node at (-1,0.9) {$F'_1$};
			 \node at (1,0.9) {$F'_2$};
			 \node at (-1,-1.5) {$F'_3$};
			 \node at (1,-1.5) {$F'_4$};
			 \node[label=$v$] at ($(v5)+(.3,-.4)$) {};
		\end{tikzpicture}
		\caption{}	     
	        \label{fig:sec4_bisimplex_marked_facets}	 
\end{figure}

Consider the graph of $F$ with a marked vertex of degree four (with respect to the graph of $F$) and four marked facets of $F$ as shown in Figure~\ref{fig:sec4_bisimplex_marked_facets}. Let $F_1$, $F_2$, $F_3$ and $F_4$ be the facets of $P$ such that $F'_1=F_1\cap F$, $F'_2=F_2\cap F$, $F'_3=F_3\cap F$ and $F'_4=F_4\cap F$. Clearly, the vertex $v$ is contained in at least five facets of $P$, namely in $F_1$, $F_2$, $F_3$, $F_4$ and $F$. Hence, there is no other facet of $P$ containing $v$, otherwise $v$ is contained in at least six facets of $P$. Additionally note that every edge of $F$ contains a vertex of degree four. By these observations, we can conclude that no facet of $P$ intersects $F$ in an edge or in a vertex of degree four.

This strengthened version of Proposition~\ref{prop:intersections} allows us to apply all the labeling rules in Figure~\ref{fig:sec4_case_octahedron_rules} to the bisimplicial case as well. Note that the graph of $F$ contains also vertices of degree three. 

For these vertices we introduce the two additional rules in Figure~\ref{fig:sec4_case_bisimplex_rules} 
which are saying that we cannot mix blue angles and red cut angles around a degree three vertex of $F$. 
Suppose we have a red cut angle $(v,F_1')$ and two blue angles. Then the two edges of $F_1$ that contain $v$ and are not contained in $F$ have to be the same which is a contradiction. If there were two red cut angles $(v,F_1')$ and $(v,F_2')$ they share an edge that contains $v$ and is not contained in $F$ and the remaining such edge of each would both get identified with the only such edge of the blue angle face, again a contradiction.

\begin{figure}	
	  \begin{center}
	   \begin{subfigure}{0.4\linewidth}
  		    \begin{tikzpicture}[scale=0.17]
   			 \coordinate (v1) at (0,0);
			 \coordinate (v2) at (0,4);
			 \coordinate (v3) at (4,-3);
			 \coordinate (v4) at (-4,-3);
			 \draw[very thick] (v2)--(v1)--(v3);
			 \draw[very thick] (v1)--(v4);
			 
			 \coordinate (v34_1) at ($.5*(v1)+.1*(v4)+.4*(v3)$);
			 \coordinate (v34_2) at ($.7*(v1)+.1*(v4)+.2*(v3)$);
			 \coordinate (v34_3) at ($.7*(v1)+.2*(v4)+.1*(v3)$);
			 \coordinate (v34_4) at ($.5*(v1)+.4*(v4)+.1*(v3)$);
		 	 \coordinate (w34) at (intersection of v34_1--v34_2 and v34_3--v34_4);		 
			 \draw[red,thick] (v34_1)--(v34_2)--(v34_3)--(v34_4) ; 
			 
			 \draw[->,double, >=angle 90] (5.5,0) --(7.5,0);

			\begin{scope}[xshift=13cm]
   			 \coordinate (v1) at (0,0);
			 \coordinate (v2) at (0,4);
			 \coordinate (v3) at (4,-3);
			 \coordinate (v4) at (-4,-3);
			 \draw[very thick] (v2)--(v1)--(v3);
			 \draw[very thick] (v1)--(v4);
			 
			 \coordinate (v34_1) at ($.5*(v1)+.1*(v4)+.4*(v3)$);
			 \coordinate (v34_2) at ($.7*(v1)+.1*(v4)+.2*(v3)$);
			 \coordinate (v34_3) at ($.7*(v1)+.2*(v4)+.1*(v3)$);
			 \coordinate (v34_4) at ($.5*(v1)+.4*(v4)+.1*(v3)$);
		 	 \coordinate (w34) at (intersection of v34_1--v34_2 and v34_3--v34_4);		 
			 \draw[red,thick] (v34_1)--(v34_2)--(v34_3)--(v34_4) ; 
				 
			 \coordinate (v32_1) at ($.4*(v1)+.15*(v2)+.45*(v3)$);
			 \coordinate (v32_2) at ($.6*(v1)+.15*(v2)+.25*(v3)$);
			 \coordinate (v32_3) at ($.6*(v1)+.25*(v2)+.15*(v3)$);
			 \coordinate (v32_4) at ($.4*(v1)+.45*(v2)+.15*(v3)$);
		 	 \coordinate (w32) at (intersection of v32_1--v32_2 and v32_3--v32_4);
			 \draw[red,thick] (v32_1)--(v32_2)--(v32_3)--(v32_4) ;	
			 
			 \coordinate (v42_1) at ($.4*(v1)+.15*(v2)+.45*(v4)$);
			 \coordinate (v42_2) at ($.6*(v1)+.15*(v2)+.25*(v4)$);
			 \coordinate (v42_3) at ($.6*(v1)+.25*(v2)+.15*(v4)$);
			 \coordinate (v42_4) at ($.4*(v1)+.45*(v2)+.15*(v4)$);
		 	 \coordinate (w42) at (intersection of v42_1--v42_2 and v42_3--v42_4);
			 \draw[red,thick] (v42_1)--(v42_2)--(v42_3)--(v42_4) ;	   
			 
			\end{scope}
	 			 	
   		    \end{tikzpicture}
		    \subcaption{}
		    \label{subfig:sec4_labeling_rule4}
	\end{subfigure}
	   \begin{subfigure}{0.4\linewidth}
  		    \begin{tikzpicture}[scale=0.17]
   			 \coordinate (v1) at (0,0);
			 \coordinate (v2) at (0,4);
			 \coordinate (v3) at (4,-3);
			 \coordinate (v4) at (-4,-3);
			 \draw[very thick] (v2)--(v1)--(v3);
			 \draw[very thick] (v1)--(v4);
			 
			 \coordinate (v34_1) at ($.5*(v1)+.1*(v4)+.4*(v3)$);
			 \coordinate (v34_2) at ($.7*(v1)+.1*(v4)+.2*(v3)$);
			 \coordinate (v34_3) at ($.7*(v1)+.2*(v4)+.1*(v3)$);
			 \coordinate (v34_4) at ($.5*(v1)+.4*(v4)+.1*(v3)$);
		 	 \coordinate (w34) at (intersection of v34_1--v34_2 and v34_3--v34_4);		 
			 \draw[blue,thick] (v34_1)--(w34)--(v34_4) ; 
			 
           		 \draw[->,double, >=angle 90] (5.5,0) --(7.5,0);

			\begin{scope}[xshift=13cm]
   			 \coordinate (v1) at (0,0);
			 \coordinate (v2) at (0,4);
			 \coordinate (v3) at (4,-3);
			 \coordinate (v4) at (-4,-3);
			 \draw[very thick] (v2)--(v1)--(v3);
			 \draw[very thick] (v1)--(v4);
			 
			 \coordinate (v34_1) at ($.5*(v1)+.1*(v4)+.4*(v3)$);
			 \coordinate (v34_2) at ($.7*(v1)+.1*(v4)+.2*(v3)$);
			 \coordinate (v34_3) at ($.7*(v1)+.2*(v4)+.1*(v3)$);
			 \coordinate (v34_4) at ($.5*(v1)+.4*(v4)+.1*(v3)$);
		 	 \coordinate (w34) at (intersection of v34_1--v34_2 and v34_3--v34_4);		 
			 \draw[blue,thick] (v34_1)--(w34)--(v34_4) ; 
				 
			 \coordinate (v32_1) at ($.4*(v1)+.15*(v2)+.45*(v3)$);
			 \coordinate (v32_2) at ($.6*(v1)+.15*(v2)+.25*(v3)$);
			 \coordinate (v32_3) at ($.6*(v1)+.25*(v2)+.15*(v3)$);
			 \coordinate (v32_4) at ($.4*(v1)+.45*(v2)+.15*(v3)$);
		 	 \coordinate (w32) at (intersection of v32_1--v32_2 and v32_3--v32_4);
			 \draw[blue,thick] (v32_1)--(w32)--(v32_4) ;
			 
			 \coordinate (v42_1) at ($.4*(v1)+.15*(v2)+.45*(v4)$);
			 \coordinate (v42_2) at ($.6*(v1)+.15*(v2)+.25*(v4)$);
			 \coordinate (v42_3) at ($.6*(v1)+.25*(v2)+.15*(v4)$);
			 \coordinate (v42_4) at ($.4*(v1)+.45*(v2)+.15*(v4)$);
		 	 \coordinate (w42) at (intersection of v42_1--v42_2 and v42_3--v42_4);
			  \draw[blue,thick] (v42_1)--(w42)--(v42_4) ;	   
			 
			\end{scope}
	 			 	
   		    \end{tikzpicture}
		    \subcaption{}
		    \label{subfig:sec4_labeling_rule5}
	\end{subfigure}
	\caption{}
	\label{fig:sec4_case_bisimplex_rules}
	\end{center}
\end{figure}

We obtain four valid diagrams up to symmetry, by considering all the different possibilities for the bottom 
triangle. The first is gotten by making all other facets simplices, which gives us the pyramid over a bisimplex. This is however dual to the pyramid over a prism,
which is already accounted for in class $8$. The remaining three possible diagrams for the graph of $F$ are shown in Figure~\ref{fig:sec4_bisimplex_label}.

\begin{figure}[h!]	
	  \begin{center}
	  	\begin{subfigure}{.3\linewidth}
  		    \begin{tikzpicture}[scale=0.4]
   			 \coordinate (v1) at (0,4);
			 \coordinate (v2) at (4,-4);
			 \coordinate (v3) at (-4,-4);
			 \coordinate (v5) at (0,0);
   			 \coordinate (v4) at (0,-2);
			 \coordinate (v6) at ($-.2*(v3)+1.2*(0,-1)-(0,.5)$);
			 \draw[very thick] (v1)--(v2)--(v3)--(v1);
			 \draw[very thick] (v5)--(v1);
			 \draw[very thick] (v4)--(v2);
			 \draw[very thick] (v4)--(v3);
			 \draw[very thick] (v5)--(v2);
			 \draw[very thick] (v5)--(v3);
			 \draw[very thick] (v5)--(v4);
			 
			 \coordinate (v1_1_1) at ($.84*(v1)+.08*(v5)+.08*(v3)-.02*(v3)+.02*(v5)+.1*(v3)-.1*(v1)$);
			 \coordinate (v1_1_2) at ($.84*(v1)+.08*(v5)+.08*(v3)-.02*(v3)-+02*(v5)$);
			 \coordinate (v1_1_3) at ($.8*(v1)+.1*(v5)+.1*(v3)-.07*(v3)+.07*(v5)$);
			 \coordinate (v1_1_4) at ($.8*(v1)+.1*(v5)+.1*(v3)-.07*(v3)+.07*(v5)+.2*(v5)-.2*(v1)$);
			 \draw[red,thick] (v1_1_1) --(v1_1_2) --(v1_1_3) --(v1_1_4) ;
			 
			 \coordinate (v1_2_1) at ($.84*(v1)+.08*(v5)+.08*(v2)-.02*(v2)+.02*(v5)+.1*(v2)-.1*(v1)$);
			 \coordinate (v1_2_2) at ($.84*(v1)+.08*(v5)+.08*(v2)-.02*(v2)-+02*(v5)$);
			 \coordinate (v1_2_3) at ($.8*(v1)+.1*(v5)+.1*(v2)-.07*(v2)+.07*(v5)$);
			 \coordinate (v1_2_4) at ($.8*(v1)+.1*(v5)+.1*(v2)-.07*(v2)+.07*(v5)+.2*(v5)-.2*(v1)$);
			 \draw[red,thick] (v1_2_1) --(v1_2_2) --(v1_2_3) --(v1_2_4) ;
			 
			 \coordinate (v1_3_1) at ($1.04*(v1)-.02*(v3)-.02*(v2)-.015*(v2)+.015*(v3)+.1*(v3)-.1*(v1)$);
			 \coordinate (v1_3_2) at ($1.04*(v1)-.02*(v3)-.02*(v2)-.015*(v2)+.015*(v3)$);
			 \coordinate (v1_3_3) at ($1.04*(v1)-.02*(v3)-.02*(v2)+.015*(v2)-.015*(v3)$);
			 \coordinate (v1_3_4) at ($1.04*(v1)-.02*(v3)-.02*(v2)+.015*(v2)-.015*(v3)+.1*(v2)-.1*(v1)$);
			 \draw[red,thick] (v1_3_1)--(v1_3_2)--(v1_3_3)--(v1_3_4);

			 \coordinate (v2_1_1) at ($.7*(v2)+.1*(v5)+.2*(v4)-.015*(v3)+.015*(v1)+.07*(v5)-.07*(v2)$);
			 \coordinate (v2_1_2) at ($.7*(v2)+.1*(v5)+.2*(v4)-.015*(v3)+.015*(v1)$);
			 \coordinate (v2_1_3) at ($.7*(v2)+.1*(v5)+.2*(v4)+.005*(v3)-.005*(v1)$);
			 \coordinate (v2_1_4) at ($.7*(v2)+.1*(v5)+.2*(v4)+.005*(v3)-.005*(v1)+.07*(v4)-.07*(v2)$);
			 \draw[red,thick] (v2_1_1) --(v2_1_2) --(v2_1_3) --(v2_1_4) ;
			 
 			 \coordinate (v2_2_1) at ($.7*(v2)+.1*(v5)+.2*(v4)+.04*(v3)-.04*(v1)+.07*(v4)-.07*(v2)$);
			 \coordinate (v2_2_2) at ($.7*(v2)+.1*(v5)+.2*(v4)+.04*(v3)-.04*(v1)$);
			 \coordinate (v2_2_3) at ($.7*(v2)+.1*(v5)+.2*(v4)+.08*(v3)-.08*(v1)$);
			 \coordinate (v2_2_4) at ($.7*(v2)+.1*(v5)+.2*(v4)+.08*(v3)-.08*(v1)+.07*(v3)-.07*(v2)$);
			 \coordinate (v2_2) at (intersection of v2_2_1--v2_2_2 and v2_2_3--v2_2_4);
			 \draw[blue,thick] (v2_2_1) --(v2_2)--(v2_2_4);

 			 \coordinate (v2_3_1) at ($.7*(v2)+.1*(v4)+.2*(v5)-.04*(v3)+.04*(v1)+.07*(v5)-.07*(v2)$);
			 \coordinate (v2_3_2) at ($.7*(v2)+.1*(v4)+.2*(v5)-.04*(v3)+.04*(v1)$);
			 \coordinate (v2_3_3) at ($.7*(v2)+.1*(v4)+.2*(v5)-.06*(v3)+.06*(v1)$);
			 \coordinate (v2_3_4) at ($.7*(v2)+.1*(v4)+.2*(v5)-.06*(v3)+.06*(v1)+.07*(v1)-.07*(v2)$);
			 \coordinate (v2_3) at (intersection of v2_3_1--v2_3_2 and v2_3_3--v2_3_4);
			 \draw[blue,thick] (v2_3_1) --(v2_3)--(v2_3_4);	
			 
			 \coordinate (v2_4_1) at ($.7*(v2)+.1*(v5)+.2*(v4)-.02*(v3)+.02*(v1)+.4*(v2)-.2*(v5)-0.2*(v4)+.1*(v1)-.1*(v2)$);
			 \coordinate (v2_4_2) at ($.7*(v2)+.1*(v5)+.2*(v4)-.02*(v3)+.02*(v1)+.4*(v2)-.2*(v5)-0.2*(v4)$);
			 \coordinate (v2_4_3) at ($.7*(v2)+.1*(v5)+.2*(v4)+.0*(v3)-.0*(v1)+.4*(v2)-.2*(v5)-0.2*(v4)$);
			 \coordinate (v2_4_4) at ($.7*(v2)+.1*(v5)+.2*(v4)+.0*(v3)-.0*(v1)+.4*(v2)-.2*(v5)-0.2*(v4)+.1*(v3)-.1*(v2)$);
			 \draw[red,thick] (v2_4_1) --(v2_4_2) --(v2_4_3) --(v2_4_4) ;

			 \coordinate (v3_1_1) at ($.7*(v3)+.1*(v5)+.2*(v4)-.015*(v2)+.015*(v1)+.07*(v5)-.07*(v3)$);
			 \coordinate (v3_1_2) at ($.7*(v3)+.1*(v5)+.2*(v4)-.015*(v2)+.015*(v1)$);
			 \coordinate (v3_1_3) at ($.7*(v3)+.1*(v5)+.2*(v4)+.005*(v2)-.005*(v1)$);
			 \coordinate (v3_1_4) at ($.7*(v3)+.1*(v5)+.2*(v4)+.005*(v2)-.005*(v1)+.07*(v4)-.07*(v3)$);
			 \coordinate (v3_1) at (intersection of v3_1_1--v3_1_2 and v3_1_3--v3_1_4);
			 \draw[blue,thick] (v3_1_1) --(v3_1)--(v3_1_4);
			 
 			 \coordinate (v3_2_1) at ($.7*(v3)+.1*(v5)+.2*(v4)+.04*(v2)-.04*(v1)+.07*(v4)-.07*(v3)$);
			 \coordinate (v3_2_2) at ($.7*(v3)+.1*(v5)+.2*(v4)+.04*(v2)-.04*(v1)$);
			 \coordinate (v3_2_3) at ($.7*(v3)+.1*(v5)+.2*(v4)+.08*(v2)-.08*(v1)$);
			 \coordinate (v3_2_4) at ($.7*(v3)+.1*(v5)+.2*(v4)+.08*(v2)-.08*(v1)+.07*(v2)-.07*(v3)$);
			 \draw[red,thick] (v3_2_1) --(v3_2_2) --(v3_2_3) --(v3_2_4) ;

 			 \coordinate (v3_3_1) at ($.7*(v3)+.1*(v4)+.2*(v5)-.04*(v2)+.04*(v1)+.07*(v5)-.07*(v3)$);
			 \coordinate (v3_3_2) at ($.7*(v3)+.1*(v4)+.2*(v5)-.04*(v2)+.04*(v1)$);
			 \coordinate (v3_3_3) at ($.7*(v3)+.1*(v4)+.2*(v5)-.06*(v2)+.06*(v1)$);
			 \coordinate (v3_3_4) at ($.7*(v3)+.1*(v4)+.2*(v5)-.06*(v2)+.06*(v1)+.07*(v1)-.07*(v3)$);
			 \draw[red,thick] (v3_3_1) --(v3_3_2) --(v3_3_3) --(v3_3_4) ;

			 \coordinate (v3_4_1) at ($.7*(v3)+.1*(v5)+.2*(v4)-.02*(v2)+.02*(v1)+.4*(v3)-.2*(v5)-0.2*(v4)+.1*(v1)-.1*(v3)$);
			 \coordinate (v3_4_2) at ($.7*(v3)+.1*(v5)+.2*(v4)-.02*(v2)+.02*(v1)+.4*(v3)-.2*(v5)-0.2*(v4)$);
			 \coordinate (v3_4_3) at ($.7*(v3)+.1*(v5)+.2*(v4)+.0*(v2)-.0*(v1)+.4*(v3)-.2*(v5)-0.2*(v4)$);
			 \coordinate (v3_4_4) at ($.7*(v3)+.1*(v5)+.2*(v4)+.0*(v2)-.0*(v1)+.4*(v3)-.2*(v5)-0.2*(v4)+.1*(v2)-.1*(v3)$);
			 \coordinate (v3_4) at (intersection of v3_4_1--v3_4_2 and v3_4_3--v3_4_4);
			 \draw[blue,thick] (v3_4_1) --(v3_4)--(v3_4_4);
			 
			 \coordinate (v4_1_1) at ($.72*(v4)+.06*(v2)+.23*(v3)$);
			 \coordinate (v4_1_2) at ($.84*(v4)+.06*(v2)+.11*(v3)$);
			 \coordinate (v4_1_3) at ($.84*(v4)+.11*(v2)+.06*(v3)$);
			 \coordinate (v4_1_4) at ($.72*(v4)+.23*(v2)+.06*(v3)$);
			 \draw[red,thick] (v4_1_1) --(v4_1_2) --(v4_1_3) --(v4_1_4) ;
			 
			 \coordinate (v4_2_1) at ($.66*(v4)+.03*(v2)+.32*(v5)$);
			 \coordinate (v4_2_2) at ($.8*(v4)+.03*(v2)+.18*(v5)$);
			 \coordinate (v4_2_3) at ($.84*(v4)+.08*(v2)+.09*(v5)$);
			 \coordinate (v4_2_4) at ($.71*(v4)+.19*(v2)+.09*(v5)$);
			 \draw[red,thick] (v4_2_1) --(v4_2_2) --(v4_2_3) --(v4_2_4) ;
			 
			 \coordinate (v4_3_1) at ($.66*(v4)+.03*(v3)+.32*(v5)$);
			 \coordinate (v4_3_2) at ($.8*(v4)+.03*(v3)+.18*(v5)$);
			 \coordinate (v4_3_3) at ($.84*(v4)+.08*(v3)+.09*(v5)$);
			 \coordinate (v4_3_4) at ($.71*(v4)+.19*(v3)+.09*(v5)$);
			 \draw[red,thick] (v4_3_1) --(v4_3_2) --(v4_3_3) --(v4_3_4) ;
			 
			 \coordinate (v5_1_1) at ($.72*(v5)+.04*(v2)+.25*(v1)$);
			 \coordinate (v5_1_2) at ($.84*(v5)+.04*(v2)+.13*(v1)$);
			 \coordinate (v5_1_3) at ($.85*(v5)+.08*(v2)+.08*(v1)$);
			 \coordinate (v5_1_4) at ($.73*(v5)+.2*(v2)+.08*(v1)$);
			 \draw[red,thick] (v5_1_1) --(v5_1_2) --(v5_1_3) --(v5_1_4) ;
			 
			 \coordinate (v5_2_1) at ($.6*(v5)+.03*(v2)+.46*(v4)$);
			 \coordinate (v5_2_2) at ($.83*(v5)+.03*(v2)+.23*(v4)$);
			 \coordinate (v5_2_3) at ($.81*(v5)+.07*(v2)+.14*(v4)$);
			 \coordinate (v5_2_4) at ($.63*(v5)+.2*(v2)+.14*(v4)$);
			 \coordinate (v5_2) at (intersection of v5_2_1--v5_2_2 and v5_2_3--v5_2_4);
			 \draw[blue,thick] (v5_2_1) --(v5_2)--(v5_2_4);
			 
			 \coordinate (v5_3_1) at ($.72*(v5)+.04*(v3)+.25*(v1)$);
			 \coordinate (v5_3_2) at ($.84*(v5)+.04*(v3)+.13*(v1)$);
			 \coordinate (v5_3_3) at ($.85*(v5)+.08*(v3)+.08*(v1)$);
			 \coordinate (v5_3_4) at ($.73*(v5)+.2*(v3)+.08*(v1)$);
			 \coordinate (v5_3) at (intersection of v5_3_1--v5_3_2 and v5_3_3--v5_3_4);
			 \draw[blue,thick] (v5_3_1) --(v5_3)--(v5_3_4);
			 
			 \coordinate (v5_4_1) at ($.6*(v5)+.03*(v3)+.46*(v4)$);
			 \coordinate (v5_4_2) at ($.83*(v5)+.03*(v3)+.23*(v4)$);
			 \coordinate (v5_4_3) at ($.81*(v5)+.07*(v3)+.14*(v4)$);
			 \coordinate (v5_4_4) at ($.63*(v5)+.2*(v3)+.14*(v4)$); 
			 \draw[red,thick] (v5_4_1) --(v5_4_2) --(v5_4_3) --(v5_4_4) ;
   		    \end{tikzpicture}
		    \subcaption{}
		    \label{fig:sec4_bisimplex_label1}
	\end{subfigure}	
	  	\begin{subfigure}{.3\linewidth}
  		    \begin{tikzpicture}[scale=0.4]
   			 \coordinate (v1) at (0,4);
			 \coordinate (v2) at (4,-4);
			 \coordinate (v3) at (-4,-4);
			 \coordinate (v5) at (0,0);
   			 \coordinate (v4) at (0,-2);
			 \coordinate (v6) at ($-.2*(v3)+1.2*(0,-1)-(0,.5)$);
			 \draw[very thick] (v1)--(v2)--(v3)--(v1);
			 \draw[very thick] (v5)--(v1);
			 \draw[very thick] (v4)--(v2);
			 \draw[very thick] (v4)--(v3);
			 \draw[very thick] (v5)--(v2);
			 \draw[very thick] (v5)--(v3);
			 \draw[very thick] (v5)--(v4);
			 
			 \coordinate (v1_1_1) at ($.84*(v1)+.08*(v5)+.08*(v3)-.02*(v3)+.02*(v5)+.1*(v3)-.1*(v1)$);
			 \coordinate (v1_1_2) at ($.84*(v1)+.08*(v5)+.08*(v3)-.02*(v3)-+02*(v5)$);
			 \coordinate (v1_1_3) at ($.8*(v1)+.1*(v5)+.1*(v3)-.07*(v3)+.07*(v5)$);
			 \coordinate (v1_1_4) at ($.8*(v1)+.1*(v5)+.1*(v3)-.07*(v3)+.07*(v5)+.2*(v5)-.2*(v1)$);
			 \coordinate (v1_1) at (intersection of v1_1_1--v1_1_2 and v1_1_3--v1_1_4);
			 \draw[blue,thick] (v1_1_1) --(v1_1)--(v1_1_4);
			 
			 \coordinate (v1_2_1) at ($.84*(v1)+.08*(v5)+.08*(v2)-.02*(v2)+.02*(v5)+.1*(v2)-.1*(v1)$);
			 \coordinate (v1_2_2) at ($.84*(v1)+.08*(v5)+.08*(v2)-.02*(v2)-+02*(v5)$);
			 \coordinate (v1_2_3) at ($.8*(v1)+.1*(v5)+.1*(v2)-.07*(v2)+.07*(v5)$);
			 \coordinate (v1_2_4) at ($.8*(v1)+.1*(v5)+.1*(v2)-.07*(v2)+.07*(v5)+.2*(v5)-.2*(v1)$);
			 \coordinate (v1_2) at (intersection of v1_2_1--v1_2_2 and v1_2_3--v1_2_4);
			 \draw[blue,thick] (v1_2_1) --(v1_2)--(v1_2_4);
			 
			 \coordinate (v1_3_1) at ($1.04*(v1)-.02*(v3)-.02*(v2)-.015*(v2)+.015*(v3)+.1*(v3)-.1*(v1)$);
			 \coordinate (v1_3_2) at ($1.04*(v1)-.02*(v3)-.02*(v2)-.015*(v2)+.015*(v3)$);
			 \coordinate (v1_3_3) at ($1.04*(v1)-.02*(v3)-.02*(v2)+.015*(v2)-.015*(v3)$);
			 \coordinate (v1_3_4) at ($1.04*(v1)-.02*(v3)-.02*(v2)+.015*(v2)-.015*(v3)+.1*(v2)-.1*(v1)$);
			 \coordinate (v1_3) at (intersection of v1_3_1--v1_3_2 and v1_3_3--v1_3_4);
			 \draw[blue,thick] (v1_3_1) --(v1_3)--(v1_3_4);

			 \coordinate (v2_1_1) at ($.7*(v2)+.1*(v5)+.2*(v4)-.015*(v3)+.015*(v1)+.07*(v5)-.07*(v2)$);
			 \coordinate (v2_1_2) at ($.7*(v2)+.1*(v5)+.2*(v4)-.015*(v3)+.015*(v1)$);
			 \coordinate (v2_1_3) at ($.7*(v2)+.1*(v5)+.2*(v4)+.005*(v3)-.005*(v1)$);
			 \coordinate (v2_1_4) at ($.7*(v2)+.1*(v5)+.2*(v4)+.005*(v3)-.005*(v1)+.07*(v4)-.07*(v2)$);
			 \coordinate (v2_1) at (intersection of v2_1_1--v2_1_2 and v2_1_3--v2_1_4);
			 \draw[blue,thick] (v2_1_1) --(v2_1)--(v2_1_4);
			 
 			 \coordinate (v2_2_1) at ($.7*(v2)+.1*(v5)+.2*(v4)+.04*(v3)-.04*(v1)+.07*(v4)-.07*(v2)$);
			 \coordinate (v2_2_2) at ($.7*(v2)+.1*(v5)+.2*(v4)+.04*(v3)-.04*(v1)$);
			 \coordinate (v2_2_3) at ($.7*(v2)+.1*(v5)+.2*(v4)+.08*(v3)-.08*(v1)$);
			 \coordinate (v2_2_4) at ($.7*(v2)+.1*(v5)+.2*(v4)+.08*(v3)-.08*(v1)+.07*(v3)-.07*(v2)$);
			 \coordinate (v2_2) at (intersection of v2_2_1--v2_2_2 and v2_2_3--v2_2_4);
			 \draw[blue,thick] (v2_2_1) --(v2_2)--(v2_2_4);

 			 \coordinate (v2_3_1) at ($.7*(v2)+.1*(v4)+.2*(v5)-.04*(v3)+.04*(v1)+.07*(v5)-.07*(v2)$);
			 \coordinate (v2_3_2) at ($.7*(v2)+.1*(v4)+.2*(v5)-.04*(v3)+.04*(v1)$);
			 \coordinate (v2_3_3) at ($.7*(v2)+.1*(v4)+.2*(v5)-.06*(v3)+.06*(v1)$);
			 \coordinate (v2_3_4) at ($.7*(v2)+.1*(v4)+.2*(v5)-.06*(v3)+.06*(v1)+.07*(v1)-.07*(v2)$);
			 \coordinate (v2_3) at (intersection of v2_3_1--v2_3_2 and v2_3_3--v2_3_4);
			 \draw[blue,thick] (v2_3_1) --(v2_3)--(v2_3_4);	
			 
			 \coordinate (v2_4_1) at ($.7*(v2)+.1*(v5)+.2*(v4)-.02*(v3)+.02*(v1)+.4*(v2)-.2*(v5)-0.2*(v4)+.1*(v1)-.1*(v2)$);
			 \coordinate (v2_4_2) at ($.7*(v2)+.1*(v5)+.2*(v4)-.02*(v3)+.02*(v1)+.4*(v2)-.2*(v5)-0.2*(v4)$);
			 \coordinate (v2_4_3) at ($.7*(v2)+.1*(v5)+.2*(v4)+.0*(v3)-.0*(v1)+.4*(v2)-.2*(v5)-0.2*(v4)$);
			 \coordinate (v2_4_4) at ($.7*(v2)+.1*(v5)+.2*(v4)+.0*(v3)-.0*(v1)+.4*(v2)-.2*(v5)-0.2*(v4)+.1*(v3)-.1*(v2)$);
			 \coordinate (v2_4) at (intersection of v2_4_1--v2_4_2 and v2_4_3--v2_4_4);
			 \draw[blue,thick] (v2_4_1) --(v2_4)--(v2_4_4);

			 \coordinate (v3_1_1) at ($.7*(v3)+.1*(v5)+.2*(v4)-.015*(v2)+.015*(v1)+.07*(v5)-.07*(v3)$);
			 \coordinate (v3_1_2) at ($.7*(v3)+.1*(v5)+.2*(v4)-.015*(v2)+.015*(v1)$);
			 \coordinate (v3_1_3) at ($.7*(v3)+.1*(v5)+.2*(v4)+.005*(v2)-.005*(v1)$);
			 \coordinate (v3_1_4) at ($.7*(v3)+.1*(v5)+.2*(v4)+.005*(v2)-.005*(v1)+.07*(v4)-.07*(v3)$);
			 \draw[red,thick] (v3_1_1) --(v3_1_2) --(v3_1_3) --(v3_1_4) ;
			 
 			 \coordinate (v3_2_1) at ($.7*(v3)+.1*(v5)+.2*(v4)+.04*(v2)-.04*(v1)+.07*(v4)-.07*(v3)$);
			 \coordinate (v3_2_2) at ($.7*(v3)+.1*(v5)+.2*(v4)+.04*(v2)-.04*(v1)$);
			 \coordinate (v3_2_3) at ($.7*(v3)+.1*(v5)+.2*(v4)+.08*(v2)-.08*(v1)$);
			 \coordinate (v3_2_4) at ($.7*(v3)+.1*(v5)+.2*(v4)+.08*(v2)-.08*(v1)+.07*(v2)-.07*(v3)$);
			 \coordinate (v3_2) at (intersection of v3_2_1--v3_2_2 and v3_2_3--v3_2_4);
			 \draw[blue,thick] (v3_2_1) --(v3_2)--(v3_2_4);

 			 \coordinate (v3_3_1) at ($.7*(v3)+.1*(v4)+.2*(v5)-.04*(v2)+.04*(v1)+.07*(v5)-.07*(v3)$);
			 \coordinate (v3_3_2) at ($.7*(v3)+.1*(v4)+.2*(v5)-.04*(v2)+.04*(v1)$);
			 \coordinate (v3_3_3) at ($.7*(v3)+.1*(v4)+.2*(v5)-.06*(v2)+.06*(v1)$);
			 \coordinate (v3_3_4) at ($.7*(v3)+.1*(v4)+.2*(v5)-.06*(v2)+.06*(v1)+.07*(v1)-.07*(v3)$);
			 \coordinate (v3_3) at (intersection of v3_3_1--v3_3_2 and v3_3_3--v3_3_4);
			 \draw[blue,thick] (v3_3_1) --(v3_3)--(v3_3_4);

			 \coordinate (v3_4_1) at ($.7*(v3)+.1*(v5)+.2*(v4)-.02*(v2)+.02*(v1)+.4*(v3)-.2*(v5)-0.2*(v4)+.1*(v1)-.1*(v3)$);
			 \coordinate (v3_4_2) at ($.7*(v3)+.1*(v5)+.2*(v4)-.02*(v2)+.02*(v1)+.4*(v3)-.2*(v5)-0.2*(v4)$);
			 \coordinate (v3_4_3) at ($.7*(v3)+.1*(v5)+.2*(v4)+.0*(v2)-.0*(v1)+.4*(v3)-.2*(v5)-0.2*(v4)$);
			 \coordinate (v3_4_4) at ($.7*(v3)+.1*(v5)+.2*(v4)+.0*(v2)-.0*(v1)+.4*(v3)-.2*(v5)-0.2*(v4)+.1*(v2)-.1*(v3)$);
			 \draw[red,thick] (v3_4_1) --(v3_4_2) --(v3_4_3) --(v3_4_4) ;
			 
			 \coordinate (v4_1_1) at ($.72*(v4)+.06*(v2)+.23*(v3)$);
			 \coordinate (v4_1_2) at ($.84*(v4)+.06*(v2)+.11*(v3)$);
			 \coordinate (v4_1_3) at ($.84*(v4)+.11*(v2)+.06*(v3)$);
			 \coordinate (v4_1_4) at ($.72*(v4)+.23*(v2)+.06*(v3)$);
			 \coordinate (v4_1) at (intersection of v4_1_1--v4_1_2 and v4_1_3--v4_1_4);
			 \draw[blue,thick] (v4_1_1) --(v4_1)--(v4_1_4);
			 
			 \coordinate (v4_2_1) at ($.66*(v4)+.03*(v2)+.32*(v5)$);
			 \coordinate (v4_2_2) at ($.8*(v4)+.03*(v2)+.18*(v5)$);
			 \coordinate (v4_2_3) at ($.84*(v4)+.08*(v2)+.09*(v5)$);
			 \coordinate (v4_2_4) at ($.71*(v4)+.19*(v2)+.09*(v5)$);
			 \coordinate (v4_2) at (intersection of v4_2_1--v4_2_2 and v4_2_3--v4_2_4);
			 \draw[blue,thick] (v4_2_1) --(v4_2)--(v4_2_4);
			 
			 \coordinate (v4_3_1) at ($.66*(v4)+.03*(v3)+.32*(v5)$);
			 \coordinate (v4_3_2) at ($.8*(v4)+.03*(v3)+.18*(v5)$);
			 \coordinate (v4_3_3) at ($.84*(v4)+.08*(v3)+.09*(v5)$);
			 \coordinate (v4_3_4) at ($.71*(v4)+.19*(v3)+.09*(v5)$);
			 \coordinate (v4_3) at (intersection of v4_3_1--v4_3_2 and v4_3_3--v4_3_4);
			 \draw[blue,thick] (v4_3_1) --(v4_3)--(v4_3_4);
			 
			 \coordinate (v5_1_1) at ($.72*(v5)+.04*(v2)+.25*(v1)$);
			 \coordinate (v5_1_2) at ($.84*(v5)+.04*(v2)+.13*(v1)$);
			 \coordinate (v5_1_3) at ($.85*(v5)+.08*(v2)+.08*(v1)$);
			 \coordinate (v5_1_4) at ($.73*(v5)+.2*(v2)+.08*(v1)$);
			 \draw[red,thick] (v5_1_1) --(v5_1_2) --(v5_1_3) --(v5_1_4) ;
			 
			 \coordinate (v5_2_1) at ($.6*(v5)+.03*(v2)+.46*(v4)$);
			 \coordinate (v5_2_2) at ($.83*(v5)+.03*(v2)+.23*(v4)$);
			 \coordinate (v5_2_3) at ($.81*(v5)+.07*(v2)+.14*(v4)$);
			 \coordinate (v5_2_4) at ($.63*(v5)+.2*(v2)+.14*(v4)$);
			 \coordinate (v5_2) at (intersection of v5_2_1--v5_2_2 and v5_2_3--v5_2_4);
			 \draw[blue,thick] (v5_2_1) --(v5_2)--(v5_2_4);
			 
			 \coordinate (v5_3_1) at ($.72*(v5)+.04*(v3)+.25*(v1)$);
			 \coordinate (v5_3_2) at ($.84*(v5)+.04*(v3)+.13*(v1)$);
			 \coordinate (v5_3_3) at ($.85*(v5)+.08*(v3)+.08*(v1)$);
			 \coordinate (v5_3_4) at ($.73*(v5)+.2*(v3)+.08*(v1)$);
			 \coordinate (v5_3) at (intersection of v5_3_1--v5_3_2 and v5_3_3--v5_3_4);
			 \draw[blue,thick] (v5_3_1) --(v5_3)--(v5_3_4);
			 
			 \coordinate (v5_4_1) at ($.6*(v5)+.03*(v3)+.46*(v4)$);
			 \coordinate (v5_4_2) at ($.83*(v5)+.03*(v3)+.23*(v4)$);
			 \coordinate (v5_4_3) at ($.81*(v5)+.07*(v3)+.14*(v4)$);
			 \coordinate (v5_4_4) at ($.63*(v5)+.2*(v3)+.14*(v4)$); 
			 \draw[red,thick] (v5_4_1) --(v5_4_2) --(v5_4_3) --(v5_4_4) ;
   		    \end{tikzpicture}
		    \subcaption{}
		    \label{fig:sec4_bisimplex_label2}
	\end{subfigure}
	  	\begin{subfigure}{.3\linewidth}
  		    \begin{tikzpicture}[scale=0.4]
   			 \coordinate (v1) at (0,4);
			 \coordinate (v2) at (4,-4);
			 \coordinate (v3) at (-4,-4);
			 \coordinate (v5) at (0,0);
   			 \coordinate (v4) at (0,-2);
			 \coordinate (v6) at ($-.2*(v3)+1.2*(0,-1)-(0,.5)$);
			 \draw[very thick] (v1)--(v2)--(v3)--(v1);
			 \draw[very thick] (v5)--(v1);
			 \draw[very thick] (v4)--(v2);
			 \draw[very thick] (v4)--(v3);
			 \draw[very thick] (v5)--(v2);
			 \draw[very thick] (v5)--(v3);
			 \draw[very thick] (v5)--(v4);
			 
			 \coordinate (v1_1_1) at ($.84*(v1)+.08*(v5)+.08*(v3)-.02*(v3)+.02*(v5)+.1*(v3)-.1*(v1)$);
			 \coordinate (v1_1_2) at ($.84*(v1)+.08*(v5)+.08*(v3)-.02*(v3)-+02*(v5)$);
			 \coordinate (v1_1_3) at ($.8*(v1)+.1*(v5)+.1*(v3)-.07*(v3)+.07*(v5)$);
			 \coordinate (v1_1_4) at ($.8*(v1)+.1*(v5)+.1*(v3)-.07*(v3)+.07*(v5)+.2*(v5)-.2*(v1)$);
			 \draw[red,thick] (v1_1_1) --(v1_1_2) --(v1_1_3) --(v1_1_4) ;
			 
			 \coordinate (v1_2_1) at ($.84*(v1)+.08*(v5)+.08*(v2)-.02*(v2)+.02*(v5)+.1*(v2)-.1*(v1)$);
			 \coordinate (v1_2_2) at ($.84*(v1)+.08*(v5)+.08*(v2)-.02*(v2)-+02*(v5)$);
			 \coordinate (v1_2_3) at ($.8*(v1)+.1*(v5)+.1*(v2)-.07*(v2)+.07*(v5)$);
			 \coordinate (v1_2_4) at ($.8*(v1)+.1*(v5)+.1*(v2)-.07*(v2)+.07*(v5)+.2*(v5)-.2*(v1)$);
			 \draw[red,thick] (v1_2_1) --(v1_2_2) --(v1_2_3) --(v1_2_4) ;
			 
			 \coordinate (v1_3_1) at ($1.04*(v1)-.02*(v3)-.02*(v2)-.015*(v2)+.015*(v3)+.1*(v3)-.1*(v1)$);
			 \coordinate (v1_3_2) at ($1.04*(v1)-.02*(v3)-.02*(v2)-.015*(v2)+.015*(v3)$);
			 \coordinate (v1_3_3) at ($1.04*(v1)-.02*(v3)-.02*(v2)+.015*(v2)-.015*(v3)$);
			 \coordinate (v1_3_4) at ($1.04*(v1)-.02*(v3)-.02*(v2)+.015*(v2)-.015*(v3)+.1*(v2)-.1*(v1)$);
			 \draw[red,thick] (v1_3_1)--(v1_3_2)--(v1_3_3)--(v1_3_4);

			 \coordinate (v2_1_1) at ($.7*(v2)+.1*(v5)+.2*(v4)-.015*(v3)+.015*(v1)+.07*(v5)-.07*(v2)$);
			 \coordinate (v2_1_2) at ($.7*(v2)+.1*(v5)+.2*(v4)-.015*(v3)+.015*(v1)$);
			 \coordinate (v2_1_3) at ($.7*(v2)+.1*(v5)+.2*(v4)+.005*(v3)-.005*(v1)$);
			 \coordinate (v2_1_4) at ($.7*(v2)+.1*(v5)+.2*(v4)+.005*(v3)-.005*(v1)+.07*(v4)-.07*(v2)$);
			 \coordinate (v2_1) at (intersection of v2_1_1--v2_1_2 and v2_1_3--v2_1_4);
			 \draw[blue,thick] (v2_1_1) --(v2_1)--(v2_1_4);
			 
 			 \coordinate (v2_2_1) at ($.7*(v2)+.1*(v5)+.2*(v4)+.04*(v3)-.04*(v1)+.07*(v4)-.07*(v2)$);
			 \coordinate (v2_2_2) at ($.7*(v2)+.1*(v5)+.2*(v4)+.04*(v3)-.04*(v1)$);
			 \coordinate (v2_2_3) at ($.7*(v2)+.1*(v5)+.2*(v4)+.08*(v3)-.08*(v1)$);
			 \coordinate (v2_2_4) at ($.7*(v2)+.1*(v5)+.2*(v4)+.08*(v3)-.08*(v1)+.07*(v3)-.07*(v2)$);
			 \coordinate (v2_2) at (intersection of v2_2_1--v2_2_2 and v2_2_3--v2_2_4);
			 \draw[blue,thick] (v2_2_1) --(v2_2)--(v2_2_4);

 			 \coordinate (v2_3_1) at ($.7*(v2)+.1*(v4)+.2*(v5)-.04*(v3)+.04*(v1)+.07*(v5)-.07*(v2)$);
			 \coordinate (v2_3_2) at ($.7*(v2)+.1*(v4)+.2*(v5)-.04*(v3)+.04*(v1)$);
			 \coordinate (v2_3_3) at ($.7*(v2)+.1*(v4)+.2*(v5)-.06*(v3)+.06*(v1)$);
			 \coordinate (v2_3_4) at ($.7*(v2)+.1*(v4)+.2*(v5)-.06*(v3)+.06*(v1)+.07*(v1)-.07*(v2)$);
			 \coordinate (v2_3) at (intersection of v2_3_1--v2_3_2 and v2_3_3--v2_3_4);
			 \draw[blue,thick] (v2_3_1) --(v2_3)--(v2_3_4);	
			 
			 \coordinate (v2_4_1) at ($.7*(v2)+.1*(v5)+.2*(v4)-.02*(v3)+.02*(v1)+.4*(v2)-.2*(v5)-0.2*(v4)+.1*(v1)-.1*(v2)$);
			 \coordinate (v2_4_2) at ($.7*(v2)+.1*(v5)+.2*(v4)-.02*(v3)+.02*(v1)+.4*(v2)-.2*(v5)-0.2*(v4)$);
			 \coordinate (v2_4_3) at ($.7*(v2)+.1*(v5)+.2*(v4)+.0*(v3)-.0*(v1)+.4*(v2)-.2*(v5)-0.2*(v4)$);
			 \coordinate (v2_4_4) at ($.7*(v2)+.1*(v5)+.2*(v4)+.0*(v3)-.0*(v1)+.4*(v2)-.2*(v5)-0.2*(v4)+.1*(v3)-.1*(v2)$);
			 \coordinate (v2_4) at (intersection of v2_4_1--v2_4_2 and v2_4_3--v2_4_4);
			 \draw[blue,thick] (v2_4_1) --(v2_4)--(v2_4_4);

			 \coordinate (v3_1_1) at ($.7*(v3)+.1*(v5)+.2*(v4)-.015*(v2)+.015*(v1)+.07*(v5)-.07*(v3)$);
			 \coordinate (v3_1_2) at ($.7*(v3)+.1*(v5)+.2*(v4)-.015*(v2)+.015*(v1)$);
			 \coordinate (v3_1_3) at ($.7*(v3)+.1*(v5)+.2*(v4)+.005*(v2)-.005*(v1)$);
			 \coordinate (v3_1_4) at ($.7*(v3)+.1*(v5)+.2*(v4)+.005*(v2)-.005*(v1)+.07*(v4)-.07*(v3)$);
			 \coordinate (v3_1) at (intersection of v3_1_1--v3_1_2 and v3_1_3--v3_1_4);
			 \draw[blue,thick] (v3_1_1) --(v3_1)--(v3_1_4);
			 
 			 \coordinate (v3_2_1) at ($.7*(v3)+.1*(v5)+.2*(v4)+.04*(v2)-.04*(v1)+.07*(v4)-.07*(v3)$);
			 \coordinate (v3_2_2) at ($.7*(v3)+.1*(v5)+.2*(v4)+.04*(v2)-.04*(v1)$);
			 \coordinate (v3_2_3) at ($.7*(v3)+.1*(v5)+.2*(v4)+.08*(v2)-.08*(v1)$);
			 \coordinate (v3_2_4) at ($.7*(v3)+.1*(v5)+.2*(v4)+.08*(v2)-.08*(v1)+.07*(v2)-.07*(v3)$);
			 \coordinate (v3_2) at (intersection of v3_2_1--v3_2_2 and v3_2_3--v3_2_4);
			 \draw[blue,thick] (v3_2_1) --(v3_2)--(v3_2_4);

 			 \coordinate (v3_3_1) at ($.7*(v3)+.1*(v4)+.2*(v5)-.04*(v2)+.04*(v1)+.07*(v5)-.07*(v3)$);
			 \coordinate (v3_3_2) at ($.7*(v3)+.1*(v4)+.2*(v5)-.04*(v2)+.04*(v1)$);
			 \coordinate (v3_3_3) at ($.7*(v3)+.1*(v4)+.2*(v5)-.06*(v2)+.06*(v1)$);
			 \coordinate (v3_3_4) at ($.7*(v3)+.1*(v4)+.2*(v5)-.06*(v2)+.06*(v1)+.07*(v1)-.07*(v3)$);
			 \coordinate (v3_3) at (intersection of v3_3_1--v3_3_2 and v3_3_3--v3_3_4);
			 \draw[blue,thick] (v3_3_1) --(v3_3)--(v3_3_4);

			 \coordinate (v3_4_1) at ($.7*(v3)+.1*(v5)+.2*(v4)-.02*(v2)+.02*(v1)+.4*(v3)-.2*(v5)-0.2*(v4)+.1*(v1)-.1*(v3)$);
			 \coordinate (v3_4_2) at ($.7*(v3)+.1*(v5)+.2*(v4)-.02*(v2)+.02*(v1)+.4*(v3)-.2*(v5)-0.2*(v4)$);
			 \coordinate (v3_4_3) at ($.7*(v3)+.1*(v5)+.2*(v4)+.0*(v2)-.0*(v1)+.4*(v3)-.2*(v5)-0.2*(v4)$);
			 \coordinate (v3_4_4) at ($.7*(v3)+.1*(v5)+.2*(v4)+.0*(v2)-.0*(v1)+.4*(v3)-.2*(v5)-0.2*(v4)+.1*(v2)-.1*(v3)$);
			 \coordinate (v3_4) at (intersection of v3_4_1--v3_4_2 and v3_4_3--v3_4_4);
			 \draw[blue,thick] (v3_4_1) --(v3_4)--(v3_4_4);
			 
			 \coordinate (v4_1_1) at ($.72*(v4)+.06*(v2)+.23*(v3)$);
			 \coordinate (v4_1_2) at ($.84*(v4)+.06*(v2)+.11*(v3)$);
			 \coordinate (v4_1_3) at ($.84*(v4)+.11*(v2)+.06*(v3)$);
			 \coordinate (v4_1_4) at ($.72*(v4)+.23*(v2)+.06*(v3)$);
			 \draw[red,thick] (v4_1_1) --(v4_1_2) --(v4_1_3) --(v4_1_4) ;
			 
			 \coordinate (v4_2_1) at ($.66*(v4)+.03*(v2)+.32*(v5)$);
			 \coordinate (v4_2_2) at ($.8*(v4)+.03*(v2)+.18*(v5)$);
			 \coordinate (v4_2_3) at ($.84*(v4)+.08*(v2)+.09*(v5)$);
			 \coordinate (v4_2_4) at ($.71*(v4)+.19*(v2)+.09*(v5)$);
			 \draw[red,thick] (v4_2_1) --(v4_2_2) --(v4_2_3) --(v4_2_4) ;
			 
			 \coordinate (v4_3_1) at ($.66*(v4)+.03*(v3)+.32*(v5)$);
			 \coordinate (v4_3_2) at ($.8*(v4)+.03*(v3)+.18*(v5)$);
			 \coordinate (v4_3_3) at ($.84*(v4)+.08*(v3)+.09*(v5)$);
			 \coordinate (v4_3_4) at ($.71*(v4)+.19*(v3)+.09*(v5)$);
			 \draw[red,thick] (v4_3_1) --(v4_3_2) --(v4_3_3) --(v4_3_4) ;
			 
			 \coordinate (v5_1_1) at ($.72*(v5)+.04*(v2)+.25*(v1)$);
			 \coordinate (v5_1_2) at ($.84*(v5)+.04*(v2)+.13*(v1)$);
			 \coordinate (v5_1_3) at ($.85*(v5)+.08*(v2)+.08*(v1)$);
			 \coordinate (v5_1_4) at ($.73*(v5)+.2*(v2)+.08*(v1)$);
			 \coordinate (v5_1) at (intersection of v5_1_1--v5_1_2 and v5_1_3--v5_1_4);
			 \draw[blue,thick] (v5_1_1) --(v5_1)--(v5_1_4);
			 
			 \coordinate (v5_2_1) at ($.6*(v5)+.03*(v2)+.46*(v4)$);
			 \coordinate (v5_2_2) at ($.83*(v5)+.03*(v2)+.23*(v4)$);
			 \coordinate (v5_2_3) at ($.81*(v5)+.07*(v2)+.14*(v4)$);
			 \coordinate (v5_2_4) at ($.63*(v5)+.2*(v2)+.14*(v4)$);
			 \coordinate (v5_2) at (intersection of v5_2_1--v5_2_2 and v5_2_3--v5_2_4);
			 \draw[blue,thick] (v5_2_1) --(v5_2)--(v5_2_4);
			 
			 \coordinate (v5_3_1) at ($.72*(v5)+.04*(v3)+.25*(v1)$);
			 \coordinate (v5_3_2) at ($.84*(v5)+.04*(v3)+.13*(v1)$);
			 \coordinate (v5_3_3) at ($.85*(v5)+.08*(v3)+.08*(v1)$);
			 \coordinate (v5_3_4) at ($.73*(v5)+.2*(v3)+.08*(v1)$);
			 \coordinate (v5_3) at (intersection of v5_3_1--v5_3_2 and v5_3_3--v5_3_4);
			 \draw[blue,thick] (v5_3_1) --(v5_3)--(v5_3_4);
			 
			 \coordinate (v5_4_1) at ($.6*(v5)+.03*(v3)+.46*(v4)$);
			 \coordinate (v5_4_2) at ($.83*(v5)+.03*(v3)+.23*(v4)$);
			 \coordinate (v5_4_3) at ($.81*(v5)+.07*(v3)+.14*(v4)$);
			 \coordinate (v5_4_4) at ($.63*(v5)+.2*(v3)+.14*(v4)$); 
			 \coordinate (v5_4) at (intersection of v5_4_1--v5_4_2 and v5_4_3--v5_4_4);
			 \draw[blue,thick] (v5_4_1) --(v5_4)--(v5_4_4);
   		    \end{tikzpicture}
		    \subcaption{}
		    \label{fig:sec4_bisimplex_label3}
	\end{subfigure}  
	\caption{}
	\label{fig:sec4_bisimplex_label}
	\end{center}
\end{figure}

The labeling in Figure~\ref{fig:sec4_bisimplex_label1} does not correspond to any $4$-polytope. To see this just note that the straightforward identification of nodes implied by the facet intersection rules leads to Figure~\ref{fig:sec4_bisimplex_result1}, where nodes of the same shape represent the same vertex. This identification implies the existence of two facets of $P$ whose intersection is not a face of $P$, for instance the outer facet and the middle left facet in the diagram intersect at two vertices not sharing an edge, which would be impossible if $P$ was actually a polytope.

The diagram in Figure~\ref{fig:sec4_bisimplex_label2} leads to a unique psd-minimal $4$-polytope, whose Schlegel diagram is in Figure~\ref{fig:sec4_bisimplex_result2}. 
It is also straightforward to obtain the Schlegel diagram in Figure~\ref{fig:sec4_bisimplex_result3} of the polytope corresponding to the diagram in Figure~\ref{fig:sec4_bisimplex_label3}. However, we can see in that Schlegel digram that the vertices of the central triangle are contained in six facets of $P$, which means that this must be dual to a previously discovered polytope, and in fact, it is dual to class $12$. So again we only get one new class of polytopes, class $3$, which turns out to be self dual.

\begin{figure}[h!]	
	  \begin{center}
	  	\begin{subfigure}{.3\linewidth}
  		    \begin{tikzpicture}[scale=0.4]
   			 \coordinate (v1) at (0,4);
			 \coordinate (v2) at (4,-4);
			 \coordinate (v3) at (-4,-4);
			 \coordinate (v5) at (0,0);
   			 \coordinate (v4) at (0,-2);
			 \coordinate (v6) at ($-.2*(v3)+1.2*(0,-1)-(0,.5)$);
			 \draw[very thick] (v1)--(v2)--(v3)--(v1);
			 \draw[very thick]  (v5)--(v1);
			 \draw[very thick]  (v4)--(v2);
			 \draw[very thick]  (v4)--(v3);
			 \draw[very thick]  (v5)--(v2);
			 \draw[very thick]  (v5)--(v3);
			 \draw[very thick]  (v5)--(v4);
			 
   			 \coordinate (u) at ($.6*(v5)+.25*(v3)+.15*(v1)$);
			 \coordinate (w) at ($.2*(v5)+.5*(v3)+.3*(v1)$);
			  \draw[very thin](u)--(v5);
			  \draw[very thin](u)--(v3);
			  \draw[very thin](u)--(v1);
			  \draw[very thin](w)--(v3);
			  \draw[very thin](w)--(v1);
			  \draw[very thin](w)--(u);
		          \node[scale=0.4,circle,fill=red!] () at (u){}; 
		          \node[scale=0.3,regular polygon,regular polygon sides=3,fill=blue!] () at (w){}; 
			  
   			 \coordinate (u) at ($.6*(v2)+.25*(v5)+.15*(v1)$);
			 \coordinate (w) at ($.2*(v2)+.5*(v5)+.3*(v1)$);
			  \draw[very thin](u)--(v5);
			  \draw[very thin](u)--(v2);
			  \draw[very thin](u)--(v1);
			  \draw[very thin](w)--(v5);
			  \draw[very thin](w)--(v1);
			  \draw[very thin](w)--(u);
		          \node[scale=0.4,fill=green!, star, star points=7] () at (u){}; 
		          \node[scale=0.4,circle,fill=red!] () at (w){}; 
			  
			  \coordinate (w) at ($.5*(v5)+.5*(v1)+(2,0)$);
			  \coordinate (u) at ($.5*(v5)+.5*(v1)+(2,2)$);
			  \draw[very thin](w)--(u);
			  \draw[very thin](u)--(v2);
			  \draw[very thin](u)--(v1);
			  \draw[very thin](w)--(v2);
			  \draw[very thin](w)--(v1); 
		          \draw[very thin] (v3) edge[out=100,in=130] (u);	
		          \node[scale=0.3,regular polygon,regular polygon sides=3,fill=blue!] () at (u){}; 
		          \node[scale=0.4,fill=green!, star, star points=7] () at (w){}; 
		          
   			 \coordinate (u) at ($.6*(v2)+.2*(v3)+.2*(v4)$);
			 \coordinate (w) at ($.2*(v2)+.4*(v3)+.4*(v4)$);
			  \draw[very thin](u)--(v2);
			  \draw[very thin](u)--(v3);
			  \draw[very thin](u)--(v4);
			  \draw[very thin](w)--(v3);
			  \draw[very thin](w)--(v4);
			  \draw[very thin](w)--(u);
		          \node[scale=0.3,regular polygon,regular polygon sides=3, fill=blue!] () at (u){}; 
		          \node[scale=0.4,circle,fill=red!] () at (w){}; 
			  
			  \coordinate (u) at ($.5*(v3)+.25*(v4)+.25*(v5)$);
			 \coordinate (w) at ($.2*(v3)+.4*(v4)+.4*(v5)$);
			  \draw[very thin](u)--(v3);
			  \draw[very thin](u)--(v4);
			  \draw[very thin](u)--(v5);
			  \draw[very thin](w)--(v4);
			  \draw[very thin](w)--(v5);
			  \draw[very thin](w)--(u);
		          \node[scale=0.4,circle,fill=red!] () at (u){}; 
		          \node[scale=0.4,fill=green!, star, star points=7] () at (w){};  
			  
			 \coordinate (u) at ($.6*(v5)+.25*(v4)+.15*(v2)$);
			 \coordinate (w) at ($.3*(v5)+.4*(v4)+.3*(v2)$);
			  \draw[very thin](u)--(v4);
			  \draw[very thin](u)--(v5);
			  \draw[very thin](u)--(v2);
			  \draw[very thin](w)--(v2);
			  \draw[very thin](w)--(v4);
			  \draw[very thin](w)--(u);
		          \node[scale=0.4,fill=green!, star, star points=7] () at (u){}; 
		          \node[scale=0.3,regular polygon,regular polygon sides=3,fill=blue!] () at (w){};

   		    \end{tikzpicture}
		    \subcaption{}
		    \label{fig:sec4_bisimplex_result1}
	\end{subfigure}  	  	
	\begin{subfigure}{.3\linewidth}
		    \begin{tikzpicture}[xscale=.7, yscale=0.55, x={(1cm,0cm)},y={(-.3cm,-0.35cm)}, z={(0cm,2cm)}]
			 \coordinate (v1) at (-2.5, -1, 0);
			 \coordinate (v2) at (2.5, -1, 0);
			 \coordinate (v3) at (0, 3, 0);
			 \coordinate (v4) at ($1/3*(v1)+1/3*(v2)+1/3*(v3)+(0,0,1.5)$);
   			 \coordinate (v5) at  ($1/3*(v1)+1/3*(v2)+1/3*(v3)-(0,0,1.5)$);
			 \draw[very thick] (v1)--(v2);
			 \draw[very thick] (v1)--(v3);
			 \draw[very thick] (v2)--(v3);
			 \draw[very thick] (v1)--(v4);
			 \draw[very thick] (v2)--(v4);
			 \draw[very thick] (v3)--(v4);
			 \draw[very thick] (v1)--(v5);
			 \draw[very thick] (v2)--(v5);
			 \draw[very thick] (v3)--(v5);
			 \coordinate (u) at ($.15*(v1)+.15*(v3)+.5*(v4)+.2*(v2)$);
			 \coordinate (t) at ($.15*(v1)+.15*(v3)+.0*(v4)+.7*(v2)$); 
			 \draw[red] (u)--(t);
			 \draw[blue] (u)--(v1);
			 \draw[blue] (u)--(v3);
			 \draw[blue] (u)--(v4);
			 \draw[blue] (t)--(v1);
			 \draw[blue] (t)--(v2);
			 \draw[blue] (t)--(v3);
			 \draw[blue] (t)--(v5);
		   \end{tikzpicture}
		    \subcaption{}
		    \label{fig:sec4_bisimplex_result2}
	\end{subfigure}  
	  	\begin{subfigure}{.3\linewidth}
 		    \begin{tikzpicture}[xscale=.7, yscale=0.55, x={(1cm,0cm)},y={(-.3cm,-0.35cm)}, z={(0cm,2cm)}]
			 \coordinate (v1) at (-2.5, -1, 0);
			 \coordinate (v2) at (2.5, -1, 0);
			 \coordinate (v3) at (0, 3, 0);
			 \coordinate (v4) at ($1/3*(v1)+1/3*(v2)+1/3*(v3)+(0,0,1.5)$);
   			 \coordinate (v5) at  ($1/3*(v1)+1/3*(v2)+1/3*(v3)-(0,0,1.5)$);
			 \draw[very thick] (v1)--(v2);
			 \draw[very thick] (v1)--(v3);
			 \draw[very thick] (v2)--(v3);
			 \draw[very thick] (v1)--(v4);
			 \draw[very thick] (v2)--(v4);
			 \draw[very thick] (v3)--(v4);
			 \draw[very thick] (v1)--(v5);
			 \draw[very thick] (v2)--(v5);
			 \draw[very thick] (v3)--(v5);
			 \coordinate (u1) at ($.7*(v1)+.15*(v3)+.15*(v2)$);
			 \coordinate (u2) at ($.15*(v1)+.15*(v3)+.7*(v2)$); 
			 \coordinate (u3) at ($.15*(v1)+.7*(v3)+.15*(v2)$); 
			 \draw[red] (u1)--(u2);
			 \draw[red] (u3)--(u2);
			 \draw[red] (u1)--(u3);
			 \draw[blue] (u1)--(v1);
			 \draw[blue] (u2)--(v2);
			 \draw[blue] (u3)--(v3);
			 \draw[blue] (u1)--(v4);
			 \draw[blue] (u2)--(v4);
			 \draw[blue] (u3)--(v4);
			 \draw[blue] (u1)--(v5);
			 \draw[blue] (u2)--(v5);
			 \draw[blue] (u3)--(v5);
		   \end{tikzpicture}   
		    \subcaption{}
		    \label{fig:sec4_bisimplex_result3}
	\end{subfigure}  
	\caption{}
	\label{fig:sec4_bisimplex_result}
	\end{center}
\end{figure}

\subsection{Square Pyramid} Now we assume that the facets of $P$ are simplices and square pyramids. By duality we can additionally assume that every vertex of $P$ is contained in at most five facets of $P$, and that every vertex of $P$ is adjacent to at most five other vertices of $P$.

Take a facet $F$ of $P$, where $F$ is a square pyramid. Let one of the triangular faces of $F$ belong to another facet of $P$, which is a square pyramid.  Figure~\ref{fig:sec4_square_pyramid} is the only possibility under this assumption. 
Indeed, note that if two square pyramids would have the same apex then this apex would be adjacent to six other vertices. Furthermore, by Proposition~\ref{prop:intersections} no facet of $P$ can intersect $F$ in an edge containing the apex of $F$. 
Now it is easy to see that if $u$ is not identified with the triangle-nodes or circle-nodes then the vertex $v_4$ is adjacent to at least six other vertices of $P$. If $u$ is a circle-node then $v_3$ is contained in at least six facets and if $u$ is triangle-node then $v_1$ is contained in at least six facets of $P$. All these cases have been 
accounted for already.

\begin{figure}[h!]	
	  \begin{center}
  		    \begin{tikzpicture}[scale=0.7]
   			 \coordinate (v1) at (0,0);
			 \coordinate (v2) at (4,0);
			 \coordinate (v3) at (4,4);
			 \coordinate (v4) at (0,4);
   			 \coordinate (v5) at ($.5*(v1)+.5*(v3)$);
			 \draw (v1)--(v2)--(v3)--(v4)--(v1);
			 \draw[very thick]  (v5)--(v1);
			 \draw[very thick]  (v5)--(v2);
			 \draw[very thick]  (v5)--(v3);
			 \draw[very thick]  (v5)--(v4);
			 \node[label=$v_1$, left] at ($(v1)-(.2,0)$) {};
			 \node[label=$v_2$, right] at ($(v2)+(.2,.2)$) {};
			 \node[label=$v_3$, below] at (v3) {};
			 \node[label=$v_4$, left] at (v4) {};
			 
   			 \coordinate (u) at ($.2*(v2)+.6*(v1)+.2*(v5)$);
			 \coordinate (w) at ($.2*(v2)+.2*(v1)+.6*(v5)$);
			  \draw[very thin](u)--(v1);
			  \draw[very thin](u)--(v2);
			  \draw[very thin](w)--(v5);
			  \draw[very thin](w)--(v2);
			  \draw[very thin](w)--(u);
		          \node[scale=0.4,circle,fill=red!] () at (u){}; 
		          \node[scale=0.3,regular polygon,regular polygon sides=3,fill=blue!] () at (w){}; 
			  
   			 \coordinate (u) at ($.2*(v4)+.6*(v1)+.2*(v5)$);
			 \coordinate (w) at ($.2*(v4)+.2*(v1)+.6*(v5)$);
			  \draw[very thin](u)--(v1);
			  \draw[very thin](u)--(v4);
			  \draw[very thin](w)--(v5);
			  \draw[very thin](w)--(v4);
			  \draw[very thin](w)--(u); 
		          \node[scale=0.4,circle,fill=red!] () at (u){}; 
		          \node[scale=0.3,regular polygon,regular polygon sides=3,fill=blue!] () at (w){};

			  \coordinate (u) at ($.3*(v4)+.4*(v3)+.3*(v5)$);
			  \draw[very thin](u)--(v4);
			  \draw[very thin](u)--(v3);
			  \draw[very thin](u)--(v5);
		          \node[scale=0.3,regular polygon,regular polygon sides=3,fill=blue!] () at (u){};   
			  
			 \coordinate (u) at ($.3*(v2)+.4*(v3)+.3*(v5)$);
			  \draw[very thin](u)--(v2);
			  \draw[very thin](u)--(v3);
			  \draw[very thin](u)--(v5);
		          \node[scale=0.3,regular polygon,regular polygon sides=3,fill=blue!] () at (u){};   
			  
			  \coordinate (u) at ($1.2*(v2)+1.2*(v3)-1.4*(v5)$);
			  \draw[very thin](u)--(v2);
			  \draw[very thin](u)--(v3);
		          \draw[very thin] (v1) edge[out=-30,in=-90] (u);	
		          \draw[very thin] (v4) edge[out=30,in=90] (u);	
			 \node[label=$u$, left] at ($(u)+(.4,0)$) {};
          			 
   		    \end{tikzpicture} 	  	
	\caption{}
	\label{fig:sec4_square_pyramid}
	\end{center}
\end{figure}

Thus we can now assume that triangular facets of $F$ are not contained in any other square pyramid facet. In this case a similar identification of vertices shows that $P$ is a pyramid over $F$, giving us the self dual polytope class $2$.

\subsection{Simplex}
Finally, we can assume that all facets of $P$ are simplices, and by duality, that every vertex of $P$ is contained in exactly four facets. Then $P$ is a simplex, and we obtain the last class of psd-minimal polytopes, class $1$.

\section{Slack ideals and the binomial condition (Classes 1--11)} \label{sec:polys1-11}
\label{sec:slackideals}

\subsection{Slack ideals of polytopes} \label{subsec:slackideals}
We now define the {\em slack ideal} of a $d$-polytope which we use
to understand psd-minimality in the rest of this paper. 

\begin{definition} \label{def:slack ideal}
The {\em slack ideal} $I_P$ of a $d$-polytope $P$ is the ideal of all 
$(d+2)$-minors of $S_P(x)$ saturated with respect to all the variables in $S_P(x)$. In mathematical notation,
if the variables in $S_P(x)$ are $x_1, \ldots, x_t$, then 
$$I_P = \langle (d+2) \textup{-minors of } S_P(x) \rangle \,:\, (\Pi_{i=1}^t x_i)^\infty.$$
\end{definition}

Recall that the saturation of an ideal $I$ with all its variables is the ideal generated by all 
polynomials $f$ for which a monomial multiple of $f$ lies in $I$.
Recall also that in this paper, a binomial is a polynomial of the form $\pm x^a \pm x^b$. We say that an ideal is binomial if 
it is generated by binomials. Since the variety of $I_P$ contains positive points, namely 
the vector of positive 
elements in the slack matrix $S_P$,
any binomial in $I_P$ is of the form $x^a-x^b$.

\begin{example} \label{ex:slack ideal examples}
For a square,
$$S_P(x) = \begin{bmatrix} 
0&x_1&x_2&0\\
0&0&x_3&x_4\\
x_5&0&0&x_6\\
x_7&x_8&0&0
\end{bmatrix} $$
and $I_P =  \langle x_2x_4x_5x_8-x_1x_3x_6x_7 \rangle$ is a binomial ideal.

We now compute the slack ideal of a polytope in class 3. One can compute the 
slack matrix of the specific polytope in class 3 given in the table and check that the following is the 
symbolic slack matrix of a polytope in this class:
$$S_P(x) = \begin{bmatrix}
0&x_1&0&0&0&x_2&0\\
x_3&0&0&0&0&x_4&0\\
x_5&0&x_6&0&0&0&x_7\\
0&x_8&x_9&0&0&0&x_{10}\\
0&0&0&0&x_{11}&0&x_{12}\\
0&0&0&x_{13}&x_{14}&x_{15}&0\\
0&0&x_{16}&x_{17}&0&0&0
\end{bmatrix}.$$

The ideal of all $6$-minors of $S_P(x)$ is generated by $49$ polynomials all of which are binomials 
except the following four:
\begin{equation}\label{eq:generators6x6}
\begin{split}
     x_4x_5x_{10}x_{11}x_{13}x_{16}-x_4x_5x_9x_{12}x_{14}x_{17}+x_3x_7x_9x_{11}x_{15}x_{17}-x_3x_6x_{10}x_{11}x_{15}x_{17},\\
     x_2x_7x_8x_{11}x_{13}x_{16}-x_2x_6x_8x_{12}x_{14}x_{17}-x_1x_7x_9x_{11}x_{15}x_{17}+x_1x_6x_{10}x_{11}x_{15}x_{17},\\
x_2x_3x_7x_8x_{13}x_{16}-x_1x_4x_5x_{10}x_{13}x_{16}-x_1x_3x_7x_9x_{15}x_{17}+x_1x_3x_6x_{10}x_{15}x_{17},\\
x_2x_3x_6x_8x_{12}x_{14}-x_1x_4x_5x_9x_{12}x_{14}+x_1x_3x_7x_9x_{11}x_{15}-x_1x_3x_6x_{10}x_{11}x_{15}.
\end{split}
\end{equation}

Saturating  the ideal of minors with the product of all variables we obtain the binomial ideal:\\
$ I_P = \langle 
 x_{7}x_{9}-x_{6}x_{10}, x_{10}x_{11}x_{13}x_{16}-x_{9}x_{12}x_{14}x_{17}, 
 x_{7}x_{11}x_{13}x_{16}-x_{6}x_{12}x_{14}x_{17},\\
x_{2}x_{8}x_{13}x_{16}-x_{1}x_{9}x_{15}x_{17}, x_{4}x_{5}x_{13}x_{16}-x_{3}x_{6}x_{15}x_{17}, 
x_{2}x_{8}x_{12}x_{14}-x_{1}x_{10}x_{11}x_{15},\\
x_{4}x_{5}x_{12}x_{14}-x_{3}x_{7}x_{11}x_{15}, 
x_{2}x_{3}x_{7}x_{8}-x_{1}x_{4}x_{5}x_{10}, 
x_{2}x_{3}x_{6}x_{8}-x_{1}x_{4}x_{5}x_{9} \rangle.
$\\

Notice that saturation changed the ideal of minors into a binomial ideal. For example, the first generator in \eqref{eq:generators6x6}, which is the $6$-minor of $S_P(x)$ obtained by dropping the first row and second column, can be written as 
$$x_4x_5(x_{10}x_{11}x_{13}x_{16} - x_9x_{12}x_{14}x_{17} ) + x_3x_{11}x_{15}x_{17}(x_7x_9 -x_6x_{10}).$$
Saturation puts $x_7x_9- x_6x_{10}$ and $x_{10}x_{11}x_{13}x_{16} - x_9x_{12} x_{14}x_{17}$ in the slack ideal and hence prevents this four-term polynomial from being a minimal generator 
 of $I_P$. Similarly, all the other four-term polynomials are also unnecessary to generate 
 $I_P$.
 This example highlights the fact that the slack ideal may be different from the 
ideal of $(d+2)$-minors of $S_P(x)$ and, in particular, that $I_P$ can be binomial even if the ideal generated by the $(d+2)$-minors is not. 
\end{example}

Matrices with a fixed zero pattern and all non-zero entries being variables are known in the literature as {\em sparse generic matrices}. Furthermore, {\em sparse determinantal ideals} are the ideals of fixed size minors of a sparse generic matrix, and have been studied in different situations \cite{GiustiMerle}, \cite{Boocher}. The slack ideal of a polytope is not exactly a sparse determinantal ideal, but is the saturation of one, i.e., the saturation of  the ideal of all $(d+2)$-minors of the sparse generic matrix $S_P(x)$. In this paper, we define and use slack ideals for psd-minimality computations, but they could be of independent interest to both algebraists and combinatorialists.

Let $\mathcal{V}(I_P)$ denote the real zeros of the slack ideal $I_P$. It is convenient to identify $s \in \mathcal{V}(I_P)$ with the matrix $S_P(s)$. Then, by construction, $\rank(S_P(s)) \leq d+1$.  If $Q$ is a polytope that is combinatorially equivalent to $P$, then a slack matrix of $Q$ is a positive element of $\mathcal{V}(I_P)$. Conversely, by \cite[Theorem~24]{slackmatrixpaper}, { every positive element of $\mathcal{V}(I_P)$ is}, up to row scaling, { the slack matrix of a polytope in the combinatorial class of $P$}. In fact, Corollary~\ref{cor:proj_equivalent slack matrix}, gives a more precise description of $\mathcal{V}(I_P)$.

\begin{theorem}
Given a polytope $P$ there is a one to one correspondence between the positive elements of $\mathcal{V}(I_P)$, modulo row and column scalings, and the classes of projectively equivalent polytopes of the same combinatorial
type as $P$.
\end{theorem}

\subsection{Binomial slack ideals} \label{subsec:binomial condition}

\begin{definition} We say that  a $d$-polytope $P$ is {\em combinatorially psd-minimal} if all $d$-polytopes
that are combinatorially equivalent to $P$ are psd-minimal.
\end{definition}

A $d$-simplex is combinatorially psd-minimal. By Proposition~\ref{prop:psd minimal in the plane}, a 
square is also combinatorially psd-minimal.

\begin{lemma} \label{lem:binomial}
If the slack ideal $I_P$ of a polytope $P$ is binomial then $P$ is combinatorially psd-minimal.
\end{lemma}

\begin{proof}
If $I_P$ is a binomial ideal then for every positive $s \in \mathcal{V}(I_P)$ 
and generator $x^a-x^b$ of $I_P$, we have that $s^a=s^b$, which implies $(\sqrt[+]{s})^a = (\sqrt[+]{s})^b$ where $\sqrt[+]{s}$ is the entry-wise positive square root of $s$. Thus $\sqrt[+]{s}$ annihilates every generator of $I_P$, so $\rank(\sqrt[+]{S_P(s)}) \leq d+1$.  
\end{proof}

It was shown in Theorem~4.3 in \cite{gouveia2013polytopes} that every $d$-polytope with 
$d+2$ vertices is combinatorially psd-minimal. In view of Lemma~\ref{lem:binomial}, this fact is a corollary of the following result.

\begin{theorem} \label{thm:n+2}
If a $d$-polytope $P$ has $d+2$ vertices or facets then $I_P$ is binomial.
\end{theorem}

By polarity it suffices to prove the result for the case when $P$ has $d+2$ vertices. We recall the structure of such
polytopes from \cite[Section 6.1]{grunbaum}. Any $d$-polytope with $d+2$ vertices is combinatorially equivalent to a repeated pyramid over a free
sum of two simplices, $\textup{pyr}_r(\Delta_k \oplus \Delta_l)$, with $k,l \geq 1$, $r \geq 0$ and $r+k+l=d$. Since taking pyramids preserves the slack ideal,
it is enough to study the slack ideals of free sums of simplices or, dually, products of simplices. This class of polytopes has a very simple slack matrix structure.

\begin{lemma} \label{lem:n+2 slack matrix}
If $P=\Delta_k \oplus \Delta_l$, then $S_P(x)$ has the zero pattern 
of the vertex-edge incidence matrix of the complete bipartite graph $K_{k+1,l+1}$.
\end{lemma}

\begin{example} \label{ex:poly5}
A polytope in class 5 in Table~\ref{table:list} is of the type $\Delta_1 \oplus \Delta_3$ hence its slack matrix is the vertex-edge incidence matrix of $K_{2,4}$. Therefore,
$$S_P(x) = \begin{bmatrix}
x_1 & x_2 & x_3 & x_4 & 0 & 0 & 0 & 0 \\
0 & 0 & 0 & 0  & x_5 & x_6 & x_7 & x_8 \\
x_9 & 0 & 0 & 0 & x_{10} & 0 & 0 & 0 \\
0 & x_{11} & 0 & 0 & 0 & x_{12}  & 0 & 0 \\
0 & 0 & x_{13} & 0 & 0 & 0 & x_{14} & 0 \\
0 & 0 & 0 & x_{15} & 0 & 0 & 0 & x_{16}
\end{bmatrix}.$$
\end{example}

Let $K$ be the complete bipartite graph associated to $P$ as above and consider a simple cycle $C$ in $K$ with $c$ edges and hence $c$ distinct vertices. Let $M_C(x)$ be the $c \times c$ symbolic matrix whose support is the vertex-edge incidence matrix of $C$
$$ 
\begin{bmatrix}
x_1 & 0 & 0 & \ldots & 0 & x_2 \\
x_3 & x_4 & 0 & \ldots & 0 & 0 \\
0 & x_5 & x_6 & \ldots & 0 & 0 \\
\vdots & & &\vdots & & \vdots\\
0 & 0 & 0 & \ldots & x_{2c-2} & 0 \\
0 & 0 & 0 & \ldots & x_{2c-1}  & x_{2c} 
\end{bmatrix}.
$$
Note that $\det(M_{C}(x))$ is a binomial. 

\begin{proposition} \label{prop:simplicial binomial} 
If $P=\Delta_k \oplus \Delta_{l}$ then $I_P$ is 
generated by the binomials $\det(M_C)$ as $C$ varies over the simple cycles of $K_{k+1,l+1}$.
\end{proposition}

\begin{proof} Let $J$ denote the ideal generated by the binomials $\det(M_C)$ as $C$ varies over the simple cycles in $K_{k+1,l+1}$. The ideal $J$ is the toric ideal of the Lawrence lifting of the vector configuration consisting of the column vectors of the vertex-edge incidence matrix of $K_{k+1,l+1}$ \cite[Corollary~14.12]{GBCP}. Further, toric ideals are saturated \cite[Chapter 12]{GBCP} and hence, 
$J \,:\, (\prod x_i)^\infty = J$. 

Since $d = \dim(P) = k+l$, by Lemma~\ref{lem:n+2 slack matrix}, a $(d+2)$-minor of $S_P(x)$ is a maximal 
minor of the vertex-edge incidence matrix of $K_{k+1,l+1}$. 
Let $M$ be a $(d+2) \times (d+2)$ submatrix of $S_P(x)$. 
Suppose that a row of $M$ has zero or one non-zero entries. In the former case, $\det(M) =0$  and in the latter case, $\det(M)$ is a variable times a $(d+1)$-minor of $S_P(x)$. Repeating this argument for rows and columns, we see that $\det(M)$ is a product of variables times the minor of a submatrix of $S_P(x)$ with at least two non-zero entries per row and column, and hence, 
exactly two non-zero entries per row and column. 
This submatrix is thus the vertex-edge incidence matrix of a disjoint union of simple cycles in $K_{k+1, l+1}$, hence  block diagonal (after permuting rows and columns) with each block indexed by one such cycle. 
 Therefore, $\det(M)$ lies in $J$ and $I_P \subseteq J \,:\, (\prod x_i)^\infty = J$.
Note that $\det(M)$ is not a product of variables, as this would imply that the corresponding minor in $S_P$ is not zero and hence $\rank(S_P) = d+2$ which is a contradiction.

For the reverse inclusion notice that the incidence matrix of a cycle $C$ can be extended to a $(d+2) \times (d+2)$ submatrix of the incidence matrix of $K_{k+1,l+1}$ by picking a vertex on each side of $K_{k+1,l+1}$ 
contained in $C$ and connecting each of these vertices to all vertices on the other side not already in $C$. As before, 
the resulting $(d+2)$-minor is a product of variables times the minor corresponding to $C$. This shows that 
$J \subseteq I_P$, and we conclude that $J=I_P$.
\end{proof}

This finishes the proof of Theorem~\ref{thm:n+2}.

\begin{example}
As we saw in Example~\ref{ex:poly5}, the symbolic slack matrix of a polytope $P$ in combinatorial class 5 
has support equal to the vertex-edge incidence
matrix of $K_{2,4}$. The graph $K_{2,4}$ has six simple cycles all of length four and the minors of the submatrices they index are 
\begin{align*}
x_{3}x_{8}x_{14}x_{15} - x_{4}x_{7}x_{13}x_{16}, \
x_{2}x_{8}x_{12}x_{15} - x_{4}x_{6}x_{11}x_{16}, \
x_{1}x_{8}x_{10}x_{15} - x_{4}x_{5}x_{9}x_{16},\\
x_{2}x_{7}x_{12}x_{13} - x_{3}x_{6}x_{11}x_{14}, \
x_{1}x_{7}x_{10}x_{13} - x_{3}x_{5}x_{9}x_{14}, \
x_{1}x_{6}x_{10}x_{11} - x_{2}x_{5}x_{9}x_{12}.
\end{align*}
By Proposition~\ref{prop:simplicial binomial}, the slack ideal $I_P$ is generated by these binomials.
\end{example}

\begin{theorem} \label{thm:binomial polytopes}
The slack ideals of classes $2, \ldots, 11$ in Table~\ref{table:list} are binomial. 
\end{theorem}

\begin{proof}
The $4$-polytopes in classes $2,4,5,6,7,8,9$ have $6=4+2$ vertices or facets and hence the result follows in these cases from Theorem~\ref{thm:n+2}. We saw in Example~\ref{ex:slack ideal examples} that 
the slack ideal of a polytope in class 3 is binomial. Since classes $10$ and $11$ are dual, we only need to 
show that the slack ideal of a polytope in class 10 is binomial.

The symbolic slack matrix of a polytope in class 10 is the following:
$$S_P(x) = \begin{bmatrix}
0&0&x_1&x_2&0&0&x_3\\
0&x_4&0&x_5&0&0&0\\
0&0&0&x_6&0&x_7&x_8\\
x_9&0&x_{10}&0&0&0&0\\
x_{11}&0&0&0&0&x_{12}&0\\
0&0&x_{13}&0&x_{14}&0&x_{15}\\
0&x_{16}&0&0&x_{17}&0&0\\
0&0&0&0&x_{18}&x_{19}&x_{20}\\
\end{bmatrix}.$$
Using Macaulay~2 one can compute that $I_P$ is binomial. More precisely,\\
$I_P = \langle 
x_8x_{19}-x_7x_{20}, \,\, x_{15}x_{18}-x_{14}x_{20}, \,\, x_3x_{13}-x_1x_{15}, \,\, x_3x_6-x_2x_8, \,\,\\
      x_{10}x_{11}x_{15}x_{19}-x_9x_{12}x_{13}x_{20}, \,\, x_3x_{10}x_{11}x_{19}-x_1x_9x_{12}x_{20}, \,\,
      x_5x_8x_{16}x_{18}-x_4x_6x_{17}x_{20}, \,\, \\
      x_5x_7x_{16}x_{18}-x_4x_6x_{17}x_{19}, \,\,
      x_3x_5x_{16}x_{18}-x_2x_4x_{17}x_{20}, \,\, x_9x_{12}x_{13}x_{18}-x_{10}x_{11}x_{14}x_{19}, \,\,\\
      x_2x_7x_{13}x_{18}-x_1x_6x_{14}x_{19}, \,\, x_5x_8x_{14}x_{16}-x_4x_6x_{15}x_{17}, \,\,
      x_3x_5x_{14}x_{16}-x_2x_4x_{15}x_{17}, \,\, \\
      x_1x_5x_{14}x_{16}-x_2x_4x_{13}x_{17}, \,\,
      x_8x_9x_{12}x_{13}-x_7x_{10}x_{11}x_{15}, \,\, x_3x_7x_{10}x_{11}-x_1x_8x_9x_{12}, \,\,\\
      x_2x_7x_{10}x_{11}-x_1x_6x_9x_{12}, \,\, x_1x_5x_9x_{12}x_{16}x_{18}-x_2x_4x_{10}x_{11}x_{17}x_{19}, \,\,\\
      x_5x_7x_{10}x_{11}x_{14}x_{16}-x_4x_6x_9x_{12}x_{13}x_{17}
      \rangle$.
\end{proof}

We will see at the end of the paper that no other combinatorial class in Table~\ref{table:list} has a binomial ideal. Since the table lists all possible combinatorial types of psd-minimal $4$-polytopes, using Lemma~\ref{lem:binomial} we can conclude that we 
have identified all $4$-polytopes with a binomial polytope ideal.

\begin{corollary} \label{cor:classes 1-11}
Any $4$-polytope in classes 1-11 is combinatorially psd-minimal.
\end{corollary}

\begin{remark} In \cite{McMullen}, McMullen exhibited 11 combinatorial classes of 
$4$-polytopes that are projectively unique and this list is conjectured to be complete. These are precisely the classes $1, \ldots, 11$ in Table~\ref{table:list}. This connection yields an alternate short proof 
of Corollary~\ref{cor:classes 1-11}.  Since psd rank is invariant under projective 
transformations, and there is only one polytope in each of these classes up to projective equivalence, Corollary~\ref{cor:classes 1-11} follows from the psd-minimality of these representative polytopes.
\end{remark}

\section{Four interesting classes of $4$-polytopes (classes 12--15)} \label{sec:polys12-15}

For the remaining classes from Table~\ref{table:list} we want to establish conditions for 
psd-minimality. Since the table lists a psd-minimal instance in each class, every class contains 
psd-minimal polytopes. In this section, we consider the dual pairs 12-13 and 14-15.
In theory, a method to derive conditions for the psd-minimality of polytopes of a fixed combinatorial type is as follows. 
\begin{enumerate}
\item Compute the slack ideal $I_P \subset \RR[x_1,\ldots,x_t]$ of a polytope $P$ in the class 
and let $J_P$ be a copy of $I_P$ in the variables $y_1, \ldots, y_t$. 
\item  Consider the ideal $$K_P = I_P + J_P + \langle y_i^2 - x_i \,\,i=1,\ldots,t \rangle \subset \RR[x_1,\ldots, x_t, y_1, \ldots, y_t].$$  By construction, for any $(x, y)\in \mathcal{V}(K_P)$ the matrix $S_P(y)$ is a Hadamard square root of $S_P(x)$. Thus, the polytope $P$ is psd-minimal if and only if there are $s$, $r\in\RR^t$ such that $(s,r)\in \mathcal{V}(K_P)$ and $S_P=S_P(s)$.

\item Eliminating $y_1,\ldots, y_t$ from $K_P$ we obtain the ideal $K_P \cap \RR[x_1,\ldots,x_t]$. A minimal generating set of this elimination ideal gives the conditions for a polytope in the class of $P$ to be psd-minimal.
\end{enumerate}

In practice, the computation of the slack ideal as defined in Definition~\ref{def:slack ideal} is often prohibitive due to the large number of variables in $S_P(x)$. Therefore, in this section and the next, we rely on various  simplifications. We first illustrate these ideas on the combinatorial class 12.  

\begin{example} [Class 12] \label{ex:poly12}

By scaling the rows and columns of a slack matrix of $P$, an operation that leaves both psd rank and square root rank unchanged, we may assume that several of the variables 
$x_i$ in $S_P(x)$ have been set to one. 
 This allows us to start with the following partially symbolic slack matrix of a polytope in class 12 in only 13 variables
\begin{small}
\begin{align} \label{eq:slack matrix poly 12}
\begin{bmatrix}
 1&  0&  0&  0&  0&  1&  1&  1\\
  x_1&  0&  0&  0&  1&  1&  0&  0\\
   0&  1&  0&  0&  0&  1&x_{12}&x_{13}\\
   0& x_4&  0&  0&x_{11}&  1&  0&  0\\
  x_2&  0&  1&  0&  0&  0&  1&  0\\
   0& x_5& x_7&  0&  0&  0&  1&  0 \\
  x_3&  0&  0&  1&  0&  0&  0&  1\\
   0& x_6&  0& x_9&  0&  0&  0&  1\\
   0&  0& x_8&x_{10}&  1&  0&  0&  0
\end{bmatrix}.
\end{align}
\end{small}

To attain as many ones as we have above, we need to scale rows and columns in a carefully chosen order. 
For instance, in the above matrix, we can start by scaling columns $1,6,7,8$ to get ones in the first row, and 
then rows $2,\ldots,8$ to get the remaining ones in the last three columns. Next we scale columns $2,\ldots,5$ to set the first non-zero entry in each of those columns to one, and finally, we scale the last row to get the remaining one in the matrix. Note that if we are only interested in preserving the usual rank then we may also scale with negative scalars. This allows us to put any matrix with the same zero pattern as the above one into the same form without changing its rank. 
In all remaining examples, we fix variables to one using a similar procedure.

Our goal is to obtain a parametrization of the slack matrices of class 12 with the 
scalings we have fixed above.
Saturating the ideal of all $6$-minors of the above matrix we get the ideal
\begin{align*}
 I & = & \langle x_{13} - 1, x_{12} - 1, x_{11} - 1, x_{10} - 1, x_9 - 1, x_8 - 1, x_7 - 1, \\
 && x_3 - x_6, x_2 - x_5, x_4 + x_5 + x_6 - 1,  x_1 + x_5 + x_6 - 1 \rangle.
\end{align*}
All slack matrices of class 12 of the form \eqref{eq:slack matrix poly 12} must satisfy the 
equations given by the generators of $I$. This fixes all variables except $x_1,x_2,x_3$ yielding
\begin{small}
$$ \begin{bmatrix}
   1&  0&  0&  0&  0&  1&  1&  1\\
  x_1&  0&  0&  0&  1&  1&  0&  0\\
   0&  1&  0&  0&  0&  1&  1&  1\\
   0& x_1&  0&  0&  1&  1&  0&  0\\
  x_2&  0&  1&  0&  0&  0&  1&  0\\
   0& x_2&  1&  0&  0&  0&  1&  0\\
  x_3&  0&  0&  1&  0&  0&  0&  1\\
   0& x_3&  0&  1&  0&  0&  0&  1\\
   0&  0&  1&  1&  1&  0&  0&  0
\end{bmatrix}.$$
\end{small}
Running the  ideal calculation again yields the principal ideal with generator 
$$x_1 + x_2 + x_3 - 1.$$
Thus we can conclude that the matrices of rank $5$ with the same zero-pattern as a slack matrix of class 12 
can be parametrized (up to scalings) as 
\begin{small}
\begin{align} \label{eq:slack matrix parametrization poly 12}
 \begin{bmatrix}
   1&  0&  0&  0&  0&  1&  1&  1\\
  x_1&  0&  0&  0&  1&  1&  0&  0\\
   0&  1&  0&  0&  0&  1&  1&  1\\
   0& x_1&  0&  0&  1&  1&  0&  0\\
  x_2&  0&  1&  0&  0&  0&  1&  0\\
   0& x_2&  1&  0&  0&  0&  1&  0\\
  1-x_1-x_2&  0&  0&  1&  0&  0&  0&  1\\
   0& 1-x_1-x_2&  0&  1&  0&  0&  0&  1\\
   0&  0&  1&  1&  1&  0&  0&  0
\end{bmatrix}\,.
\end{align}
\end{small} 
In particular, for every polytope $P$ in class 12 we can choose a slack matrix $S_P$ of the above form.
If $P$ is psd-minimal then there is a Hadamard square root of $S_P$  of rank at most five. As remarked earlier,  we can scale rows and columns of the Hadamard square root to bring it into the form~\eqref{eq:slack matrix poly 12}. Due to the rank condition, the Hadamard square root has the same symbolic form as~\eqref{eq:slack matrix parametrization poly 12}. 

Let us denote the parameters of $S_P$ by $x_1$ and $x_2$ and the ones of its square root by $y_1$ and $y_2$. 
Therefore, for psd-minimality we must have 
 \begin{align}
 y_1^2 = x_1, \,\, y_2^2 = x_2, \,\, (1-y_1-y_2)^2=1-x_1-x_2.
 \end{align}
 Eliminating $y_1,y_2$ from these equations we obtain the following condition on $x_1,x_2$: 
 \begin{align} \label{eq:conditions poly 12}
 x_1^4 + 2x_1^3x_2 + 3x_1^2x_2^2 + 2x_1x_2^3 + x_2^4 - 2x_1^3 - 2x_2^3 + x_1^2 -  2x_1x_2 + x_2^2 = 0.
 \end{align}

 Thus the psd-minimal polytopes in the combinatorial class 12 are those whose 
 slack matrices can be parametrized (up to scalings) as in \eqref{eq:slack matrix parametrization poly 12} with $x_1,x_2$ satisfying 
 \eqref{eq:conditions poly 12}. The algebraic variety of  \eqref{eq:conditions poly 12} is shown in Figure~\ref{fig:poly12}.
 
 \begin{figure}
 \includegraphics[width=4.5cm]{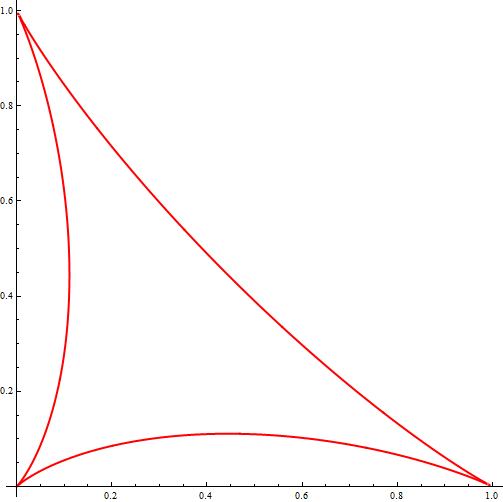} 
 \caption{Set of parameters for which psd-minimality is verified on combinatorial type $12$.}\label{fig:poly12}
 \end{figure}

\end{example}

The above calculation proves the following statement for classes 12 and 13. 

\begin{proposition}\label{prop:polytope 12 13}
A polytope in the combinatorial class 12 (respectively, 13) is psd-minimal if and only if, the entries of its slack matrix 
in the form \eqref{eq:slack matrix parametrization poly 12} (respectively, transpose of \eqref{eq:slack matrix parametrization poly 12}) satisfy the quartic equation \eqref{eq:conditions poly 12}. 
\end{proposition}

This example breaks some conjectures on the behavior of square root rank and psd-minimality which we now discuss.

\begin{enumerate}
\item Up to now, in all known psd-minimal polytopes, the positive square root $\sqrt[+]{S_P}$ 
of the slack matrix $S_P$ had rank $d+1$. It was asked in Problem~$2$ \cite{beasley_et_al:DR:2013:4019} whether this was always the case for psd-minimal polytopes.
If we set $x_1=1/9$ and $x_2=4/9$ in the parametrization \eqref{eq:slack matrix parametrization poly 12}, condition \eqref{eq:conditions poly 12} is verified and the resulting matrix is a slack matrix of a psd-minimal polytope in class 12. However it is easy to see that the positive square root of this slack matrix is not of rank five, in fact, we must take the square root parametrized by $y_1=-1/3$ and $y_2=2/3$ to obtain one of rank five.

\item The support matrix of a matrix $M$ is the $0,1$ matrix with the same zero pattern as $M$. A weaker version of the previous conjecture that is also broken by this example is that the support matrix of a psd-minimal $d$-polytope is always of rank $d+1$.  This was true in all previously known examples, but here it is trivial to check that the support matrix has rank larger than $5$. This, in some sense, provides evidence that the class of psd-minimal polytopes is fundamentally larger than the class of $2$-level polytopes.
\end{enumerate}

 \begin{example} [Class $14$] \label{ex:polytope 14}

We employ the same general ideas that we used for class $12$. 
Scaling the rows and columns of the 
 symbolic slack matrix of a polytope in class 14 to set as many entries to 1 as possible, 
 we obtain the symbolic matrix
\begin{align} \label{eq:slack matrix poly 14}
\begin{small}
\begin{bmatrix}
   1&    1&   1&   1&    0&    0&   0&    0\\
   1&  x_3&   0&   0&  x_7&    0&   0&    0\\
   0&  x_4&   1& x_5&    0& x_{10}&   0&    0\\
   0&    1&   0&   0&  x_8& x_{11}&   0&    0\\
 x_1&    0&   0&   1&    0&    0&   1&    0\\
   0&    0&   0& x_6&    0&    1&   1&    0\\
 x_2&    0&   1&   0&    0&    0&   0&    1\\
   0&    0&   1&   0&    0& x_{12}&   0& x_{16}\\
   1&    0&   0&   0&    1&    0&x_{14}& x_{17}\\
   0&    0&   0&   0&  x_9& x_{14}&x_{15}& 1
\end{bmatrix}.
\end{small}
\end{align}

We then saturate the ideal of the $6$-minors of the above matrix, and 
use all the linear generators to eliminate variables as we did in the previous example.
In addition, we also use some of the quadratic generators to eliminate more variables. 
For instance, the polynomial $x_8x_{17}-1$ is a generator of the ideal and can be 
used to replace $x_{17}$ by $1/x_8$. Since scalings are allowed, we then multiply 
the row by $x_8$ to get rid of denominators. Continuing this way, we arrive at 
a parametrization of the slack matrices of class 14 in the form
\begin{small}
\begin{align} \label{eq:slack matrix parametrization poly 14}
 \begin{bmatrix}
 x_{10}&    1&   1&   1&    0&    0&   0&    0\\
 x_{11}&    1&   0&   0&    1&    0&   0&    0\\
   0&    1&   1&   1&    0&    x_{10}&   0&    0\\
   0&    1&   0&   0&    1&    x_{11}&   0&    0\\
   1&    0&   0&   1&    0&    0&  1 &    0\\
   0&    0&   0&   1&    0&    1&   1&    0\\
 x_{12}&    0&   1&   0&    0&    0&   0&    1\\
   0&    0&   1&   0&    0&    x_{12}&   0&    1\\
 1+x_{12}+x_{11}-x_{10}&    0&   0&   0&    1&    0&   1&    1\\
   0&    0&   0&   0&    1&    1+x_{12}+x_{11}-x_{10}&   1&    1
\end{bmatrix}.
\end{align}
\end{small}
As before, we let $y_1, y_2$ and $y_3$ be the parameters of a symbolic square root of 
the above slack matrix. Then psd-minimality imposes the conditions 
\begin{align}
 y_1^2 = x_{10}, \,\, y_2^2 = x_{11}, \,\, y_3^2 = x_{12}, \,\, (1-y_1+y_2+y_3)^2=1-x_{10}+x_{11}+x_{12}.
\end{align}
Eliminating the $y$ variables, results in the following degree eight algebraic equation:

\begin{tiny}
\begin{equation}\label{eq:conditions poly 14}
\begin{split}
x_{10}^8-4 x_{11} x_{10}^7-4 x_{12} x_{10}^7-4 x_{10}^7+6 x_{11}^2 x_{10}^6+6 x_{12}^2 x_{10}^6+16 x_{11} x_{10}^6+16 x_{11} x_{12} x_{10}^6+ 16 x_{12} x_{10}^6+6 x_{10}^6-\\
4  x_{11}^3 x_{10}^5  -4 x_{12}^3 x_{10}^5-24 x_{11}^2 x_{10}^5-24 x_{11} x_{12}^2 x_{10}^5-24 x_{12}^2 x_{10}^5-24 x_{11} x_{10}^5-24 x_{11}^2 x_{12} x_{10}^5-60 x_{11} x_{12}  x_{10}^5-\\
24 x_{12} x_{10}^5-4 x_{10}^5 +x_{11}^4 x_{10}^4+x_{12}^4 x_{10}^4+16 x_{11}^3 x_{10}^4+16 x_{11} x_{12}^3 x_{10}^4+16 x_{12}^3 x_{10}^4+36 x_{11}^2 x_{10}^4+\\
36 x_{11}^2 x_{12}^2 x_{10}^4+76 x_{11} x_{12}^2 x_{10}^4+36 x_{12}^2 x_{10}^4 +16 x_{11} x_{10}^4+16 x_{11}^3 x_{12} x_{10}^4+76 x_{11}^2 x_{12} x_{10}^4+76 x_{11} x_{12} x_{10}^4+\\
16 x_{12} x_{10}^4+x_{10}^4-4 x_{11}^4 x_{10}^3-4 x_{11} x_{12}^4 x_{10}^3-4 x_{12}^4 x_{10}^3-24 x_{11}^3 x_{10}^3-24 x_{11}^2 x_{12}^3 x_{10}^3-36 x_{11} x_{12}^3 x_{10}^3-\\
24 x_{12}^3 x_{10}^3-24 x_{11}^2 x_{10}^3-24 x_{11}^3 x_{12}^2 x_{10}^3-76 x_{11}^2 x_{12}^2 x_{10}^3-76 x_{11} x_{12}^2 x_{10}^3-24 x_{12}^2 x_{10}^3-4 x_{11} x_{10}^3\\
-4 x_{11}^4 x_{12} x_{10}^3-36 x_{11}^3 x_{12} x_{10}^3-76 x_{11}^2 x_{12} x_{10}^3-36 x_{11} x_{12} x_{10}^3-4 x_{12} x_{10}^3+6 x_{11}^4 x_{10}^2+6 x_{11}^2 x_{12}^4 x_{10}^2+\\
4 x_{11} x_{12}^4 x_{10}^2+6 x_{12}^4 x_{10}^2+16 x_{11}^3 x_{10}^2+16 x_{11}^3 x_{12}^3 x_{10}^2+20 x_{11}^2 x_{12}^3 x_{10}^2+20 x_{11}  x_{12}^3 x_{10}^2+16 x_{12}^3 x_{10}^2+\\
6 x_{11}^2 x_{10}^2+6 x_{11}^4 x_{12}^2 x_{10}^2+20 x_{11}^3 x_{12}^2 x_{10}^2+82 x_{11}^2 x_{12}^2 x_{10}^2+20 x_{11} x_{12}^2 x_{10}^2+6 x_{12}^2 x_{10}^2+4 x_{11}^4 x_{12} x_{10}^2+\\
20 x_{11}^3 x_{12} x_{10}^2+20 x_{11}^2 x_{12} x_{10}^2+4 x_{11} x_{12} x_{10}^2-4 x_{11}^4 x_{10}-4 x_{11}^3 x_{12}^4 x_{10}+4 x_{11}^2 x_{12}^4 x_{10}+4 x_{11} x_{12}^4 x_{10}-\\
4 x_{12}^4 x_{10}-4 x_{11}^3 x_{10}-4 x_{11}^4 x_{12}^3 x_{10}+4 x_{11}^3 x_{12}^3 x_{10}-32 x_{11}^2 x_{12}^3 x_{10}+4 x_{11} x_{12}^3 x_{10}-4 x_{12}^3 x_{10}+\\
4 x_{11}^4 x_{12}^2 x_{10}-32 x_{11}^3 x_{12}^2 x_{10}-32 x_{11}^2 x_{12}^2 x_{10}+4 x_{11} x_{12}^2 x_{10}+4 x_{11}^4 x_{12} x_{10}+4 x_{11}^3 x_{12}x_{10}+4 x_{11}^2 x_{12} x_{10}+\\
x_{11}^4+x_{11}^4 x_{12}^4-4 x_{11}^3 x_{12}^4+6 x_{11}^2 x_{12}^4-4 x_{11} x_{12}^4+x_{12}^4-4 x_{11}^4 x_{12}^3+4 x_{11}^3 x_{12}^3+4 x_{11}^2 x_{12}^3-4 x_{11} x_{12}^3+\\
6 x_{11}^4 x_{12}^2+4 x_{11}^3 x_{12}^2+6 x_{11}^2 x_{12}^2-4 x_{11}^4 x_{12}-4 x_{11}^3 x_{12} = 0.  
\end{split}
\end{equation}
\end{tiny}

 Thus the psd-minimal polytopes in the combinatorial class 14 are those whose 
 slack matrices can be parametrized (up to scalings) as in \eqref{eq:slack matrix parametrization poly 14} with $x_{10},x_{11},x_{12}$ satisfying 
 \eqref{eq:conditions poly 14}. This algebraic variety is shown in Figure~\ref{fig:poly14}.
 
 \begin{figure}
 \includegraphics[width=6cm]{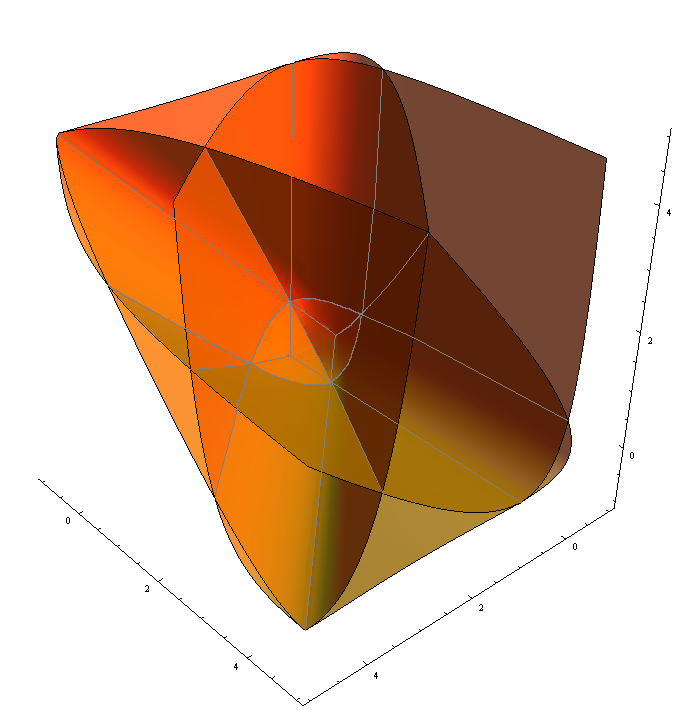} 
 \caption{Set of parameters for which psd-minimality is verified on combinatorial type $14$.}\label{fig:poly14}
 \end{figure}

\end{example}

\begin{proposition}\label{prop:polytope 14 15}
A polytope in the combinatorial class 14 (respectively, 15) is psd-minimal if and only if, the entries of its slack matrix 
in the form \eqref{eq:slack matrix parametrization poly 14} (respectively, transpose of \eqref{eq:slack matrix parametrization poly 14}) satisfy the degree eight equation \eqref{eq:conditions poly 14}. 
\end{proposition}

\section{The remaining cases: classes 16--31} \label{sec:polys16-31}

In this section we briefly describe the methods to characterize psd-minimality 
in the remaining combinatorial classes from Table~\ref{table:list}.

\subsection{The $3$-cube}

To illustrate the techniques, and because we use this in the forthcoming arguments, we reprove the conditions under which a cube (or dually, an octahedron) is psd-minimal.

As before, we scale the rows and columns of the symbolic slack matrix of a cube to fix $15$ entries to be one, obtaining the following matrix:
\begin{small}
\begin{align} \label{eq:slack cube}
\begin{bmatrix}
 1&  0&  1&  0&  1&  0\\
 1&  0& x_1&  0&  0&  1\\
 1&  0&  0&  1& x_6&  0\\
 1&  0&  0& x_3&  0& x_9\\
 0&  1&  1&  0& x_7&  0\\
 0&  1& x_2&  0&  0& x_{10}\\
 0&  1&  0& x_4& x_8&  0\\
 0&  1&  0& x_5& 0& x_{11}
\end{bmatrix}.
\end{align}
\end{small}
Computing the saturation of the ideal of $5$-minors of this matrix, we obtain algebraic conditions 
that allow us to eliminate variables and get the matrix
\begin{small}
\begin{align} \label{eq:param cube}
\begin{bmatrix}
 1&  0&  1          &  0                &  1&  0\\
 1&  0& x_1         &  0                &  0&  1\\
 1&  0&  0          &  1                & x_6&  0\\
 1&  0&  0          & x_1x_9-x_9+1      &  0& x_9\\
 0&  1&  1          &  0                & x_7&  0\\
 0&  1& x_1x_7-x_7+1&  0                &  0& x_7\\
 0&  1&  0          & 1                 & x_6+x_7-1&  0\\
 0&  1&  0          & x_1x_{11}-x_{11}+1& 0& x_{11}
\end{bmatrix}.
\end{align}
\end{small}
Instead of eliminating more variables, we now impose psd-minimality requirements on this matrix.  
If matrix \eqref{eq:slack cube} and one of its square roots have rank four, both have the parametrization \eqref{eq:param cube}. 
Let $y_1$ and $y_{11}$ denote the square roots of $x_1$ and $x_{11}$ respectively. Entry $(8,4)$ then implies that $$y_1^2y_{11}^2-y_{11}^2+1=x_1x_{11}-x_{11}+1=(y_1y_{11}-y_{11}+1)^2.$$
Simplifying, we obtain $y_1(1-y_1)(1-y_{11})=0$ which implies either $y_1$ or $y_{11}$ is $1$. Similarly, looking at the other entries of 
\eqref{eq:param cube} 
we get that $y_1$ or $y_7$ must be $1$, $y_1$ or $y_9$ must be $1$, and $y_6$ or $y_7$ must be $1$. 

Therefore, if $y_1$ is not $1$ then $y_9, y_7$ and $y_{11}$ must be $1$. Making this substitution and 
recomputing the ideal we see that $y_6$ must be $1$. If $y_1$ is $1$, further computations force us to decide if $y_6$ or $y_7$ is $1$. In the end we get three possibilities for a Hadamard square root of rank four:
\begin{small}
\begin{align} \label{eq:cube_slacks}
\begin{bmatrix}
 1&  0&  1&  0&  1&  0\\
 1&  0& y_1&  0&  0&  1\\
 1&  0&  0&  1& 1&  0\\
 1&  0&  0& y_1&  0& 1\\
 0&  1&  1&  0& 1&  0\\
 0&  1& y_1&  0&  0& 1\\
 0&  1&  0& 1& 1&  0\\
 0&  1&  0& y_1& 0& 1
\end{bmatrix} \textrm{ or }
\begin{bmatrix}
 1&  0&  1&  0&  1&  0\\
 1&  0&  1&  0&  0&  1\\
 1&  0&  0&  1& y_6&  0\\
 1&  0&  0&  1&  0& y_6\\
 0&  1&  1&  0& 1&  0\\
 0&  1&  1&  0&  0& 1\\
 0&  1&  0&  1& y_6&  0\\
 0&  1&  0&  1& 0& y_6
\end{bmatrix} \textrm{ or }
\begin{bmatrix}
 1&  0&  1&  0&  1&  0\\
 1&  0&  1&  0&  0&  1\\
 1&  0&  0&  1& 1&  0\\
 1&  0&  0&  1&  0& 1\\
 0&  1&  1&  0& y_7&  0\\
 0&  1&  1&  0&  0& y_7\\
 0&  1&  0&  1& y_7&  0\\
 0&  1&  0&  1& 0& y_7
\end{bmatrix}.
\end{align}
\end{small}
The three possibilities above are the same up to permuting rows and columns, which yields the following result.
\begin{proposition}\label{prop:cube_classification}
A $3$-cube is psd-minimal if and only if its slack matrix, up to scalings and permutations of rows and columns is 
of the form
\begin{equation}
\begin{bmatrix}  \label{eq:param psd minimal cube}
 1&  0&  1&  0&  1&  0\\
 1&  0&  1&  0&  0&  1\\
 1&  0&  0&  1& x&  0\\
 1&  0&  0&  1&  0& x\\
 0&  1&  1&  0& 1&  0\\
 0&  1&  1&  0&  0& 1\\
 0&  1&  0&  1& x&  0\\
 0&  1&  0&  1& 0& x
\end{bmatrix}.
\end{equation}
\end{proposition}

\begin{remark}\label{rem:psd minimal cube dependencies}
From the proof of Proposition~\ref{prop:cube_classification} it follows that a psd-minimal $3$-cube has a 
slack matrix of the same support as \eqref{eq:slack cube}
such that both 
the slack matrix, and each of its Hadamard square roots of rank four, have the property that 
their first, second, third and fourth columns are linearly dependent, and also the third, fourth, fifth and sixth columns 
are dependent.
\end{remark}

Furthermore, the matrix in \eqref{eq:param psd minimal cube} is the slack matrix of the Cartesian product of the unit segment and the trapezoid with vertices $\{(0,0),(1,0),(0,1),(x,1)\}$, hence by Corollary~\ref{cor:proj_equivalent slack matrix}, a $3$-cube is psd-minimal if and only if it is projectively equivalent to the product of a segment and a trapezoid. Note that Proposition~\ref{prop:cube_classification}, together with Proposition~\ref{prop:combinatorial classification in 3-space} and Theorem~\ref{thm:n+2} finishes an alternate proof of the complete characterization of psd-minimality in $\RR^3$, as stated in Theorem~\ref{thm:classification in 3-space}.

In what follows we make extensive use of the above idea of imposing psd-minimality requirements during 
the process of parametrizing the slack matrix to make some entries constant at the price of branching the computation. This makes the computations possible and easier to track. We can also explicitly use the parametrization~\eqref{eq:param psd minimal cube} whenever a $4$-polytope or its dual has a cubical or octahedral facet, since facets of psd-minimal polytopes must be psd-minimal. This allows us to set many variables to $1$ at the start, 
leading to big computational savings. These tricks are enough to complete the characterization of psd-minimality in the remaining classes from Table~\ref{table:list}.

\subsection{Classes $16$--$31$}
We first provide the full result for classes $19$ through $27$ leaving the remaining classes  to be explored in more detail after that. 
The computations follow the exact same ideas as in the previous sections, and hence are not included in the paper. 

\begin{proposition}\label{prop:classes 16-29}
A polytope of combinatorial type 16, 18, 20, 22, 24 or 26 (respectively 17, 19, 21, 23, 25 or 27) is psd-minimal if and only if, after scaling and permuting rows and columns, its slack matrix (respectively, its transpose) 
is of the following form: 
\begin{small}
$$
\begin{bmatrix}
  {1}&  {0}&  {1}&  {0}&    {1}&   {0}& 0\\
  {1}&  {0}&  {1}&  {0}&    {0}&   {1}& 0\\
  {1}&  {0}&  {0}&  {1}&    {1}&   {0}& 0\\
  {1}&  {0}&  {0}&  {1}&    {0}&   {1}& 0\\
  {0}&  {1}&  {1}&  {0}&    {x}&   {0}& 0\\
  {0}&  {1}&  {1}&  {0}&    {0}&   {x}& 0\\
  {0}&  {1}&  {0}&  {1}&    {x}&   {0}& 0\\
  {0}&  {1}&  {0}&  {1}&    {0}&   {x}& 0\\
    0&    0&    0&    0&      0&     0& 1
\end{bmatrix},
\begin{bmatrix}
1& 1&   0&    0&    0&    0&      0&     0\\
0& 0& {1}&  {0}&  {1}&  {0}&    {1}&   {0}\\
0& 0& {1}&  {0}&  {1}&  {0}&    {0}&   {1}\\
0& 1& {1}&  {0}&  {0}&  {1}&    {1}&   {0}\\
0& 1& {1}&  {0}&  {0}&  {1}&    {0}&   {1}\\
1& 0& {0}&  {1}&  {1}&  {0}&    {x}&   {0}\\
1& 0& {0}&  {1}&  {1}&  {0}&    {0}&   {x}\\
0& 0& {0}&  {1}&  {0}&  {1}&    {x}&   {0}\\
0& 0& {0}&  {1}&  {0}&  {1}&    {0}&   {x}
\end{bmatrix},  $$ $$
\begin{bmatrix}
1&    0&    0&    0&    0&      1&     0\\
1&    0&    0&    0&    0&      0&     1\\
0&  {1}&  {0}&  {1}&  {0}&    {1}&   {0}\\
0&  {1}&  {0}&  {1}&  {0}&    {0}&   {1}\\
0&  {1}&  {0}&  {0}&  {1}&    {1}&   {0}\\
0&  {1}&  {0}&  {0}&  {1}&    {0}&   {1}\\
0&  {0}&  {1}&  {1}&  {0}&    {x}&   {0}\\
0&  {0}&  {1}&  {1}&  {0}&    {0}&   {x}\\
0&  {0}&  {1}&  {0}&  {1}&    {x}&   {0}\\
0&  {0}&  {1}&  {0}&  {1}&    {0}&   {x}
\end{bmatrix}, 
\begin{bmatrix}
1&1&1&    0&    0&    0&    0&      0&     0\\
1&1&0&  {1}&  {0}&  {1}&  {0}&    {1}&   {0}\\
1&0&0&  {1}&  {0}&  {1}&  {0}&    {0}&   {1}\\
0&0&0&  {1}&  {0}&  {0}&  {1}&    {1}&   {0}\\
1&0&1&  {1}&  {0}&  {0}&  {1}&    {0}&   {1}\\
0&x&0&  {0}&  {1}&  {1}&  {0}&    {x}&   {0}\\
0&0&0&  {0}&  {1}&  {1}&  {0}&    {0}&   {x}\\
0&x&1&  {0}&  {1}&  {0}&  {1}&    {x}&   {0}\\
0&0&1&  {0}&  {1}&  {0}&  {1}&    {0}&   {x}
\end{bmatrix},$$  $$
\begin{bmatrix}
      {1}&      {0}&      {1}&      {0}&        {1}&       {0} & 1 & 0 & 0 & 0\\
      {1}&      {0}&      {1}&      {0}&        {0}&       {1} & 1 & 1 & 1 & 0\\
      {1}&      {0}&      {0}&      {1}&        {1}&       {0} & 1 & x & 0 & 1\\
      {1}&      {0}&      {0}&      {1}&        {0}&       {1} & 0 & x & 0 & 0\\
      {0}&      {1}&      {1}&      {0}&        {x}&       {0} & x & 0 & x & x\\
      {0}&      {1}&      {1}&      {0}&        {0}&       {x} & 0 & 0 & x & 0\\
      {0}&      {1}&      {0}&      {1}&        {x}&       {0} & 0 & 0 & 0 & 1\\
      {0}&      {1}&      {0}&      {1}&        {0}&       {x} & 0 & x & x & 1\\
        0&        0&        0&        0&          0&         0 & 1 & 1 & 1 & 1\\
        1&        1&        1&        1&          1&         x & 0 & 0 & 0 & 0
\end{bmatrix},
\begin{bmatrix}
  {1}&  {0}&  {1}&  {0}&    {1}&   {0} &  {1}&  {0}&  {1}&  {0}&    {1}&   {0}\\
  {1}&  {0}&  {1}&  {0}&    {0}&   {1} &  {1}&  {0}&  {1}&  {0}&    {0}&   {1}\\
  {1}&  {0}&  {0}&  {1}&    {1}&   {0} &  {1}&  {0}&  {0}&  {1}&    {1}&   {0} \\
  {1}&  {0}&  {0}&  {1}&    {0}&   {1} &  {1}&  {0}&  {0}&  {1}&    {0}&   {1}\\
  {0}&  {1}&  {1}&  {0}&    {x}&   {0} &  {0}&  {1}&  {1}&  {0}&    {x}&   {0}\\
  {0}&  {1}&  {1}&  {0}&    {0}&   {x} &  {0}&  {1}&  {1}&  {0}&    {0}&   {x}\\
  {0}&  {1}&  {0}&  {1}&    {x}&   {0} &  {0}&  {1}&  {0}&  {1}&    {x}&   {0}\\
  {0}&  {1}&  {0}&  {1}&    {0}&   {x} &  {0}&  {1}&  {0}&  {1}&    {0}&   {x}\\
    0&    0&    0&    0&      0&      0&   y &    1&    1&    y&      y&    x\\
   y &    1&    1&    y&      y&      x&    0&    0&    0&    0&      0&    0
\end{bmatrix}.$$
\end{small}
\end{proposition}

A point to note is that in all these cases, the positive square root and the support of the slack matrix are of rank five. Thus these classes fail to provide counter-examples to the conjectures that were disproved 
by the polytopes in classes 12-15 as discussed before. Furthermore, unlike in classes 12-15 where the 
spaces of slack matrices of psd-minimal polytopes were higher degree algebraic varieties, in the above cases, these are just affine spaces. 

We now come to the last two dual pairs of combinatorial classes from Table~\ref{table:list}.

\begin{proposition} \label{prop:polys 28-29}
A polytope of combinatorial type 28 (respectively 29) is psd-minimal if and only if, after scaling and permuting rows and columns, its slack matrix (respectively, its transpose) 
is of one of the following forms: 
$$\begin{bmatrix}
0&  {1}&  {0}&  {1}&  {0}&    {1}&   {0}\\
0&  {1}&  {0}&  {1}&  {0}&    {0}&   {1}\\
0&  {1}&  {0}&  {0}&  {1}&    {x}&   {0}\\
0&  {1}&  {0}&  {0}&  {1}&    {0}&   {x}\\
0&  {0}&  {1}&  {1}&  {0}&    {1}&   {0}\\
0&  {0}&  {1}&  {1}&  {0}&    {0}&   {1}\\
0&  {0}&  {1}&  {0}&  {1}&    {x}&   {0}\\
0&  {0}&  {1}&  {0}&  {1}&    {0}&   {x}\\
1&    0&    0&    1&    0&      1&     0\\
1&    0&    0&    1&    0&      0&     1\\ 
1&    0&    0&    0&    1&      x&     0\\
1&    0&    0&    0&    1&      0&     x
\end{bmatrix},  \begin{bmatrix}
0&  {1}&  {0}&  {1}&  {0}&    {1}&   {0}\\
0&  {1}&  {0}&  {1}&  {0}&    {0}&   {1}\\
0&  {1}&  {0}&  {0}&  {1}&    {1}&   {0}\\
0&  {1}&  {0}&  {0}&  {1}&    {0}&   {1}\\
0&  {0}&  {1}&  {1}&  {0}&    {y}&   {0}\\
0&  {0}&  {1}&  {1}&  {0}&    {0}&   {y}\\
0&  {0}&  {1}&  {0}&  {1}&    {y}&   {0}\\
0&  {0}&  {1}&  {0}&  {1}&    {0}&   {y}\\
1&    0&    0&    1&    0&      x&     0\\
1&    0&    0&    1&    0&      0&     x\\ 
1&    0&    0&    0&    1&      x&     0\\
1&    0&    0&    0&    1&      0&     x
\end{bmatrix}$$
\end{proposition}

While it is still true that the positive square root and the support of these slack matrices have rank five, it differs from the previous instances since the space of slack matrices is a union
of two components, each an affine space. One can check that the first case corresponds to polytopes projectively equivalent to the product of a triangle and a trapezoid, while the second corresponds to those projectively equivalent to a prism over a wedge.

The computation involved in Proposition~\ref{prop:polys 28-29} is similar to those in Proposition~\ref{prop:classes 16-29}. The only 
difference is that while in classes 16-27 all the slack matrices obtained by the disjunctive process collapse to the same one, in classes 28 and 29 two distinct slack matrices remain at the end.

We now consider the $4$-cube and its dual.
Once again we will see that the space of slack matrices is the union of two components, however the techniques used are more involved than in the previous cases.

\begin{example} [Class $30$: the $4$-cube] \label{ex:4cube} 

We start by reordering rows and columns of the slack matrix so that the upper left $8 \times 6$ submatrix corresponds to the slack matrix of a facet of the $4$-cube. If the $4$-cube is psd-minimal, this facet is as well, 
and hence, we can assume that this submatrix is as in Proposition~\ref{prop:cube_classification}:
\begin{small}
$$
\begin{pmatrix}
  \red{1}&    \red{0}&   \red{1}&  \red{0}&      \red{1}&    \red{0} & x_{25}& 0\\
  \red{1}&    \red{0}&   \red{1}&  \red{0}&      \red{0}&    \red{1} & x_{26}& 0\\ 
  \red{1}&    \red{0}&   \red{0}&  \red{1}&      \red{x}&    \red{0} & x_{27}& 0\\
  \red{1}&    \red{0}&   \red{0}&  \red{1}&      \red{0}&    \red{x} & x_{28}& 0\\ 
  \red{0}&    \red{1}&   \red{1}&  \red{0}&      \red{0}&    \red{0} & x_{29}& 0\\
  \red{0}&    \red{1}&   \red{1}&  \red{0}&      \red{0}&    \red{0} & x_{30}& 0\\
  \red{0}&    \red{1}&   \red{0}&  \red{1}&      \red{x}&    \red{0} & x_{31}& 0\\
  \red{0}&    \red{1}&   \red{0}&  \red{1}&      \red{0}&    \red{x} & x_{32}& 0\\
      x_1&        {0}&       x_9&      {0}&       x_{17}&        {0} & 0& x_{33}\\
      x_2&        {0}&    x_{10}&      {0}&          {0}&     x_{21} & 0& x_{34}\\
      x_3&        {0}&       {0}&   x_{13}&       x_{18}&        {0} & 0& x_{35}\\
      x_4&        {0}&       {0}&   x_{14}&          {0}&     x_{22} & 0& x_{36}\\
      {0}&        x_5&    x_{11}&      {0}&       x_{19}&        {0} & 0& x_{37}\\
      {0}&        x_6&    x_{12}&      {0}&          {0}&     x_{23} & 0& x_{38}\\
      {0}&        x_7&       {0}&   x_{15}&       x_{20}&        {0} & 0& x_{39}\\
      {0}&        x_8&       {0}&   x_{16}&          {0}&     x_{24} & 0& x_{40}
\end{pmatrix}.$$
\end{small}
There are still too many variables for computations, but we can repeat the idea on the lower left submatrix, which also corresponds to the slack matrix of a $3$-cube. Recall that we already used scalings and permutations of the first six columns to put the top left $8 \times 6$ submatrix into the present form.
Therefore we can only assume that the bottom left $8 \times 6$ submatrix is of one of the forms in \eqref{eq:cube_slacks} up to column scalings.
Since we want to keep the top left submatrix intact, we cannot freely permute to reduce the three possible matrices in \eqref{eq:cube_slacks} 
to just one as we did in 
Proposition~\ref{prop:cube_classification}. However, there is enough freedom to reduce it to one of the last two types in \eqref{eq:cube_slacks}. 
In addition, scaling the last two columns of the full slack matrix, we get two possibilities:
\begin{tiny} 
$$\begin{pmatrix}
  \red{1}&    \red{0}&   \red{1}&  \red{0}&      \red{1}&    \red{0} & 1& 0\\
  \red{1}&    \red{0}&   \red{1}&  \red{0}&      \red{0}&    \red{1} & x_9& 0\\ 
  \red{1}&    \red{0}&   \red{0}&  \red{1}&      \red{1}&    \red{0} & x_{10}& 0\\
  \red{1}&    \red{0}&   \red{0}&  \red{1}&      \red{0}&    \red{1} & x_{11}& 0\\ 
  \red{0}&    \red{1}&   \red{1}&  \red{0}&      \red{x_1}&    \red{0} & x_{12}& 0\\
  \red{0}&    \red{1}&   \red{1}&  \red{0}&      \red{0}&    \red{x_1} & x_{13}& 0\\
  \red{0}&    \red{1}&   \red{0}&  \red{1}&      \red{x_1}&    \red{0} & x_{14}& 0\\
  \red{0}&    \red{1}&   \red{0}&  \red{1}&      \red{0}&    \red{x_1} & x_{15}& 0\\
  \blue{x_2}&  \blue{0}&  \blue{x_4}&  \blue{0}&    \blue{x_6}&   \blue{0} & 0& 1\\
  \blue{x_2}&  \blue{0}&  \blue{x_4}&  \blue{0}&    \blue{0}&   \blue{x_8} & 0& x_{16}\\
  \blue{x_2}&  \blue{0}&  \blue{0}&  \blue{x_5}&    \blue{x_6x_7}&   \blue{0} & 0& x_{17}\\
  \blue{x_2}&  \blue{0}&  \blue{0}&  \blue{x_5}&    \blue{0}&   \blue{x_8x_7} & 0& x_{18}\\
  \blue{0}&  \blue{x_3}&  \blue{x_4}&  \blue{0}&    \blue{x_6}&   \blue{0} & 0& x_{19}\\
  \blue{0}&  \blue{x_3}&  \blue{x_4}&  \blue{0}&    \blue{0}&   \blue{x_8} & 0& x_{20}\\
  \blue{0}&  \blue{x_3}&  \blue{0}&  \blue{x_5}&    \blue{x_6x_7}&   \blue{0} & 0& x_{21}\\
  \blue{0}&  \blue{x_3}&  \blue{0}&  \blue{x_5}&    \blue{0}&   \blue{x_8x_7} & 0& x_{22}
\end{pmatrix} \hfill
\begin{pmatrix}
  \red{1}&    \red{0}&   \red{1}&  \red{0}&      \red{1}&    \red{0} & 1& 0\\
  \red{1}&    \red{0}&   \red{1}&  \red{0}&      \red{0}&    \red{1} & x_9& 0\\ 
  \red{1}&    \red{0}&   \red{0}&  \red{1}&      \red{1}&    \red{0} & x_{10}& 0\\
  \red{1}&    \red{0}&   \red{0}&  \red{1}&      \red{0}&    \red{1} & x_{11}& 0\\ 
  \red{0}&    \red{1}&   \red{1}&  \red{0}&      \red{x_1}&    \red{0} & x_{12}& 0\\
  \red{0}&    \red{1}&   \red{1}&  \red{0}&      \red{0}&    \red{x_1} & x_{13}& 0\\
  \red{0}&    \red{1}&   \red{0}&  \red{1}&      \red{x_1}&    \red{0} & x_{14}& 0\\
  \red{0}&    \red{1}&   \red{0}&  \red{1}&      \red{0}&    \red{x_1} & x_{15}& 0\\
  \blue{x_2}&  \blue{0}&  \blue{x_4}&  \blue{0}&    \blue{x_6}&   \blue{0} & 0& 1\\
  \blue{x_2}&  \blue{0}&  \blue{x_4}&  \blue{0}&    \blue{0}&   \blue{x_8} & 0& x_{16}\\
  \blue{x_2}&  \blue{0}&  \blue{0}&  \blue{x_5}&    \blue{x_6}&   \blue{0} & 0& x_{17}\\
  \blue{x_2}&  \blue{0}&  \blue{0}&  \blue{x_5}&    \blue{0}&   \blue{x_8} & 0& x_{18}\\
  \blue{0}&  \blue{x_3}&  \blue{x_4}&  \blue{0}&    \blue{x_6x_7}&   \blue{0} & 0& x_{19}\\
  \blue{0}&  \blue{x_3}&  \blue{x_4}&  \blue{0}&    \blue{0}&   \blue{x_8x_7} & 0& x_{20}\\
  \blue{0}&  \blue{x_3}&  \blue{0}&  \blue{x_5}&    \blue{x_6x_7}&   \blue{0} & 0& x_{21}\\
  \blue{0}&  \blue{x_3}&  \blue{0}&  \blue{x_5}&    \blue{0}&   \blue{x_8x_7} & 0& x_{22}
\end{pmatrix}.$$
\end{tiny}

We then apply our techniques to each of them, with the added idea that we not always take the full ideal of the $6$-minors, opting instead to work with minors of submatrices to reduce the size of 
computations while still eliminating variables. There are several branching points but in the end we obtain just two distinct possibilities, as described in the proposition below.
\end{example}

\begin{proposition}\label{prop:polytope 30 31}
 A combinatorial $4$-cube (respectively a combinatorial $4$-cross polytope) is psd-minimal if and only if, after scaling and permuting rows and columns its slack matrix (respectively its transpose) 
 has one of the following forms:
 $$\begin{pmatrix}
  \red{1}&    \red{0}&   \red{1}&  \red{0}&      \red{1}&    \red{0} & 1& 0\\
  \red{1}&    \red{0}&   \red{1}&  \red{0}&      \red{0}&    \red{1} & 1& 0\\ 
  \red{1}&    \red{0}&   \red{0}&  \red{1}&      \red{1}&    \red{0} & y& 0\\
  \red{1}&    \red{0}&   \red{0}&  \red{1}&      \red{0}&    \red{1} & y& 0\\ 
  \red{0}&    \red{1}&   \red{1}&  \red{0}&      \red{x}&    \red{0} & 1& 0\\
  \red{0}&    \red{1}&   \red{1}&  \red{0}&      \red{0}&    \red{x} & 1& 0\\
  \red{0}&    \red{1}&   \red{0}&  \red{1}&      \red{x}&    \red{0} & y& 0\\
  \red{0}&    \red{1}&   \red{0}&  \red{1}&      \red{0}&    \red{x} & y& 0\\
  \blue{1}&  \blue{0}&  \blue{1}&  \blue{0}&    \blue{1}&   \blue{0} & 0& 1\\
  \blue{1}&  \blue{0}&  \blue{1}&  \blue{0}&    \blue{0}&   \blue{1} & 0& 1\\
  \blue{1}&  \blue{0}&  \blue{0}&  \blue{1}&    \blue{1}&   \blue{0} & 0& y\\
  \blue{1}&  \blue{0}&  \blue{0}&  \blue{1}&    \blue{0}&   \blue{1} & 0& y\\
  \blue{0}&  \blue{1}&  \blue{1}&  \blue{0}&    \blue{x}&   \blue{0} & 0& 1\\
  \blue{0}&  \blue{1}&  \blue{1}&  \blue{0}&    \blue{0}&   \blue{x} & 0& 1\\
  \blue{0}&  \blue{1}&  \blue{0}&  \blue{1}&    \blue{x}&   \blue{0} & 0& y\\
  \blue{0}&  \blue{1}&  \blue{0}&  \blue{1}&    \blue{0}&   \blue{x} & 0& y
\end{pmatrix}  \textrm{ or }
\begin{pmatrix}
  \red{1}&    \red{0}&   \red{1}&  \red{0}&      \red{1}&    \red{0} & 1& 0\\
  \red{1}&    \red{0}&   \red{1}&  \red{0}&      \red{0}&    \red{1} & 1& 0\\ 
  \red{1}&    \red{0}&   \red{0}&  \red{1}&      \red{1}&    \red{0} & 1& 0\\
  \red{1}&    \red{0}&   \red{0}&  \red{1}&      \red{0}&    \red{1} & 1& 0\\ 
  \red{0}&    \red{1}&   \red{1}&  \red{0}&      \red{x}&    \red{0} & y& 0\\
  \red{0}&    \red{1}&   \red{1}&  \red{0}&      \red{0}&    \red{x} & y& 0\\
  \red{0}&    \red{1}&   \red{0}&  \red{1}&      \red{x}&    \red{0} & y& 0\\
  \red{0}&    \red{1}&   \red{0}&  \red{1}&      \red{0}&    \red{x} & y& 0\\
  \blue{1}&  \blue{0}&  \blue{1}&  \blue{0}&    \blue{1}&   \blue{0} & 0& 1\\
  \blue{1}&  \blue{0}&  \blue{1}&  \blue{0}&    \blue{0}&   \blue{1} & 0& 1\\
  \blue{1}&  \blue{0}&  \blue{0}&  \blue{1}&    \blue{1}&   \blue{0} & 0& 1\\
  \blue{1}&  \blue{0}&  \blue{0}&  \blue{1}&    \blue{0}&   \blue{1} & 0& 1\\
  \blue{0}&  \blue{1}&  \blue{1}&  \blue{0}&    \blue{x}&   \blue{0} & 0& y\\
  \blue{0}&  \blue{1}&  \blue{1}&  \blue{0}&    \blue{0}&   \blue{x} & 0& y\\
  \blue{0}&  \blue{1}&  \blue{0}&  \blue{1}&    \blue{x}&   \blue{0} & 0& y\\
  \blue{0}&  \blue{1}&  \blue{0}&  \blue{1}&    \blue{0}&   \blue{x} & 0& y
\end{pmatrix}.$$
 \end{proposition}
 
 One can check that the first slack matrix is the slack matrix of the product of two trapezoids, namely the ones with vertices $\{(0,0),(1,0),(0,1),(1,x)\}$ and $\{(0,0),(1,0),(0,1),(1,y)\}$,
 while the second one is the slack matrix of the polytope with vertices $(\{0\} \times \{0,1\}^3) \cup (\{1\} \times \{0,1\} \times \{0,x\} \times \{0,y\})$, a cubical prismoid.

\section{Open questions}

In this paper we classified all psd-minimal  $4$-polytopes.  Such a classification in 
higher dimensions seems unwieldy with the current techniques, although it might be 
possible to tackle $5$-polytopes using similar ideas to those in this paper.
For any fixed $d$ there is only a finite number of combinatorial classes of psd-minimal $d$-polytopes.
An answer to the following question, besides being of independent interest, would also be useful for efficiently enumerating these classes.

\begin{problem}
What is the maximum number of facets in a psd-minimal $d$-polytope?
\end{problem}

By the invariance of psd rank under polarity, this is equivalent to asking how many vertices a psd-minimal $d$-polytope can have. For $2$-level polytopes in $\RR^d$, it was shown in \cite{gouveia2010thetabodies} that the number of vertices and facets cannot exceed $2^d$ and that this bound is tight. Since $2$-level polytopes are all psd-minimal, $2^d$ is a lower bound for the maximum number of facets in a psd-minimal $d$-polytope, and we suspect that it is 
in fact also an upper bound.

It might be easier to classify psd-minimal polytopes of a fixed combinatorial type. For instance, 
$[0,1]^d$ is $2$-level and thus a psd-minimal $d$-polytope. 

\begin{problem} What are the precise conditions for psd-minimality of a $d$-cube (cross-polytope)?
\end{problem}

The cube $[0,1]^d$ is a product of psd-minimal polytopes, and it is natural to ask about the connection between psd-minimality and standard polytope operations.

\begin{problem}
How does psd-minimality behave under the polytope operations of free sum, product and join?
\end{problem}

We saw that, unlike a $d$-cube, some polytopes are combinatorially psd-minimal, in particular, those with a binomial slack ideal. All known combinatorially psd-minimal polytopes have a binomial slack ideal (see Section~\ref{sec:slackideals}).

\begin{problem} Is the slack ideal of a combinatorially psd-minimal polytope always binomial?
\end{problem}

All $d$-polytopes with at most $d+2$ vertices have a binomial slack ideal. We also saw in Section~\ref{sec:slackideals} that there are psd-minimal $d$-polytopes with more than $d+2$ vertices that have binomial slack ideals.

\begin{problem} What can be said about the combinatorics of a polytope with a binomial slack ideal?
\end{problem}

Toric ideals \cite{GBCP} form a rich class of binomial ideals and we saw in Section~\ref{sec:slackideals} that the 
slack ideal of a $d$-polytope with $d+2$ vertices is, in fact, toric.

\begin{problem}
When a slack ideal is binomial, is it also toric?
\end{problem}

We saw that the stable set polytope of a perfect graph is $2$-level, and thus psd-minimal. Further, this polytope admits a simple 
semidefinite representation that allows the stable set problem to be solved in polynomial time in a perfect graph.

\begin{problem} Are there other interesting examples of polytopes from 
combinatorial optimization that are psd-minimal (and not necessarily $2$-level)? Do these polytopes admit semidefinite 
representations that lead to efficient algorithms for linear optimization over them?
\end{problem}

\bibliography{psdrankbibfile}

\newcommand{\etalchar}[1]{$^{#1}$}
\begin{thebibliography}{LWdW14}

\bibitem[BDP13]{BrietDadushPokutta}
J.~Bri\"et, D.~Dadush, and S.~Pokutta.
\newblock On the existence of 0/1 polytopes with high semidefinite extension
  complexity.
\newblock In {\em Algorithms, ESA 2013}, volume 8125 of {\em Lecture Notes in
  Computer Science}, pages 217--228. Springer Berlin Heidelberg, 2013.

\bibitem[BFF{\etalchar{+}}a]{Sam&Co2levelEnumeration}
A.~Bohn, Y.~Faenza, S.~Fiorini, V.~Fisikopoulos, M.~Macchia, and K.~Pashkovich.
\newblock Enumeration of 2-level polytopes.
\newblock preprint.

\bibitem[BFF{\etalchar{+}}b]{Sam&Co2level}
A.~Bohn, Y.~Faenza, S.~Fiorini, V.~Fisikopoulos, M.~Macchia, and K.~Pashkovich.
\newblock Survey on 2-level polytopes.
\newblock in preparation.

\bibitem[BKLT13]{beasley_et_al:DR:2013:4019}
L.B. Beasley, H.~Klauck, T.~Lee, and D.O. Theis.
\newblock {Communication complexity, linear optimization, and lower bounds for
  the nonnegative rank of matrices (Dagstuhl Seminar 13082)}.
\newblock {\em Dagstuhl Reports}, 3(2):127--143, 2013.

\bibitem[Boo12]{Boocher}
A.~Boocher.
\newblock Free resolutions and sparse determinantal ideals.
\newblock {\em Math. Res. Lett.}, 19(4):805--821, 2012.

\bibitem[FGP{\etalchar{+}}15]{FGPRT}
H.~Fawzi, J.~Gouveia, P.A. Parrilo, R.Z. Robinson, and R.R. Thomas.
\newblock Positive semidefinite rank.
\newblock {\em Math. Program.}, 153(1, Ser. B):133--177, 2015.

\bibitem[FP13]{FawziParrilo}
H.~Fawzi and P.A. Parrilo.
\newblock Exponential lower bounds on fixed-size psd rank and semidefinite
  extension complexity.
\newblock {\em CoRR}, abs/1311.2571, 2013.

\bibitem[FSP]{FawziSaundersonParrilo}
H.~Fawzi, J.~Saunderson, and P.A. Parrilo.
\newblock Equivariant semidefinite lifts of regular polygons.
\newblock {\em Mathematics of Operations Research}.
\newblock to appear.

\bibitem[GGK{\etalchar{+}}13]{slackmatrixpaper}
J.~Gouveia, R.~Grappe, V.~Kaibel, K.~Pashkovich, R.Z. Robinson, and R.R.
  Thomas.
\newblock Which nonnegative matrices are slack matrices?
\newblock {\em Linear Algebra and its Applications}, 439(10):2921 -- 2933,
  2013.

\bibitem[GM82]{GiustiMerle}
M.~Giusti and M.~Merle.
\newblock Singularit\'es isol\'ees et sections planes de vari\'et\'es
  d\'eterminantielles. {II}. {S}ections de vari\'et\'es d\'eterminantielles par
  les plans de coordonn\'ees.
\newblock In {\em Algebraic geometry ({L}a {R}\'abida, 1981)}, volume 961 of
  {\em Lecture Notes in Math.}, pages 103--118. Springer, Berlin, 1982.

\bibitem[GPT10]{gouveia2010thetabodies}
J.~Gouveia, P.A. Parrilo, and R.R. Thomas.
\newblock Theta bodies for polynomial ideals.
\newblock {\em SIAM J. Optim.}, 20(4):2097--2118, 2010.

\bibitem[GPT13]{gouveia2011lifts}
J.~Gouveia, P.A. Parrilo, and R.R. Thomas.
\newblock Lifts of convex sets and cone factorizations.
\newblock {\em Mathematics of Operations Research}, 38(2):248--264, 2013.

\bibitem[GRT13]{gouveia2013polytopes}
J.~Gouveia, R.~Z. Robinson, and R.~R. Thomas.
\newblock Polytopes of minimum positive semidefinite rank.
\newblock {\em Discrete \& Computational Geometry}, 50(3):679--699, 2013.

\bibitem[Gr{\"u}03]{grunbaum}
B.~Gr{\"u}nbaum.
\newblock {\em Convex Polytopes}, volume 221 of {\em Graduate Texts in
  Mathematics}.
\newblock Springer, New York, 2003.

\bibitem[GSa]{GrandeSanyal}
F.~Grande and R.~Sanyal.
\newblock Theta rank, levelness, and matroid minors.
\newblock arXiv:1408.1262.

\bibitem[GSb]{M2}
D.R. Grayson and M.E. Stillman.
\newblock Macaulay2, a software system for research in algebraic geometry.
\newblock Available at http://www.math.uiuc.edu/Macaulay2/.

\bibitem[Lov79]{LovaszThetaBody}
L.~Lov{\'a}sz.
\newblock On the {S}hannon capacity of a graph.
\newblock {\em IEEE Trans. Inform. Theory}, 25(1):1--7, 1979.

\bibitem[LRS14]{LeeRaghavendraSteurer}
J.R. Lee, P.~Raghavendra, and D.~Steurer.
\newblock Lower bounds on the size of semidefinite programming relaxations.
\newblock {\em CoRR}, abs/1411.6317, 2014.

\bibitem[LW14]{LeeWei}
T.~Lee and Z.~Wei.
\newblock The square root rank of the correlation polytope is exponential.
\newblock {\em CoRR}, abs/1411.6712, 2014.

\bibitem[LWdW14]{LeeWeiWolf}
T.~Lee, Z.~Wei, and R.~de~Wolf.
\newblock Some upper and lower bounds on psd-rank.
\newblock {\em CoRR}, abs/1407.4308, 2014.

\bibitem[McM76]{McMullen}
P.~McMullen.
\newblock Constructions for projectively unique polytopes.
\newblock {\em Discrete Mathematics}, 14(4):347 -- 358, 1976.

\bibitem[S{\etalchar{+}}15]{sage}
W.\thinspace{}A. Stein et~al.
\newblock {\em {S}age {M}athematics {S}oftware ({V}ersion 5.11)}.
\newblock The Sage Development Team, 2015.
\newblock {\tt http://www.sagemath.org}.

\bibitem[Stu96]{GBCP}
B.~Sturmfels.
\newblock {\em Gr\"obner {B}ases and {C}onvex {P}olytopes}, volume~8 of {\em
  University Lecture Series}.
\newblock American Mathematical Society, Providence, RI, 1996.

\end{thebibliography}
\bibliographystyle{alpha}

\end{document}